\documentclass[11pt]{amsart}
\usepackage{url, 
	amssymb,setspace, mathrsfs,fontenc, comment}
\usepackage[alphabetic]{amsrefs}
\usepackage{fullpage} 
\usepackage{color}
\usepackage{tikz-cd}
\usepackage[all]{xy}
\usepackage{amsmath}

\usepackage{url, 
	amssymb,setspace, mathrsfs,fontenc}
\usepackage{amsrefs}
\usepackage{graphicx}
\usepackage[mathcal]{euscript}
\usepackage{verbatim}
\usepackage{hyperref}
\hypersetup{backref=true}
\usepackage{mathtools}
\usepackage{tikz}
\usetikzlibrary{chains}

\tikzset{node distance=2em, ch/.style={circle,draw,on chain,inner sep=2pt},chj/.style={ch,join},every path/.style={shorten >=4pt,shorten <=4pt},line width=1pt,baseline=-1ex}

\newtheorem{thm}{Theorem}
\newtheorem{lem}[thm]{Lemma}
\newtheorem{prop}[thm]{Proposition}
\newtheorem{conj}[thm]{Conjecture}
\newtheorem{cor}[thm]{Corollary}
\newtheorem{defe}[thm]{Definition}

\theoremstyle{remark}
\newtheorem{rem}[thm]{Remark}

\theoremstyle{definition}

\DefineSimpleKey{bib}{myurl}
\newcommand\myurl[1]{\url{#1}}
\BibSpec{webpage}{
	+{}{\PrintAuthors} {author}
	+{,}{ \textit} {title}
	+{}{ \parenthesize} {date}
	+{,}{ \myurl} {myurl}
}
\usepackage{arydshln}

\newcommand{\nc}{\newcommand}

\nc{\ssec}{\subsection}

\nc{\on}{\operatorname}

\nc{\sE}{\mathscr{E}}
\nc{\sF}{\mathscr{F}}
\nc{\sL}{\mathscr{L}}
\nc{\sD}{\mathscr{D}}
\nc{\sA}{\mathscr{A}}

\nc{\cC}{\mathcal{C}}
\nc{\cG}{\mathcal{G}}
\nc{\cV}{\mathcal{V}}
\nc{\CB}{\mathcal{B}}
\nc{\cK}{{k(\!(s)\!)}}
\nc {\K}{\mathcal{K}}

\nc{\cE} {\mathcal{E}}
\nc{\Kl}{\mathrm{Kl}}
\nc{\ord}{\mathrm{ord}}
\nc{\cO}{\mathcal{O}}
\nc{\cF}{\mathcal{F}}
\nc{\cZ}{\mathcal{Z}}
\nc{\bcZ}{\overline{\mathcal{Z}}}
\nc{\bcB}{\overline{\mathcal{B}}}
\nc{\cD}{\mathcal{D}}
\nc{\cDt}{\mathcal{D}^\times}
\nc{\cH}{\mathcal{H}}
\nc{\bZ}{\mathbb{Z}}
\nc{\bQ}{\mathbb{Q}}
\nc{\bR}{\mathbb{R}}
\nc{\bC}{\mathbb{C}}
\nc{\bQl}{\overline{\mathbb{Q}}_\ell}
\nc{\bQlt}{\bQl^\times} 
\nc{\FG}{\mathrm{FG}}
\nc{\dR}{\mathrm{dR}}
\nc{\rs}{\mathrm{rs}}
\nc{\der}{\mathrm{der}}
\nc{\bj}{\bar{\mathfrak{j}}}
\nc{\uG}{\underline{G}}
\nc{\uc}{\underline{c}}
\nc{\uu}{\underline{u}}
\nc{\cU}{\mathcal{U}}
\nc{\cI}{\mathcal{I}}
\nc{\rat}{\mathrm{rat}}
\nc{\Hyp}{\mathrm{Hyp}}
\nc{\Lie}{\mathrm{Lie}}
\nc{\ctheta}{\check{\theta}}
\nc{\nil}{\mathrm{nil}}
\nc{\diff}{\mathrm{diff}}

\nc{\fF}{\mathfrak{F}}
\nc{\fB}{\mathfrak{B}}
\nc{\fZ}{\mathfrak{Z}}
\nc{\fx}{\mathfrak{x}}
\nc{\fy}{\mathfrak{y}}
\nc{\fb}{\mathfrak{b}}
\nc{\fk}{\mathfrak{k}}
\nc{\fI}{\mathfrak{i}}
\nc{\fj}{\mathfrak{j}}
\nc{\fg}{\mathfrak{g}}
\nc{\fu}{\mathfrak{u}}
\nc{\fl}{\mathfrak{l}}
\nc{\fn}{\mathfrak{n}}
\nc{\cP}{\mathcal{P}}
\nc{\ft}{\mathfrak{t}}
\nc{\fz}{\mathfrak{z}}
\nc{\fc}{\mathfrak{c}}
\nc{\cfc}{\check{\mathfrak{c}}}
\nc{\fh}{\mathfrak{h}}
\nc{\fp}{\mathfrak{p}}
\nc{\bone}{\mathbf{1}}
\nc{\tg}{\mathtt{g}}
\nc{\hfg}{\widehat{\fg}}
\nc{\ch}{\check{\fh}}
\nc{\hP}{\hat{P}}
\nc{\hg}{\widehat{\mathfrak{g}}}
\nc{\gO}{\mathfrak{g}[\![t]\!]}
\nc{\Ug}{\widehat{U}(\mathfrak{g})}
\nc{\dl}{/\!\!/}

\nc{\bGm}{\mathbb{G}_m}
\nc{\bGa}{\mathbb{G}_a}
\nc{\bL}{\mathbf{L}}
\nc{\bK}{\mathbf{K}}
\nc{\bJ}{\mathbf{J}}
\nc{\bI}{\mathbf{I}}
\nc{\bV}{\mathbb{V}}
\nc{\bM}{\mathbb{M}}
\nc{\bP}{\mathbb{P}}
\nc{\bA}{\mathbb{A}}
\nc{\bN}{\mathbb{N}}

\nc {\Q}{\mathrm{Q}}
\nc{\diag}{\mathrm{diag}}
\nc{\ev}{\mathrm{ev}}
\nc{\Res}{\mathrm{Res}}
\nc{\Fl}{\mathcal{F}\ell}
\nc{\Ad}{\mathrm{Ad}}
\nc{\ad}{\mathrm{ad}}
\nc{\pr}{\mathrm{pr}}
\nc{\Sl}{\mathfrak{sl}}
\nc{\gl}{\mathfrak{gl}}
\nc{\ra}{\rightarrow}
\nc{\tra}{\twoheadrightarrow}
\nc{\hra}{\hookrightarrow}
\nc{\quo}{\mathopen{ /\!/}}
\nc{\GL}{\mathrm{GL}}
\nc{\SL}{\mathrm{SL}}
\nc{\Sp}{\mathrm{Sp}}
\nc{\SO}{\mathrm{SO}}
\nc{\so}{\mathfrak{so}}
\nc{\PGL}{\mathrm{PGL}}
\nc{\Bun}{\mathrm{Bun}}
\nc{\supp}{\mathrm{supp}}
\nc{\bgamma}{\bar{\gamma}}
\nc{\I}{\mathrm{I}}
\nc{\II}{\mathrm{II}}
\nc{\III}{\mathrm{III}}
\nc{\ab}{\mathrm{ab}}
\nc{\td}{\mathrm{d}}
\nc{\Ht}{\mathrm{ht}}
\nc{\red}{\mathrm{red}}

\nc         {\rar}[1]       {\stackrel{#1}{\longrightarrow}}

\nc{\fa}{\mathfrak{a}}
\nc{\Hit}{\mathrm{Hit}}

\nc{\RS}{\mathrm{RS}}
\nc{\Loc}{\mathrm{Loc}}
\nc{\tLoc}{\widetilde{\mathrm{Loc}}}
\nc{\tphi}{\tilde{\phi}}
\nc{\reg}{\mathrm{reg}}
\nc{\im}{\mathrm{Im}}

\nc{\tp}{\mathfrak{p}}
\nc{\cA}{\mathcal{A}}
\nc{\cY}{\mathcal{Y}}

\nc{\opp}{\mathrm{opp}}
\nc{\Ind}{\mathrm{Ind}}
\nc{\sAn}{\mathrm{can}}
\nc{\Vac}{\mathrm{Vac}}
\nc{\Op}{\mathrm{Op}}
\nc{\Lg}{\check{\fg}}
\nc{\cDelta}{\check{\Delta}}
\nc{\cPhi}{\check{\Phi}}
\nc{\LV}{\check{V}}
\nc{\Lh}{\check{h}}
\nc{\LM}{\check{M}}
\nc{\Lm}{\check{\mathfrak{m}}}
\nc{\Lz}{\check{\mathfrak{z}}}
\nc{\La}{\check{\mathfrak{a}}}
\nc{\LG}{\check{G}}
\nc{\cT}{\check{T}}
\nc{\ct}{\check{\ft}}
\nc{\cB}{\check{B}}
\nc{\cb}{\check{\fb}}
\nc{\cN}{\check{N}}
\nc{\cn}{\check{\fn}}
\nc{\Spec}{\mathrm{Spec}}
\nc{\End}{\mathrm{End}}
\nc{\crho}{\check{\rho}}
\nc{\clambda}{\check{\lambda}}
\nc{\rX}{\mathring{X}}
\nc{\ru}{\mathring{u}}

\nc{\sW}{\mathscr{W}}
\nc{\sH}{\mathscr{H}}
\nc{\sV}{\mathscr{V}}
\nc{\geom}{\mathrm{geom}}
\nc{\Irr}{\mathrm{Irr}}
\nc{\fm}{\mathfrak{m}}
\nc{\aff}{\mathrm{aff}}
\nc{\Aut}{\mathrm{Aut}}
\nc{\cJ}{\mathcal{J}}
\nc{\fs}{\mathfrak{s}}
\nc{\Stab}{\mathrm{Stab}}
\nc{\st}{\mathrm{st}}
\nc{\tw}{{\widetilde{w}}}
\nc{\gen}{\mathrm{gen}}
\nc{\genn}{\mathrm{genn}}
\nc{\sss}{\mathrm{ss}}
\nc{\fsp}{\mathfrak{sp}}
\nc{\Hom}{\mathrm{Hom}}
\nc{\bm}{\mathbf{m}}
\nc{\HG}{\mathcal{HG}}
\nc{\Gal}{\mathrm{Gal}}
\nc{\Sym}{\mathrm{Sym}}
\nc{\rank}{\mathrm{rank}}

\nc{\tP}{\mathtt{P}}
\nc{\tL}{\mathtt{L}}
\nc{\tU}{\mathtt{U}}

\nc{\tW}{\widetilde{W}}
\nc{\Hk}{\on{Hk}}
\nc{\cL}{\mathcal{L}}
\nc{\talpha}{\widetilde{\alpha}}
\nc{\tQ}{{\widetilde{Q}}}
\nc{\ochi}{\overline{\chi}}
\nc{\tdelta}{\widetilde{\Delta}}
\nc{\wt}{\mathrm{wt}}
\nc{\fQ}{\mathfrak{Q}}

\nc{\Rep}{\mathrm{Rep}}
\nc{\Conn}{\mathrm{Conn}}
\nc{\Hecke}{\mathrm{Hecke}}
\nc{\Gr}{\mathrm{Gr}}
\nc{\GR}{\mathrm{GR}}
\nc{\IC}{\mathrm{IC}}
\nc{\Std}{\mathrm{Std}} 
\nc{\Db}{\mathrm{D}^{\mathrm{b}}}
\nc{\tr}{\mathrm{tr}}
\nc{\gr}{\mathrm{gr}}
\nc{\tmin}{\mathrm{min}}
\nc{\Fun}{\mathrm{Fun}~}

\newcommand{\quash}[1]{}

\AtEndDocument{\bigskip{\footnotesize
		
		\textsc{Lingfei Yi, Shanghai Center for Mathematical Sciences, Fudan University} \par
		\textit{E-mail address}: \texttt{yilingfei@fudan.edu.cn} \par
}}

\begin{document} 
	\renewcommand{\thepart}{\Roman{part}}
	
	\renewcommand{\partname}{\hspace*{20mm} Part}
	
	\begin{abstract}
		We formulate a conjecture on local geometric Langlands for 
		supercuspidal representations in the de Rham setting
		using Yu data and the Feigin-Frenkel isomorphism.
		We refine our conjecture for a large family of 
		regular supercuspidal representations defined by Kaletha,
		and then confirm the conjecture for 
		toral supercuspidal representations of Adler
		whose Langlands parameters are precisely
		all the irreducible isoclinic connections.
		As an application, 
		we establish the conjectural correspondence between
		global Airy connections for reductive groups
		and the family of Hecke eigensheaves constructed by Jakob-Kamgarpour-Yi.
	\end{abstract}

	\title{An explicit local geometric Langlands for supercuspidal representations: the toral case}
	\author{Lingfei Yi} 
	\date{\today} 
	\maketitle
	
	\tableofcontents
	
	\section{Introduction}
	\subsection{Supercuspidal representations and Yu data}
	In the seminal work of Jiu-Kang Yu \cite{Yu},
	he proposed a construction of supercuspidal representations
	of tame $p$-adic groups,
	which turns out to be exhaustive \cite{Kim, FintzenType}
	under some assumptions on the characteristic of the residue field.
	Let $G$ be a connected reductive group over a non-archimedean local field $F$
	with residue characteristic $p\neq2$.
	Assume $G$ splits over a tamely ramified extension of $F$.
	The input of the construction is a \emph{Yu datum}
	consisting of a tuple
	\[
	((G_i)_{1\leq i\leq n},x,(r_i)_{1\leq i\leq n},\rho,(\phi_i)_{1\leq i\leq n}),
	\]
	where $G=G_0\supset G_1\supset\cdots\supset G_n$
	are twisted Levi subgroups that split over a tamely ramified extension of $F$,
	$x$ is a point in the Bruhat-Tits building of $G_n$,
	$r_1>r_2>\cdots>r_n>0$ are real numbers,
	$\rho$ is an irreducible representation of the reductive quotient
	of the parahoric subgroup of $G_n$ defined by $x$,
	and $\phi_i$ is a character of $G_i(F)$ of depth $r_i$,
	where these data need to satisfy some extra assumptions.
	From such a datum, Yu constructed a compact open subgroup
	$J\subset G(F)$ and a representation $\tphi$ of $J$
	such that the compactly induced representation
	\begin{equation}\label{eq:c-ind s.c}
		\mathrm{c-ind}^{G(F)}_J \tphi
	\end{equation}
	is irreducible, thus supercuspidal \cite{Yu,FintzenTame}.
	
	On the other hand, in a series of works 
	\cite{KalethaSimple, KalethaEpipelagic, KalethaRegular, KalethaPacket},
	Kaletha constructed L-packets for supercuspidal L-parameters,
	providing a potential explicit local Langlands correspondence 
	for supercuspidal representations.
	In the case of epipelagic representations, 
	he proved that for $G=\GL_n$,
	his bijection coincides with the local Langlands correspondence
	of Bushnell-Henniart, see \cite[\S6]{KalethaEpipelagic}.
	Also, in a recent work \cite{Oi},
	Oi showed that Kaletha's construction 
	coincides with Arthur's local Langlands correspondence
	for toral supercuspidal representations
	of quasi-split symplectic or special orthogonal groups.
	To the knowledge of the author,
	beyond the above situations,
	there is no known compatibility of Kaletha's construction 
	with any local Langlands correspondence that has been established.

	\subsection{A conjectural supercuspidal local geometric Langlands}
	The goal of this article is to propose a compatibility
	of Kaletha's construction with \emph{local geometric Langlands}
	of Frenkel-Gaitsgory \cite{FGLocal}
	in the de Rham setting,
	which we confirm for toral supercuspidal representations.
	
	\subsubsection{Local geometric Langlands}
	Let $F=\bC(\!(t)\!)$ and $\cO=\bC[\![t]\!]$,
	then $D^\times=\Spec F$ and $D=\Spec\cO$
	are the punctured formal disk and formal disk respectively.
	Let $G/\bC$ be a simply-connected simple algebraic group
	with Lie algebra $\fg$,
	and let $\LG$, $\Lg$ be its dual group and dual Lie algebra.
	In the de Rham setting, 
	a local L-parameter is a formal $\LG$-connection $\nabla$ over $D^\times$.
	It is proved by Frenkel-Zhu \cite{FZOper} that	
	any formal $\LG$-connection $\nabla$
	can be equipped with a so-called \emph{oper} structure.
	Roughly speaking, a $\LG$-oper consists of
	a $\LG$-bundle $\cF$, a flat connection $\nabla$ on $\cF$,
	and a Borel reduction $\cF_B$ of $\cF$ satisfying a strong transversality condition,
	see \cite{BD} for the precise definition.
	Denote by $\Op_{\Lg}(D^\times)$ the space of $\LG$-opers on $D^\times$
	and by $\Loc_{\LG}(D^\times)$ the isomorphism classes of 
	$\LG$-connections on $D^\times$.
	We have a surjection
	\begin{equation}\label{eq:p:Op to Loc}
		p:\ \Op_{\Lg}(D^\times)\rightarrow\Loc_{\LG}(D^\times).
	\end{equation}
	
	Let $\hg$ be the affine Kac-Moody algebra for $\fg$ at critical level.
	It is a particular central extension of the loop algebra $\fg(\!(t)\!)$
	by central element $\bone$.
	Let $\Ug$ be the completed universal enveloping algebra of $\hg$
	in which the central element $\bone$ is identified with the unit $1$. 
	Its modules $\hg-\mathrm{mod}$ 
	are smooth representations of $\hg$
	on which $\bone$ acts as $1$.
	Let $\fZ$ be the center of $\Ug$.
	The Feigin-Frenkel isomorphism states that there is an isomorphism
	\begin{equation}\label{eq:FF isom}
		\Op_{\Lg}(D^\times)\simeq\Spec\fZ
	\end{equation}
	compatible with a $\bGm$-action and certain Poisson structures on both sides.
	
	Given a $\LG$-connection $\nabla$ on $D^\times$,
	we take any oper structure $\chi\in p^{-1}(\nabla)$ above it.
	Via \eqref{eq:FF isom}, it defines a character 
	$\chi:\fZ\rightarrow\bC$.
	Let $\hg-\mathrm{mod}_\chi$ be the category
	of $\Ug$-modules
	on which $\fZ$ acts via $\chi$.
	This category is what Frenkel and Gaitsgory proposed to 
	correspond under local geometric Langlands to $\nabla$.
	More precisely,
	the loop group $LG=G(\!(t)\!)$ naturally acts on $\hg$ \cite[page 29]{FrenkelBook},
	inducing an action of $LG$ on the $\hg$-modules by 
	pre-composing the action on $\hg$.
	This action commutes with the action of center $\fZ$,
	so that $LG$ acts on category $\hg-\mathrm{mod}_\chi$.
	The induced representation of $LG$ 
	on the Grothendieck group $K_0(\hg-\mathrm{mod}_\chi)$
	is the analog of the admissible representation
	in the classical local Langlands correspondence for $p$-adic groups.
	
	It is conjectured in \cite{FGLocal} that 
	for different oper structures $\chi,\chi'$ over the same connection $\nabla$,
	the categories $\hg-\mathrm{mod}_\chi$, $\hg-\mathrm{mod}_{\chi'}$
	are equivalent.
	Only the unramified case is proved 
	\cite{FGSpherical, RaskinGL2, RYLocalization},
	and some partial results are known 
	in the tamely ramified case	\cite{FGFlag}.

	\subsubsection{An explicit correspondence for supercuspidal representations}
	Replacing the $p$-adic group by the loop group in a Yu datum,
	we obtain a $K$-type $(J,\tphi)$ 
	where $J\subset LG$ and $\tphi$ is a representation of $J$.
    Choosing the point $x$ in the closure of the fundamental alcove,
    we can let $J$ be contained in the positive loop group
    $L^+G=G[\![t]\!]$.
    Denote by $\fj$ the Lie algebra of $J$,
    and still denote by $\tphi$ the associated representation of $\fj$.
	Since $\fj\subset\fg[\![t]\!]$, 
	$\fj+\bC\bone\subset\hg$ is a subalgebra.
	Consider the induced $\hg$-representation
	\begin{equation}\label{eq:Vac_fj,tphi}
		\Vac_{\fj,\tphi}:=\mathrm{Ind}_{\fj+\bC\bone}^{\hg}\tphi.
	\end{equation}
	Denote its central support by
	\begin{equation}
		\fZ_{\fj,\tphi}:=\mathrm{Im}(\fZ\rightarrow\End\Vac_{\fj,\tphi}),
		\qquad
		\Op_{\Lg}(\fj,\tphi):=\Spec\fZ_{\fj,\tphi}\hookrightarrow\Op_{\Lg}(D^\times).
	\end{equation}
	
	\begin{conj}\label{c:main}\mbox{}
		\begin{itemize}
			\item [(i)]
			There exists a unique irreducible formal $\LG$-connection $\nabla_{\tphi}$
			such that $\Op_{\Lg}(\fj,\tphi)=p^{-1}(\nabla_{\tphi})$,
			i.e. all the oper structures on $\nabla_{\tphi}$.
			
			\item [(ii)]
			For different Yu data that induce isomorphic supercuspidal representations
			of the $p$-adic group, 
			the corresponding $K$-types $(J,\tphi)$, $(J',\tphi')$ for the loop group
			induce the same central support: 
			$\Op_{\Lg}(\fj,\tphi)=\Op_{\Lg}(\fj',\tphi')$.
			In particular, their corresponding formal $\LG$-connections in (i)
			are isomorphic:
			$\nabla_{\tphi}\simeq\nabla_{\tphi'}$.
			
			\item [(iii)]
			Every irreducible formal $\LG$-connection arises as some $\nabla_{\tphi}$.
			
			\item [(iv)]
			The above correspondence between $\nabla_{\tphi}$ and Yu data
			coincides with the supercuspidal L-packets
			constructed by Kaletha \cite{KalethaPacket}.
		\end{itemize}
	\end{conj}
	
	\begin{rem}\mbox{}
		\begin{itemize}
			\item [(i)]
			The conditions for different Yu data to give 
			isomorphic supercuspidal representations
			are given by Hakim and Murnaghan \cite{HM}.
			
			\item [(ii)]
			For epipelagic representations, a special family of supercuspidal representations of depth $\frac{1}{m}$ for $m$ a regular elliptic number,
			part (i) of the conjecture
			has been essentially proved in \cite[\S4, \S5]{CYTheta}.
		\end{itemize}
	\end{rem}
    
    We will refine the above conjecture by
    introducing the notion of \emph{refined leading terms}
    for both formal connections and $K$-types $(J,\tphi)$
    in \S\ref{s:conn refined leading terms} and \S\ref{s:sc refined leading term}.
    In \S\ref{s:regular sc conj},
    we use refined leading terms to give a refinement of Conjecture \ref{c:main}
    for a large family of \emph{regular supercuspidal representations}.
    These supercuspidal representations are defined and studied by 
    Kaletha \cite{KalethaRegular}.
    We will see that the refinement of the conjecture leads to a potential approach
    to its proof,
    which we plan to revisit in the future work.
    
    Our formulation of local geometric Langlands 
    for the above representations is explicit
    in the following sense:
    the Langlands parameter $\nabla_{\tphi}$
    is conjectured to be the gauge equivalence class
    of a `minimal' oper canonical form
    whose coefficients of the equation
    are directly determined by the character $\tphi$
    as in Remark \ref{r:explicit local Langlands}.

    \subsubsection{Potential evidence for Frenkel-Gaitsgory's conjecture}
    Conjecture \ref{c:main} provides potential evidence 
    for Frenkel-Gaitsgory's conjecture that 
    the $LG$-category $\hg-\mathrm{mod}_\chi$
    is independent of the choice of the oper structure.
    Let $\chi$ be an arbitrary oper structure on $\nabla_{\tphi}$.
    By the definition of $\nabla_{\tphi}$,
    $\Vac_{\fj,\tphi}/\ker\chi$ is a $(J,\tphi)$-equivariant object in
    the category $\hg-\mathrm{mod}_\chi$.
    Suppose $\Vac_{\fj,\tphi}/\ker\chi$ is nonzero.
    Then it provides a nonzero $(J,\tphi)$-eigenvector
    in $\hg-\mathrm{mod}_\chi$.
    Denote by $D(LG), D(J)$ the categories of $D$-modules on $LG$ and $J$ respectively.
    Then the restriction of the $LG$-action gives 
    a $D(J)$-action on $\hg-\mathrm{mod}_\chi$.
    Let $\cC$ be the subcategory of $\hg-\mathrm{mod}_\chi$
    generated by $\Vac_{\fj,\tphi}/\ker\chi$
    together with the $D(J)$-action.
    Consider the natural map
    \begin{equation}
    	\iota:D(LG)\otimes_{D(J)}\cC
    	\rightarrow D(LG)\otimes_{D(J)}\mathrm{Res}^{LG}_J\hg-\mathrm{mod}_\chi
    	\rightarrow \hg-\mathrm{mod}_\chi
    \end{equation}
    between $LG$-categories.
    If one can show that the above composition is an equivalence,
    and $D(LG)\otimes_{D(J)}\cC$ is an $LG$-category 
    that is independent of the choice of $\chi$,
    then these together would imply Frenkel-Gaitsgory's conjecture in this case.
    
    The above is inspired by the fact that if we replace the loop group $LG$ 
    by a $p$-adic group $G(F)$,
    then the compactly induced representation 
    $\mathrm{c-ind}^{G(F)}_J\tphi$ is irreducible.

	\subsection{Isoclinic formal connections and toral supercuspidal representations}
	In this article, we confirm Conjecture \ref{c:main}
	for \emph{toral supercuspidal representations},
	a family constructed by Adler \cite{Adler}
	and named by Kaletha \cite{KalethaRegular}.
	The Yu data of such supercuspidal representations
	contain only one proper twisted Levi subgroup given by an elliptic torus.
	This family of supercuspidal representations
	contains epipelagic representations
	and can attain arbitrarily large depth.
	
	It turns out that their $L$-parameters $\nabla_{\tphi}$ 
	defined via Conjecture \ref{c:main} 
	are precisely all the \emph{irreducible isoclinic connections},
	i.e. those irreducible formal $\LG$-connections
	whose canonical forms have regular semisimple leading terms,
	see \S\ref{s:isoclinic oper} for details.
    Such formal connections have been studied in \cite{JYDeligneSimpson},
	where the authors determine the conditions for 
	the existence of a $\LG$-connection on $\bP^1-\{0,\infty\}$
	with one regular singularity at $0$ with fixed monodromy class and 
	one irregular singularity at $\infty$ with fixed isoclinic irregular type.
	See \S\ref{ss:isoclinic conn} for a review.
	
	We summarize our main result as follows.
	\begin{thm}\label{t:intro main}
		Conjecture \ref{c:main}
		holds for toral supercuspidal representations.
		The L-parameters $\nabla_{\tphi}$ are all the 
		irreducible isoclinic formal connections.
	\end{thm}
    The above can be regarded as the de Rham version of
    \cite[Proposition 6.1.2]{KalethaRegular}.
    We will prove the theorem in \S\ref{s:isoclinic local Langlands},
	see Lemma \ref{l:central support of toral sc repn},
	Theorem \ref{t:main},
	and Proposition \ref{p:irr comparison}.
	An explicit connection form of $\nabla_{\tphi}$ in terms of $\tphi$
	can be given either by \eqref{eq:Op=Loc(X,nu)}
	or by the connection form $\td+A_\phi\frac{\td u}{u}$ 
	in Proposition \ref{p:irr comparison}.

	\subsection{Application to Airy connections}
	Let $h$ be the Coxeter number of $\LG$.
	Let $t$ be a coordinate at $0\in\bP^1$.
	The isoclinic formal connections with slope $\frac{1+h}{h}$ over $D^\times$
	can be globalized to
	$\LG$-connections over $\bP^1-\{0\}$
	called \emph{Airy $\LG$-connections}, cf. \cite{HJRigid,KSRigid,JKY}.
	Such connections greatly generalize 
	the classical Airy equation $y''=zy$ for $z=t^{-1}$.
	In \cite{KSRigid},
	the authors constructed a family of Airy $\LG$-connections explicitly as follows.
	
	Fix a root system of $\Lg$ with simple roots.
	Let $p_{-1}$ be the sum of a set of basis of negative simple root subspaces,
	and $p_n$ a basis of the highest root subspace.
	Then
	\begin{equation}\label{eq:KS Airy}
	\nabla=\td+(t^{-2}p_{-1}+t^{-3}p_n)\td t
	\end{equation}
	defines a family of Airy $\LG$-connections,
	see \cite[equation (2)]{KSRigid}.
	We give equations of arbitrary Airy $\LG$-connections in 
	Proposition \ref{p:Airy eq}.
	
	In Theorem \ref{t:Airy Langlands}
	we show that the eigenvalues of 
	the Hecke eigensheaves constructed in \cite{JKY}
	are Airy $\LG$-connections, 
	generalizing \cite[Corollary 36]{JKY} for the case of general linear groups.

	\subsection{Future directions}
	\subsubsection{Frenkel-Gaitsgory conjecture}
	We plan to study the properties of the $\hg$-module $\Vac_{\fj,\tphi}$
	in more detail
	and establish similar results for the type $(J,\tphi)$
	as in \cite{FGSpherical,FGFlag,RaskinGL2,RYLocalization}.
	The goal is to prove Frenkel-Gaitsgory's conjecture
	for irreducible isoclinic formal connections.

	\subsubsection{Applications in positive characteristic}
	The $K$-types $(J,\tphi)$ of toral supercuspidal representations of small depth 
	have appeared as part of various \emph{rigid automorphic data}.
	These data give rise to $\LG$-local systems on punctured projective line
	in both positive characteristic and characteristic zero setting.
	For $\ell$-adic sheaves on punctured projective line over a finite field, 
	their $p$-adic companions can be realized as $F$-isocrystals on a rigid space.
	In a recent joint work with Daxin Xu \cite{XYRigid}, 
	we study the $p$-adic local systems over Robba ring
	whose underlying connection is isoclinic.
	As an application, 
	we combine the above with results in this article and \cite{CYTheta}
	to establish some conjectural properties 
	of the $\ell$-adic $\LG$-local systems
	constructed as eigenvalues of those Hecke eigensheaves in 
	\cite{KXY, YunEpipelagic}.
	In particular, we will deduce their $p$-adic physical rigidity
	and compute their Swan conductors.

	\subsubsection{Regular supercuspidal representations}
	Regular supercuspidal representations are a generalization of 
	toral supercuspidal representations
	that arise from a much simpler form of data than Yu data
	consisting of an elliptic maximal torus $S$
	and a certain character $\phi:S(F)\rightarrow\bC^\times$.
	Kaletha associated Yu data to such pairs,
	which give rise to the regular supercuspidal representations.	
	He also constructed L-packets for such representations.
	Our next step towards an explicit local geometric Langlands
	aims to prove Conjecture \ref{c:main} for regular supercuspidal representations.
	A refined conjecture will be given in Conjecture \ref{c:local conj}.

	\subsubsection{Comparison with a wildly ramified Langlands correspondence}
	The results of this article should be closely related to 
	the work of Bezrukavnikov-Boixeda-McBreen-Yun \cite{BBMY}.
	In \S5 of \emph{loc. cit.},
	the authors proposed a conjectural wildly ramified global Langlands correspondence.
	At the unique wildly ramified place,
	the level structure they imposed is exactly $(J,\tphi)$ for a toral supercuspidal representation.
	We hope to investigate the application of our local results to their conjecture.

	\subsection{Organization of the article}
	In \S\ref{s:conn refined leading terms} and 
	\S\ref{s:sc refined leading term},
	we review the canonical forms of formal connections 
	and Yu data for regular supercuspidal representations,
	and introduce their refined leading terms respectively.
	In \S\ref{s:regular sc conj}, 
	we refine our conjecture
	for a large family of regular supercuspidal representations.
	In \S\ref{s:isoclinic and oper},
	we determine oper structures on isoclinic connections.
	In \S\ref{s:toral central support},
	we review the construction of toral supercuspidal representations
	and compute the corresponding central supports.
	In \S\ref{s:isoclinic local Langlands},
	we prove the refined conjecture for irreducible isoclinic connections.
	In \S\ref{s:Airy conn},
	we discuss applications to Airy $\LG$-connections.

	\subsection*{Notation}
	Let $G$ be a simply-connected simple algebraic group over $\bC$.
	Let $T\subset B\subset G$ be a maximal torus and Borel subgroup of $G$,
	and $\ft\subset\fb\subset\fg$ their Lie algebras.
	Let $\cT\subset\cB\subset\LG$ 
	be the dual maximal torus and a fixed Borel in the dual group,
	and $\ct\subset\cb\subset\Lg$ their Lie algebras.
	Let $n$ be the rank of $G$.
	Denote by $\Phi(G,T),\check\Phi(G,T),\Phi^+(G,T),\check\Phi^+(G,T),
	\Delta_{G},\check\Delta_{G}$ 
	the (co)roots, positive (co)roots, simple (co)roots of $G$,
	and similarly for $\LG$ and Levi subgroups.
	
    Denote by $\crho:=\check{\rho}_{\LG}$ 
	the half sum of positive coroots of $\LG$.
	Let $\Lg^{rs}$ be the locus of regular semisimple elements.
	For any subspace $\fh\subset\Lg$,
	$\fh^{rs}:=\fh\cap\Lg^{rs}$.
	
	Let $W_G\simeq W_{\LG}$ be the Weyl group acting on $\ct$.
	We say an element $w\in W$ is regular if it has an eigenvector $v\in\ct$
	that is not fixed by any element of $W$.
	The orders of regular elements are called regular numbers of $W$.
	
	Let $F=\bC(\!(t)\!)$, $\cO=\bC[\![t]\!]$.
	Let $\overline{F}$ be the algebraic closure of $F$.
	Let $D^\times=\Spec F\subset D=\Spec\cO$ 
	be punctured formal disk and formal disk.
	
	\subsection*{Acknowledgement}
	The author thanks Konstantin Jakob and Daxin Xu for helpful discussions,
	thanks Zhiwei Yun for providing comments on a draft of the paper,
	and thanks Will Sawin for pointing out a minor mistake in \S\ref{sss:Yu construction}.
	We would also like to thank an anonymous referee for pointing out a mistake and
	providing valuable suggestions.
	The research of Lingfei Yi is supported by the grant No. JIH1414033Y of Fudan University.

	\section{Refined leading terms of formal connections}\label{s:conn refined leading terms}
	
	\subsection{Canonical forms}
	
	A formal $\LG$-connection $\nabla$ on $D^\times$ 
	is a $\LG(\!(t)\!)$-gauge equivalence class 
	of an equation of the following form: 
	\[
		\td+(A_{-r}t^{-r}+A_{-r+1}t^{-r+1}+\cdots)\td t,\quad A_i\in\Lg,\ r\in\bZ,
	\]
	where the gauge transform of $g\in \LG(\!(t)\!)$ acts by
	\[
		g\cdot\nabla
		=\td+\Ad_g(A_{-r}t^{-r}+A_{-r+1}t^{-r+1}+\cdots)\td t-g^{-1}(\td g).
	\]
	
	The main theorem of \cite[\S9.5]{BV} tells us that
	a formal connection $\nabla$ can always be 
	$\LG(\overline{F})$-gauge transformed 
	into the following form:
	\begin{equation}\label{eq:canonical form}
		\td+(D_1t^{-r_1}+\cdots+D_kt^{-r_k}+D_{k+1})\frac{\td t}{t},
	\end{equation}
	where $r_1>r_2>\cdots>r_k>0$ are rational numbers, $r_{k+1}=0$,
	$D_i\in\Lg$ are mutually commuting nonzero elements(if any), 
	and $D_i$'s are all semisimple for $1\leq i\leq k$.
	This is called a \emph{canonical form} of $\nabla$.
	Any two canonical forms of $\nabla$
	must have the same $r_i$'s,
	and their tuples $(D_1,...,D_k,\exp(2\pi\sqrt{-1} D_{k+1}))$ 
	are conjugate by $\LG$.
	In particular, $\nu:=r_1$ is called the
	\emph{slope} of $\nabla$.
	By \cite[\S9.8 Lemma 1]{BV},
	we can choose the canonical form to be \emph{weakly $\bZ$-reduced},
	i.e. the semisimple part $D_{k+1,s}$ of $D_{k+1}$ satisfies that
	for any $m\in\bZ-\{0\}$,
	$D_{k+1,s}$ and $\exp(2\pi\sqrt{-1}mD_{k+1,s})$ have the same centralizer in
	the common centralizer 
	$\Lg_{D_1,...,D_k}=\cap_{i=1}^k\fz_{\Lg}(D_i)$.
	
	Conversely, given a weakly $\bZ$-reduced form $\nabla$ over $\overline{F}$
	as in \eqref{eq:canonical form},
	\cite[\S9.8 Proposition]{BV} shows that
	such a form is the canonical form of a formal connection over $F$
	if and only if there exists an integer $b\geq1$ and $\theta\in\LG$ satisfying
	$\theta^b=1$, $br_j\in\bZ$, 
	$\Ad_\theta(D_i)=\exp(-2\pi\sqrt{-1}r_i)D_i$ for all $i$.

	\subsection{Refined leading terms}
	\subsubsection{}
	The nonzero semisimple coefficient $D_1\neq0$ 
	of \eqref{eq:canonical form} is called 
	its	\emph{leading term}.
	In the following we define \emph{refined leading terms}
	that encode more information from the connection.
	
	We begin with a weakly $\bZ$-reduced canonical form $\nabla$
	as in \eqref{eq:canonical form}.
	Fix once and for all a Cartan subalgebra 
	$\ch\subset\Lg_{D_1,...,D_k}$ 
	contained in the common centralizer of the irregular part.
	Let $\check{H}$ be the corresponding maximal torus in $\LG$.
	For $1\leq i\leq k$,
	let $\LM_i=C_{\LG}(D_1,...,D_i)$ be the common centralizer.
	We also let $\LM_{k+1}:=\check{H}$.
	Since $D_i$'s are mutually commuting semisimple elements of $\ch$, 
	$\LM_i$ is the centralizer of the minimal subtorus 
	whose Lie algebra contains $D_1,...,D_i$,
	thus a connected Levi subgroup containing $\check{H}$.
	Denote its Lie algebra by $\Lm_i$, the center of $\Lm_i$ by $\Lz_i$,
	the derived subalgebra by $\Lm_i^\der=[\Lm_i,\Lm_i]$.
	We adopt the convention $\Lm_0=\Lg,\Lz_0=0$.
	The Levi decomposition gives $\Lm_i=\Lz_i\oplus\Lm_i^\der$,
	where $\Lz_i\subset\ch$.
	Note $D_1,...,D_i\in\Lz_i$,
	but $\Lz_i$ may not be spanned by $D_j$'s, 
	for instance when $D_1$ is regular semisimple.
	We have filtrations
	\[
	\Lg=\Lm_0\supset\Lm_1\supset\Lm_2\supset\cdots\supset\Lm_k\supset\Lm_{k+1}=\ch,\qquad
	0=\Lz_0\subset\Lz_1\subset\Lz_2\subset\cdots\subset\Lz_k\subset\Lz_{k+1}=\ch.
	\]
	
	Denote $\La_i:=\Lz_i\cap\Lm_{i-1}^\der$.
	Here $\La_1=\Lz_1$.
	Since $\Lz_{i-1}\subset\Lz_i$, 
	the Levi decomposition
	$\Lz_i\subset\Lm_{i-1}=\Lz_{i-1}\oplus\Lm_{i-1}^\der$
	induces decompositions
	\[
	\Lz_i=\Lz_{i-1}\oplus\La_i,\qquad
	\Lz_i=\bigoplus_{1\leq j\leq i}\La_j.
	\]
	Denote the projection of $D_i\in\Lz_i$ to $\La_i$ by $X_i$.
	Denote
	\[
	I=\{1\leq i\leq k\mid X_i\neq0\}=\{i_1<i_2<\cdots<i_p\},\quad
	R=\{r_i\mid i\in I\},\quad
	X=(X_i)_{i\in I}.
	\]
	Note that $i_1=1$.
	Equivalently, $I$ is the set of breaks of the Levi subalgebra filtration:
	\[
	\Lm_1=\Lm_{i_1}=\Lm_{i_1+1}=\cdots=\Lm_{i_2-1}\supsetneq\Lm_{i_2}=\cdots=\Lm_{i_3-1}
	\supsetneq\cdots\supsetneq\Lm_{i_p}=\cdots=\Lm_k.
	\]
	
	\begin{defe}\label{d:refined leading term}
		The tuple $X=(X_i)_{i\in I}$ is the \emph{refined leading terms}
		of the canonical form of $\nabla$
		with slopes $R$.
	\end{defe}
    
    \begin{lem}\label{l:refined leading term determines Levi seq}
    	Fix refined leading terms $X=(X_i)_{i\in I}$. 
    	For any $1\leq i\leq k$, let $i_s$ be the maximal element $i_s\in I$
    	with $i_s\leq i$.
    	Then
    	\begin{equation}\label{eq:Levi via refined leading term}
    		\LM_i=\LM_{i_s}=C_{\LG}(X_1,...,X_{i_s}).
    	\end{equation}
        In particular, the Levi sequence $\LM_i$
        is determined by the refined leading terms.
    \end{lem}
    \begin{proof}
    	First, $\LM_1$ is the centralizer of $X_1=D_1$.
    	We have $D_i\in\Lz_1$ and $\LM_i=\LM_1$
    	for $1\leq i<i_2$.
    	Write $D_{i_2}=Z_{i_2}+X_{i_2}$ for the decomposition
    	$\Lz_{i_2}=\Lz_{i_2-1}\oplus\La_{i_2}$.
    	Then an element of $\LM_{i_2-1}=\LM_1$
    	centralizes $D_{i_2}$ if and only if it centralizes $X_{i_2}$.
    	Repeating the process determines all the Levi subgroups $\LM_i$
    	as in \eqref{eq:Levi via refined leading term}.
    \end{proof}
    
    \begin{defe}\label{d:Levi generic}
    	A tuple $X=(X_i\in\La_i)_{i\in I}$ is \emph{generic}
    	for the sequence $\LM_i,1\leq i\leq k$ if 
    	\eqref{eq:Levi via refined leading term}
    	is satisfied.
    \end{defe}
    
    Since the irregular part of any canonical form of a formal connection
    is unique up to $\LG(\bC)$-conjugation,
    the set $R$, the sequence of Levi subgroups $\LM_i$,
    and the refined leading terms $(X_i)_i$
    are all uniquely determined by the isomorphism class of the formal connection
    up to conjugation.
    
    The following property follows immediately from the definition.
    \begin{lem}
    	If the refined leading terms of $\nabla$ 
    	are generic for $\LM_i$ with slopes $r_i$, $i\in I$,
    	then the slopes of the adjoint connection $\nabla^\Ad$ 
    	consist of
    	$\dim\Lm_{k+1}=\mathrm{rk}\Lg$ copies of $0$ and
    	$|\Phi(\LM_{i_j})-\Phi(\LM_{i_{j+1}})|$ copies of $r_{i_j}$, $i_j\in I$.
    \end{lem}

	\subsubsection{}
	Recall there exists $\theta\in\LG$ satisfying $\theta^b=1$, $br_i\in\bZ$,
	$\theta(D_i)=\exp(-2\pi\sqrt{-1}r_i)D_i$.
	The torsion inner automorphism $\Ad_\theta$ defines a grading
	\begin{equation}
		\Lg=\bigoplus_{i\in\bZ/b\bZ}\Lg_i.
	\end{equation} 
    
    Since $D_i$'s are $\theta$-eigenvectors, 
    $\theta$ acts on $\LM_i$, $\Lm_i$, $\Lz_i$, $\Lm_i^\der$, $\La_i$.
    For each of these Lie algebras $\fc$, denote $\fc_j=\fc\cap\Lg_j$. 
    Let $m$ be the minimal common denominator of all the $r_i$,
    so that $m|b$.
    Define
    \begin{equation}\label{eq:Loc(R,theta)}
    	\Loc(R,\theta)=
    	\td+(\sum_{i\in I}\sum_{-mr_i\leq j\leq -1}\La_{i,j}t^{\frac{j}{m}}+\Lm_k)
    	\frac{\td t}{t},
    \end{equation}
    and
	\begin{equation}\label{eq:Loc(X,R,theta)}
		\Loc(X,R,\theta)=
		\td+\sum_{i\in I}(X_it^{-r_i}
		+\sum_{-mr_i+1\leq j\leq-1}\La_{i,j}t^{\frac{j}{m}}
		+\Lm_k)\frac{\td t}{t}.
	\end{equation}
    
    \begin{lem}\label{l:distinct canonical form}
    	Two equations with the same refined leading terms
    	$\nabla_1,\nabla_2\in\Loc(X,R,\theta)$ 
    	give isomorphic formal connections
    	only if these equations \eqref{eq:Loc(X,R,theta)} 
    	have identical irregular parts.
    \end{lem}
    \begin{proof}
    	Write the canonical forms of $\nabla_1,\nabla_2$
    	as in \eqref{eq:canonical form}:
    	\[
    	\nabla_j=\td+(D_{1,j}t^{-r_1}+\cdots+D_{k,j}t^{-r_k}+D_{k+1,j})\frac{\td t}{t},
    	\quad j=1,2,
    	\]
    	Suppose $\nabla_1,\nabla_2$ are isomorphic.
    	Then there exists $g\in\LG$ such that $Ad_g(D_{i,1})=D_{i,2}$ for $r_i\neq0$.
    	Since $D_1=X_1$, $g\in C_{\LG}(X_1)=\LM_1$.
    	Thus $\Ad_g D_{i,1}=D_{i,1}=D_{i,2}$ for $1\leq i<i_2$.
    	Write $D_{i_2,j}=Z_{i_2,j}+X_{i_2}$, $j=1,2$.
    	Note $g\in \LM_1=\LM_{i_2-1}$ centralizes $Z_{i_2,j}\in\Lz_{i_2-1}$.
    	Thus $Ad_g(D_{i_2,1})=D_{i_2,2}$ implies 
    	$Z_{i_2,1}+\Ad_g X_{i_2}=Z_{i_2,2}+X_{i_2}$.
    	Note that $\LM_{i_2-1}$ acts on $\Lm_{i_2-1}^\der$,
    	so that $\Ad_g X_{i_2}-X_{i_2}=Z_{i_2,2}-Z_{i_2,1}\in\Lm_{i_2-1}^\der\cap\Lz_{i_2-1}=0$.
    	We obtain $D_{i_2,1}=D_{i_2,2}$ and $g\in \LM_{i_2}$.
    	Repeat the procedure, we obtain $g\in \LM_k$
    	and $D_{i,1}=D_{i,2}$ for all $i\leq k$.
    \end{proof}

    \subsubsection{}
    Alternatively, we can characterize Levi subgroups $\LM_i$ as follows.
    Let $I_{\diff}$ be the differential Galois group of $D^\times$,
    so that each formal connection $\nabla$
    corresponds to its monodromy representation $\rho_\nabla$ of $I_{\diff}$,
    each formal $\LG$-connection corresponds to a group homomorphism
    $\rho_\nabla:I_{\diff}\rightarrow\LG$.
    Let $I_{\diff}^r$ (resp. $I_{\diff}^{r+}$) be 
    the intersection of the kernels of those representations of $I_{\diff}$
    corresponding to formal connections with slopes $<r$ (resp. $\leq r$).
    Here $I_{\diff}^0=I_{\diff}$.
    For $r>0$ (resp. $r\geq0$), 
    $\rho_\nabla(I_{\diff}^r)$ (resp. $\rho_\nabla(I_{\diff}^{r+})$) is a connected torus,
    thus can be conjugated into the maximal torus $\check{H}$.
    \begin{lem}
    	Let $\rho_\nabla:I_{\diff}\rightarrow\LG$
    	be the monodromy representation 
    	of a formal connection $\nabla$.
    	For $r>0$, denote by $\Phi_r$ the root subsystem
    	\[
    	\Phi_r:=\{\alpha\in\Phi(\LG,\check{H})\mid\alpha(\rho_\nabla(I_{\diff}^{r}))=1\}.
    	\]
    	Then $R$ is the set of $r$ where $\Phi_r$ jumps,
    	and $\Phi_{r_i}$, $r_i\in R$ is the root system of $\LM_i$.
    \end{lem}
    \begin{proof}
    	This is clear from the observation that
    	the restricted representation 
    	$\rho_\nabla:I_{\diff}^r\rightarrow\LG$
    	is determined only by the coefficients $D_i$, $r_i\geq r$
    	in the canonical form \eqref{eq:canonical form}.
    \end{proof}
    
    We will see that the components 
    $(X_{i_s}t^{-r_i}+\sum_{\ell=1}^{s-1}\sum_{j=-mr_{i_{s-1}}+1}^{-mr_{i_s}}Y_{i_\ell,j}t^{\frac{j}{m}})\frac{\td t}{t}\in\Lz_i(\!(t)\!)\frac{\td t}{t}$, 
    $Y_{i_\ell,j}\in\fa_{i_\ell,j}$ $i_s\in I$
    should be dual to the \emph{Howe factorization} 
    \cite[Definition 3.6.2]{KalethaRegular},
    see discussion in 
    \S\ref{ss:reg supercus},\S\ref{ss:dual refined leading terms}.

    \quash{\subsubsection{}\label{sss:canonical form}
    Lastly we construct a family of canonical forms from scratch for future use.
    Fix a maximal torus $\check{H}$ with Lie algebra $\ch$
    and slopes $r_1>r_2>\cdots>r_k>r_{k+1}=0$.
    Let $m$ be the minimal common denominator of $r_i$'s, $mr_i\in\bZ$.
    Assume there is a conjugacy class of order $m$ in the Weyl group,
    represented by an element $w_m\in W(\check{H})$,
    which acts on $\ch$.
    Denote the grading defined by the action of $w_m$ on $\ch$ by
    \begin{equation}
    	\ch=\bigoplus_{i\in\bZ/m}\ch_i.
    \end{equation}

    Consider an equation
    \begin{equation}\label{eq:eq from grading}
    	\nabla=\td+(X_{r_1}t^{-r_1}+\cdots+X_{r_k}t^{-r_k}+X_{r_{k+1}})\frac{\td t}{t}
    \end{equation}
    satisfying $X_{r_i}\in\ch_{-mr_i}$.
    For any lift $\dot{w}$ of $w$,
    $\dot{w}\cdot X_{r_i}=\exp(-2\pi i r_i)X_{r_i}$.
    From the proof of \cite[Proposition 9.8]{BV},
    there exists element $a\in\LG$ centralizer $X_{r_i}$'s
    and integer $b$ divisible by $m$
    such that $(a\dot{w})^b=1$.
    By the same proposition in \emph{loc. cit.},
    $\nabla$ can be gauge-transformed to a formal connection over $F$.
    Thus we can define
    \begin{equation}\label{eq:Loc from grading}
    	\Loc(\ch,r_i)=\{\td+(X_{r_1}t^{-r_1}+\cdots+X_{r_k}t^{-r_k}+X_{r_{k+1}})\frac{\td t}{t}\mid X_{r_i}\in\ch_{-mr_i}\}/\sim\subset\Loc_{\LG}(D^\times),
    \end{equation}
    where $\sim$ means the $\LG(\overline{F})$-gauge equivalence.}

    \subsection{Oper canonical forms}\label{ss:oper can form}
    A formal connection can always be represented in a non-unique way
    by a connection in \emph{oper canonical forms},
    which we review.
    
    A $\LG$-oper over a space is 
    a $\LG$-bundle with a connection and a Borel reduction
    satisfying a strong transversality condition, cf. \cite{BD,FGLocal}.
    Denote by $\Op_{\Lg}(D^\times)$ the scheme of opers over $D^\times$.
    Let $E_{-\alpha}\in\Lg_{-\alpha}$
    be a basis of the negative simple root subspace, 
    $\alpha\in\Delta_{\LG}$ a simple root.
    Then $p_{-1}=\sum_{\alpha\in\Delta_{\LG}}E_{-\alpha}$
    is a regular nilpotent element.
    Let $\{p_{-1},2\crho_{\LG},p_1\}$ be the unique principal $\mathfrak{sl}_2$-triple,
    and $p_i$ be a homogeneous basis of $\Lg^{p_1}$
    with weight $d_i-1$ under the action of $\crho$,
    $1\leq i\leq n$.
    Any $\chi\in\Op_{\Lg}(D^\times)$
    can be uniquely represented by 
    an \emph{oper canonical form} as follows:
    \begin{equation}\label{eq:oper form}
    	\nabla_\chi=\td+(p_{-1}+\sum_{i=1}^n v_i(t)p_i)\td t,\quad v_i(t)=\sum_j v_{ij}t^{-j-1}\in\bC(\!(t)\!),\ v_{ij}\in\bC.
    \end{equation}
    In particular, we have isomorphism
    \[
    \Op_{\Lg}(D^\times)\simeq\bC(\!(t)\!)^n,\qquad
    \nabla\mapsto(v_1(t),...,v_n(t)).
    \]
    
    For an oper $\chi$ with canonical form \eqref{eq:oper form}, 
    its \emph{slope} is defined to be
    \begin{equation}\label{eq:oper slope}
    	s(\chi)=\sup\{0,\sup_{1\leq i\leq n}\{\frac{-\ord_t v_i(t)}{d_i}-1\}\}
    \end{equation} 
    where we let $\ord_t(0)=+\infty$.
    
    We have a forgetful functor from opers to 
    the underlying local system forgetting the Borel reduction:
    \begin{equation}
    	p:\Op_{\Lg}(D^\times)\rightarrow\Loc_{\LG}(D^\times).
    \end{equation}
    Equivalently, $p$ sends an oper canonical form $\nabla_\chi$
    to its $G(\!(t)\!)$-gauge equivalence class.
    It is proved in \cite{FZOper} that $p$ is surjective,
    but it has infinite-dimensional fibers.
    
    By \cite[Proposition 1]{CK},
    the slope of an oper coincides with the slope of its underlying connection,
    so that $s(\chi)$ is constant on the fibers of $p$.

    \subsection{Relationship between connection canonical forms and oper canonical forms}
    \subsubsection{}\label{sss:algorithm from oper form to canonical form}
    We describe another way of deducing the sequence $r_{i_1},...,r_{i_k}$
    and Levi subgroups $\LM_{i_1},...,\LM_{i_k}$
    from a formal connection starting from its oper canonical form.
    Denote $r_{(j)}=r_{i_j}$, $\LM_{(j)}=\LM_{i_j}$, 
    and similarly for $\Lm_{(j)},\Lz_{(j)},\Lm_{(j)}^\der$.
    
    Let $\nabla$ be an irregular formal connection.
    Let $\chi=\td+p_{-1}+\sum_{i,j}v_{i,j}t^{-j-1}p_i\td t$
    be an oper canonical form of $\nabla$.
    Assume $\nabla$ has slope $r_{(1)}=r_1$.
    
    By the oper slope formula \eqref{eq:oper slope},    
    $\chi$ is of the form
    \begin{equation}
    	\td+(p_{-1}+\sum_{i=1}^n\sum_{j+1\leq d_i(r_1+1)} v_{ij}t^{-j-1}p_i)\td t.
    \end{equation}
    Applying the gauge transform by $\crho(t^{r_1+1})$ to the above equation,
    we obtain
    \begin{equation}\label{eq:0}
    	\td+(t^{-r_1}p_{-1}+\sum_{i=1}^n\sum_{-r_1\leq -j+(d_i-1)(r_1+1)} v_{ij}t^{-j+(r_1+1)(d_i-1)}p_i-(r_1+1)\crho)\frac{\td t}{t}.
    \end{equation}
    Adopting the algorithm in \cite[Proposition 4.6]{BV}, 
    we can gauge transform the above into a canonical form by induction.
    We sketch the algorithm in this situation in the following.
     
    First,
    by the slope formula, there exist $1\leq i\leq n$ and $j$
    such that $j+1=d_i(r_1+1)$ and $v_{i,j}\neq0$.
    Let
    \begin{equation}
    	D_{r_1}=p_{-1}+\sum_{j+1=d_i(r_1+1)}v_{i,j}p_i.
    \end{equation}
    Let $D_{r_1}=D_{r_1,s}+D_{r_1,n}$ be Jordan decomposition
    and denote $X_{r_1}=D_{r_1,s}$, $Y_{r_1}=D_{r_1,n}$.
    Let $\Lm_1=\fz_{\Lg}(X_{r_1})$.
    Observe that $D_{r_1}$ is in the Kostant section, so that it is regular.
    Equivalently, $Y_{r_1}$ is regular nilpotent in $\Lm_1$.
    Let $\cO_{\overline{F}}$ be the valuation ring 
    of the algebraic closure $\overline{F}$,
    and $\cO_{\overline{F}}^+$ its maximal ideal so that
    $\cO_{\overline{F}}/\cO_{\overline{F}}^+\simeq\bC$.
    Let $\LG(\cO_{\overline{F}})_{0+}$ be the kernel of 
    the reduction map $\LG(\cO_{\overline{F}})\rightarrow\LG(\bC)$.
    Picking an $\mathfrak{sl}_2$-triple containing $X_{r_1}$
    and applying gauge transforms by elements in $\LG(\cO_{\overline{F}})_{0+}$,
    we can transfer \eqref{eq:0} into the form
    \begin{equation}
    	\nabla_1\in\td+((X_{r_1}+Y_{r_1})t^{-r_1}+t^{-r_1}\Lm_1(\cO_{\overline{F}})_{0+})\frac{\td t}{t}\subset\td+\Lm_1(\overline{F})\td t.
    \end{equation}
    
    We decompose $\nabla_1$ according to 
    the Levi decomposition $\Lm_1=\Lz_1+\Lm_1^\der$
    and write $\nabla_1=\nabla_{\Lz_1}+\nabla_{\Lm_1^\der}$.
    Recall $\Lz_1=\La_1$.
    Let $\check{Z}_1\subset\check{H}$ be the subtorus with Lie algebra $\Lz_1$.
    We can transfer $\nabla_1$ into canonical form using $\check{Z}_1(\cO_{\overline{F}})_{0+}$,
    so we assume it already is.
    Here $X_{r_1}$ is the leading term of $\nabla_{\Lz_1}$.
    
    The $\LM_1^\der$-connection $\nabla_{\Lm_1^\der}$
    has regular nilpotent leading term $Y_{r_1}$.
    Take an $\mathfrak{sl}_2$-triple $\{e_1,f_1,h_1\}$ inside $\Lm_1^\der$
    where $f_1=Y_{r_1}$.
    Using gauge transform by $\LM_1^\der(\cO_{\overline{F}})_{0+}$,
    we can change $\nabla_{\Lm_1^\der}$ into the form
    \[
    \nabla'=\td+(Y_{r_1}t^{-r_1}+t^{-r_1}(\Lm_1^\der)^{e_1}(\cO_{\overline{F}})_{0+})\frac{\td t}{t}.
    \]
    
    Let $(\Lm_1^\der)^{e_1}=\bigoplus_{1\leq k\leq q}\bC Z_k$ be a homogeneous basis,
    where $[h_1,Z_k]=\lambda_k Z_k$ for integer $\lambda_k\geq0$.
    This is exactly an oper canonical form for $\Lm_1^\der$.
    Denote $\Lambda(Y_{r_1})=\sup_k(\frac{1}{2}\lambda_k+1)$.
    Write 
    \[
    \nabla'=\td+(Y_{r_1}t^{-r_1}+\sum_{r>-r_1}\sum_{1\leq k\leq q}a_{r,k}t^rZ_k)\frac{\td t}{t},\qquad a_{r,k}\in\bC.
    \]
    Here $r$ ranges over rational numbers such that $a_{r,k}\neq0$.
    Define
    \begin{equation}
    	\delta(\nabla')=\inf\{\frac{r+r_1}{\frac{1}{2}\lambda_k+1}\mid
    	0<r+r_1<\Lambda(Y_{r_1})|r|,\ a_{r,k}\neq0\}.
    \end{equation}
    If there is no such nonzero coefficient $a_{r,k}$, 
    let $\delta(\nabla')=\infty$.
    
    If $\delta(\nabla')\neq\infty$,
    we apply gauge transform by 
    $t^{-\frac{\delta(\nabla')}{2}h_1}$ to $\nabla'$.
    Otherwise apply gauge transform by 
    $t^{-\frac{r_1}{2}h_1}$ to $\nabla'$.
    The resulting connection $\nabla''$ is either regular singular,
    or is of order $r_{(2)}=r_1-\delta(\nabla')$
    whose leading term is of the form
    \begin{equation}\label{eq:1}
    	D_{r_{(2)}}=Y_{r_1}+\sum_{k,r}a_{r,k}Z_k.
    \end{equation}

    If $\nabla''$ becomes regular singular, we stop.
    Otherwise,
    since $Y_{r_1}$ is regular nilpotent in $\Lm_1^\der$,
    $\{e_1,f_1=Y_{r_1},h_1\}$ is a principal $\mathfrak{sl}_2$-triple of $\Lm_1^\der$,
    so that \eqref{eq:1} is contained in a Kostant section of $\Lm_1^\der$.
    If $\nabla''$ is not regular singular, 
    then $\delta(\nabla')\neq\infty$,
    there exists $a_{r,k}\neq0$ in \eqref{eq:1},
    so that \eqref{eq:1} is a non-nilpotent regular element.
    We can now just repeat the previous procedure:
    take Jordan decomposition $D_{r_{(2)}}=X_{r_{(2)}}+Y_{r_{(2)}}$,
    let $\Lm_{(2)}=\Lz_{\Lm_1}(X_{r_{(2)}})$ with center $\Lz_{(2)}$.
    Note that the center of $\Lz_{\Lm_1^\der}(X_{r_{(2)}})$
    equals $\La_{(2)}=\Lz_{(2)}\cap\Lm_1^\der$.
    Denote the canonical form of the 
    direct summand of $\nabla''$ inside $\La_{(2)}$ 
    by $\nabla_{\La_{(2)}}$.
    
    This procedure produces a sequence of slopes $r_{(1)}>r_{(2)}>\cdots>r_{(p)}$
    and Levi subalgebras $\Lm_1=\Lm_{(1)}\supset\Lm_{(2)}\supset\cdots\supset\Lm_{(p)}$.
    We conclude that the irregular part of the canonical form of $\nabla$ is
    \begin{equation}
    	\bigoplus_{s=1}^p\nabla_{\La_{(s)}}.
    \end{equation}

    \subsubsection{}
    Denote 
    \begin{equation}\label{eq:Op(X,R,theta)}
    	\Op_{\Lg}(X,R,\theta)=p^{-1}(\Loc(X,R,\theta)).
    \end{equation}
    Then the canonical form of any such oper induces 
    slopes $r_{(1)},...,r_{(p)}$ and Levi subgroups $\LM_{(1)},...,\LM_{(p)}$
    as in \S\ref{sss:algorithm from oper form to canonical form}.  
    
    \begin{prop}\label{p:supp of oper form}
    	\begin{itemize}
    		\item [(i)]
    		The oper canonical form \eqref{eq:oper form} of an oper
    		$\chi\in\Op_{\Lg}(X,R,\theta)$ 
    		satisfies $v_{i,j}=0$ for $-j+(d_i-1)(r_1+1)<-r_{1}$.
    		
    		\item [(ii)]
    		Denote the Coxeter number of $\LM_{(s)}$ by $h_s$,
    		which is set to be $1$ if $\LM_{(s)}$ is a torus.
    		The direct summand $\nabla_{\La_{(s)}}$ of 
    		the canonical form of $\chi$ in (i)
    		depends only on $v_{i,j}$ for 
    		$-r_1\leq -j+(d_i-1)(r_1+1)\leq\sum_{a=1}^{s-1}(h_a-1)(r_{(a)}-r_{(a+1)})$.
    		In particular,
    		the irregular part of the canonical form of $\chi$
    		depends only on $v_{i,j}$ for 
    	    $-r_1\leq -j+(d_i-1)(r_1+1)\leq\sum_{a=1}^{p-1}(h_a-1)(r_{(a)}-r_{(a+1)})$.
    	    Moreover,
    	    the regular part of the canonical form, therefore the whole canonical form,
    	    depends only on $v_{i,j}$ for
    	    $-r_1\leq -j+(d_i-1)(r_1+1)\leq\sum_{a=1}^{p-1}(h_a-1)(r_{(a)}-r_{(a+1)})
    	    +(h_p-1)r_{(p)}$.
    	\end{itemize}
    \end{prop}
    \begin{proof}
    	Part (i) follows from the oper slope formula \eqref{eq:oper slope}.
    	
    	For part (ii),
    	in the algorithm \S\ref{sss:algorithm from oper form to canonical form},
    	beginning from equation \eqref{eq:0},
    	coefficients $v_{i,j}$ with $-r_1\leq-j+(d_i-1)(r_1+1)$ get involved
    	in the canonical form only when we apply gauge transform by
    	$t^{-\frac{\delta(\nabla')}{2}h}$.
    	Precisely, 
    	coefficients of $t^{q}\td t$ with $q\geq0$ become coefficients of
    	$t^{q-(d_{i,(a)}-1)(r_{(a)}-r_{(a+1)})}\td t$,
    	where $d_{i,(a)}$ is a fundamental degree of $\LM_{(a)}$
    	which is at most $h_a$.
    	The irregular part of part (ii) follows immediately.
    	The dependence of the regular part of canonical form 
    	follows from \cite[Proposition 4.6.(a)]{BV}.
    \end{proof}
    
    \begin{rem}
    		Proposition \ref{p:supp of oper form}.(ii)
    		refines the bound given by \cite[Theorem 9.7]{BV}
    		applied to \eqref{eq:0}.
    \end{rem}

    \section{Refined leading terms of tame elliptic pairs}\label{s:sc refined leading term}
    We will refine Conjecture \ref{c:main}
    for a large family of regular supercuspidal representations
    that includes the toral supercuspidal representations.
    In this section we review their construction
    and define their refined leading terms.
    
    \subsection{Recollection on certain regular supercuspidal representations}\label{ss:reg supercus}
    We first review the construction of regular supercuspidal representations
    following \cite{KalethaRegular}.
    
    \subsubsection{}
    Let $S\subset G_F$ be an elliptic maximal torus over $F=\bC(\!(t)\!)$
    that splits over $E=\bC(\!(u)\!)$, $u=t^{\frac{1}{m}}$.
    Explicitly,
    fix a maximal torus $H\subset G$, 
    let $w_m\in W(G_E,H_E)$ be an elliptic element of order $m$ in the Weyl group,
    $\sigma:u\mapsto\zeta_mu$ a generator of $\mathrm{Gal}(E/F)$. 
    We assume $S\simeq(\mathrm{Res}_{E/F}(H_E))^{w_m\times\sigma^{-1}}$,
    so that $S_E\simeq H_E$.
    Let $\phi:S(F)\rightarrow\bC^\times$ be a character of depth $\nu$
    satisfying conditions in \cite[Definition 3.7.5]{KalethaRegular}.
    Such $(S,\phi)$ is called a \emph{tame elliptic pair}.
    We further assume that
    \begin{equation}\label{eq:Phi_0+=empty}
    	\{\alpha\in\Phi(G,S)\mid\phi(N_{E/F}(\check{\alpha}(E^\times_{0+})))=1\}=\varnothing.
    \end{equation}
    where $E^\times_{0+}=1+u\bC[\![u]\!]$.
    This condition is to avoid the appearance of 
    the depth zero supercuspidal representation in Yu's datum.
    It corresponds to those L-parameters such that the
    centralizers of the images of the wild inertia group are tori.
    
    We associate a Yu datum to $(S,\phi)$ as in 
    \cite[Proposition 3.7.8]{KalethaRegular}.
    We slightly modify the notation to simplify it in our setting
    and to be more consistent with the notations
    on the spectral side.
    First, $S$ has a unique point $x$ 
    in its Bruhat-Tits building,
    giving a point in the building of $G$.
    We can then define a sequence of root subsystems
    \[
    \Phi_r=\{\alpha\in\Phi(G,S)\mid\phi(N_{E/F}(\check{\alpha}(E^\times_r)))=1\},
    \]
    where $E_r^\times=1+u^{\lceil mr\rceil}\bC[\![u]\!]$.
    By Lemma 3.6.1 of \emph{loc. cit.},
    $\Phi_r$ is a Levi subsystem.
    Let $R=\{r_{(1)},r_{(2)},...,r_{(p)}\}$ be real numbers 
    $r_{(1)}>r_{(2)}>\cdots>r_{(p)}>0$
    at which $\Phi_r$ jumps.
    Note $r_{(1)}=\nu$.
    Let $G_i'\subset G(E)$ be the standard Levi subgroup containing $S_E$
    with root system $\Phi_{r_{(i)}}$.
    Then the groups $G_i=G(F)\cap G_i'$ are twisted Levi subgroups of $G(F)$:
    \[
    G=G_0\supset G_1\supset G_2\supset\cdots\supset G_p=S.
    \footnote{Our convention on indexing is different from \cite{KalethaRegular,FintzenTame} in order to be more consistent with the indexing on the dual side.}
    \]
    
    By \cite[Proposition 3.6.7]{KalethaRegular}, 
    character $\phi$ has a \emph{Howe factorization}
    as in Definition 3.6.2 of \emph{loc. cit.}
    It consists of a sequence of characters 
    $\phi_i:G_i\rightarrow\bC^\times$, $1\leq i\leq p$
    satisfying
    \begin{equation}
    	\phi=\prod_{i=1}^p\phi_i|_{S(F)}.
    \end{equation}
    Here two of $\phi_i$'s in \emph{loc. cit.} are trivial
    because of our assumptions \eqref{eq:Phi_0+=empty} and $\nu=r_{(1)}$. 
    For $1\leq i\leq p$,
    $\phi_i$ is assumed to be trivial on $G_{i+1,sc}(F)$ where 
    $G_{i+1,sc}$ is the simply-connected cover of the derived group of $G_{i+1}$,
    and $\phi_i$ has depth $r_{(i)}$.
    Moreover, $\phi_i$ needs to be $G_{i-1}$-generic in the sense of \cite[\S9]{Yu},
    which we recall below.
    
    Let $Z_i=Z(G_i)$, with Lie algebra $\fz_i\subset\fg_i$.
    The dual $\fz_i^*$ can be naturally identified via restriction with
    the subspace of $\fg_i^*$ fixed under coadjoint action of $G_i$. 
    Also, $\fg_i^*$ can be identified via restriction with
    the subspace of $\fg_{i-1}^*$ fixed by $Z_i$.
    Thus $\fz_i^*\subset\fg_i^*\subset\fg_{i-1}^*$.
    
    Denote by $\fz_{i,-r}^*$, $\fg_{i,-r}^*$, $\fg_{i-1,-r}^*$ 
    the Moy-Prasad filtration on the dual spaces.
    For example, $\fg_{i,-r}^*$ is the subspace of $X\in\fg_i^*$
    satisfying $\ord_t(X(\fg_{i,r+}))>0$.
    We say $\phi_i:G_i\rightarrow\bC^\times$ is 
    \emph{$G_{i-1}$-generic of depth $r_{(i)}$}
    if $\phi_i$ is trivial on $G_{i,r_{(i)}+}$, non-trivial on $G_{i,r_{(i)}}$,
    and its differential lands inside $\td\phi_i\in\fz_{i,-r_{(i)}}^*$ 
    such that 
    \begin{equation}\label{eq:G-generic}
    	\mathrm{ord}(\td\phi_i(H_\alpha))=-r_{(i)},\qquad
    	\forall\ \alpha\in\Phi(G_{i-1},S)\backslash\Phi(G_i,S).
    \end{equation}
    Here for a root $\alpha$ with coroot $\check{\alpha}:\bGm\rightarrow S$,
    $H_\alpha=d\check{\alpha}(1)\in\fg_{i-1}$.
    This is condition GE1 in \cite[\S8]{Yu},
    and by Lemma 8.1 of \emph{loc. cit.}
    GE1 implies GE2 of \cite[\S8]{Yu} in our case.

    \subsubsection{}\label{sss:Yu construction}
    Denote
    \begin{equation}
    	(G_i)_{x,r,r'}=G(F)\cap\langle T(E)_r,U_\alpha(E)_{x,r},U_\beta(E)_{x,r'}\mid
    	\alpha\in\Phi(G_i,S),\beta\in\Phi(G_i,S)\backslash\Phi(G_{i+1},S)\rangle.
    \end{equation}
    Let $(G_i)_{x,r,r'+}=\cup_{s>r'}(G_i)_{x,r,s}$.
    The quotient $V_i=(G_i)_{x,r_{(i+1)},r_{(i+1)}/2}/(G_i)_{x,r_{(i+1)},r_{(i+1)}/2+}$ is commutative.
    Extend $\phi_{i+1}$ to $G_i$ via projection with respect to root subspaces.
    The pairing $\langle a,b\rangle=\phi_{i+1}(aba^{-1}b^{-1})$
    defines a symplectic form on $V_i$.
    Choose a Lagrangian subspace $\fm_i\subset V_i$.
    Denote by $(G_i)_{x,r_{(i+1)},r_{(i+1)}/2,\fm_i}$ 
    the preimage of $\fm_i$ in $(G_i)_{x,r_{(i+1)},r_{(i+1)}/2}$
    \footnote{The usual Yu's construction uses
    Weil-Heisenberg representation
    while we instead take a Lagrangian.
    We make this change so that the eventual representation 
    is induced from a character. 
    The same change is made in \cite{BBMY,JKY}}.
    Denote
    \begin{equation}
    	J=S_{0+}\prod_{i=0}^{p-1}(G_i)_{x,r_{(i+1)},r_{(i+1)}/2,\fm_i},
    \end{equation}
    where we set $r_{(p+1)}=0$.
    Note that by our assumption,
    $(G_p)_x=S_0\subset J$.
    The quotient $S_0/S_{0+}$
    is the fixed-point subgroup of a maximal torus under 
    a regular elliptic element of the Weyl group,
    which is a nontrivial finite group in general.
    Here we use $S_{0+}$ instead of $(G_p)_x$ 
    so that $J$ is connected and pro-unipotent.
    
    Consider the character $\tphi=\prod_{i=1}^p\phi_i$ on $J$.
    Here the characters $\phi_i$ are extended to $J$ as follows.
    
    On $(G_j)_{x,r_{(j+1)},r_{(j+1)}/2,\fm_j}\subset(G_j)_{x,r_{(j+1)},r_{(j+1)}/2}$
    where $i\leq j$, $G_i\supset G_j$, 
    $\phi_i$ is restricted from $G_i$.
    
    For $i>j$ where $G_i\subset G_j$,
    the character $\phi_i$ is first extended to $G_{x,r_{(i+1)}/2+}$ via projection
    \begin{align*}
    G_{x,r_{(i+1)}/2+}/G_{x,r_{(i+1)}+}
    \simeq&\fg_{x,r_{(i+1)}/2+}/\fg_{x,r_{(i+1)}+}\\
    \rightarrow&
    (\fg_i)_{x,r_{(i+1)}/2+}/(\fg_i)_{x,r_{(i+1)}+}\simeq (G_i)_{x,r_{(i+1)}/2+}/(G_i)_{x,r_{(i+1)}+},
    \end{align*}
    where the projection in the middle is via root subspace decomposition
    \[
    \fg(E)=\fg_i(E)\oplus\bigoplus_{\alpha\in\Phi(G,S)-\Phi(G_i,S)}\fg_\alpha(E).
    \]
    Then restrict the above extension of $\phi_i$
    to $(G_j)_{x,r_{(j+1)},r_{(j+1)}/2,\fm_j}
    \subset(G_j)_{x,r_{(j+1)}/2}\subset(G_j)_{x,r_{(i+1)}/2+}\subset G_{x,r_{(i+1)}/2+}$.

    \subsection{Refined leading terms}
    Denote $\fs=\Lie(S)$,
    $\fs_0=\Lie(S_0)=\Lie(S_{0+})=\fs_{0+}$.
    The character $\phi$ induces a vector $\td\phi\in\fs^*$.
    As on the spectral side, denote
    \[
    \fa_i:=\fz_i\cap\fg_{i-1}^\der,\quad \fz_i=\bigoplus_{1\leq j\leq i}\fa_j,\quad
    \fz_{i,-r_{(i)}}^*=\bigoplus_{1\leq j\leq i}\fa_{j,-r_{(j)}}^*.
    \]
    Note $\fa_1=\fz_1$, $\fz_p=\fs$.
    Here $\td\phi|_{\fa_i}$ has depth $r_{(i)}$.
    
    The point $x$ in the building of $S$ induces a grading on $\fs_E$ of order $m$,
    which induces a grading on $\fs^*$,
    allowing us to define a section of $\fs_r^*/\fs_{r+}^*$ into $\fs^*$.
    In particular, $\fa_{i,-r_{(i)}}^*$ has an induced grading with a summand isomorphic to
    $\fa_{i,-r_{(i)}}^*/\fa_{i,-r_{(i)}+}^*$.
    Denote by $\phi_{(i)}$ 
    the projection of $\td\phi$ first to $\fa_{i,-r_{(i)}}^*$
    then to the summand isomorphic to $\fa_{i,-r_{(i)}}^*/\fa_{i,-r_{(i)}+}^*$.
    
    We call the tuple $\{\phi_{(1)},...,\phi_{(p)}\}$
    the \emph{refined leading terms} of $\phi$.

    \section{A local geometric Langlands for some regular supercuspidal representations}\label{s:regular sc conj}
    
    We formulate a refinement of Conjecture \ref{c:main}
    for those regular supercuspidal representations
    defined in \S\ref{s:sc refined leading term}.
    
    \subsection{Dual refined leading terms}\label{ss:dual refined leading terms}
    Take a tame elliptic pair $(S,\phi)$ 
    satisfying extra assumptions in \S\ref{s:sc refined leading term},
    such that the unique point $x$ in its building
    is in the closure of fundamental alcove.
    As in \S\ref{s:sc refined leading term},
    we associate to $(S,\phi)$
    the tuple $r_{(1)}>\cdots>r_{(p)}>0$
    and twisted Levi subgroups $G_1\supset\cdots\supset G_p=S$.
    The denominator $m$ is an elliptic number of the Weyl group $W(G,S)$.
    We assume  $mr_{(j)}\in\bZ$.
    
    Let $R=\{r_{(1)},...,r_{(p)}\}$.
    Fix a maximal torus $\check{H}\subset\LG$
    and isomorphism with dual torus $\check{H}\simeq\check{S}$,
    as well as their Lie algebras $\ch\simeq\check{\fs}$.
    Let $E=\bC(\!(u)\!)$, $u=t^{\frac{1}{m}}$,
    and generator $\sigma(u)=\zeta_mu$ of Galois group $\mathrm{Gal}(E/F)$.
    There is elliptic element $w_m\in W$ of order $m$
    such that the isomorphism $\ch\otimes_\bC E\simeq\fs^*\otimes_F E$
    is compatible with the 
    action of $w_m\times\sigma^{-1}$ on the left 
    and the Galois action of $\sigma$ on the right.
    Consider the grading defined by $\Ad_{w_m}$:
    \[
    \ch=\bigoplus_{i\in\bZ/m\bZ}\ch_i.
    \]
    Then we have an isomorphism
    \[
    \ch_{-mr_{(j)}}\simeq\fs_{-r_{(j)}}^*/\fs_{-r_{(j)}+}^*
    \]
    for any $1\leq j\leq p$.
    
    Let $\LM_{(j)}\subset\LG$ be the standard Levi subgroup containing $\check{H}$
    that is dual to $G_j$, with Lie algebra $\Lm_{(j)}$.
    By letting $\LM_{i_j}=\LM_{(j)}$,
    we can define Lie algebras $\La_{(j)}$, $\Lz_{(j)}\subset\ch$ from $\Lm_{(j)}$ 
    as in \S\ref{s:conn refined leading terms}.
    We have isomorphism $\Lz_{(j)}\otimes_\bC E\simeq\fz_j(E)^*$. 
    Here $\fz_j^*$ can be identified with the subspace of $\fg_j^*$ fixed by $G_j$,
    which embeds into $\fs_j^*$ via restriction.
    Moreover, for $i<j$, where $\fg_i\supset\fg_j$ and $\fz_i\subset\fz_j$, 
    the image of $\fz_i^*\hookrightarrow\fg_i^*$
    lands inside the image of $\fz_j^*\hookrightarrow\fg_j^*$
    under the restriction map $\fg_i^*\twoheadrightarrow\fg_j^*$.
    This induces an embedding $\fz_i^*\hookrightarrow\fz_j^*$
    because the composition with restriction
    $\fz_i^*\hookrightarrow\fg_i^*\twoheadrightarrow\fg_j^*
    \xrightarrow{\mathrm{res}_{\fz_i}}\fz_i^*$ is identity.
    Thus we have decomposition $\fz_i^*\simeq\fz_{i-1}^*\oplus\fa_i^*$,
    which induces natural isomorphism
    \begin{equation}\label{eq:La_ij vs fa_i*}
    	\La_{(j)}\otimes_\bC E\simeq\fa_j(E)^*,\qquad
    	\La_{(j)}\cap\ch_{-mr_{(j)}}
    	\simeq\fa_{j,-r_{(j)}}^*/\fa_{j,-r_{(j)}+}^*.
    \end{equation}
    
    Let $\{X_{(j)}\in\La_{(j)}\cap\ch_{-mr_{(j)}}\}$ be the images of
    refined leading terms 
    $\{\phi_{r_{(j)}}\in\fa_{j,-r_{(j)}}^*/\fa_{j,-r_{(j)}+}^*\}$
    of $\phi$ under \eqref{eq:La_ij vs fa_i*}.

    \subsection{The conjectural correspondence}
    
    \subsubsection{A normal subgroup}
    Consider $\tau_i=\fg(F)\cap(\bigoplus_{0\neq\alpha\in\Phi(G_i,S)}\fg(E)_\alpha)$,
    $0\leq i\leq p-1$.
    Define
    \begin{equation}\label{eq:j'}
    	\fj'=\fg_{x,r_{(1)}+}+\sum_{i=0}^{p-1}((\tau_i)_{r_{(i+1)}/2}\cap(\fg_i)_{x,r_{(i+1)},r_{(i+1)}/2,\fm_i})+\sum_{i=2}^p(\fa_i)_{r_{(i)}+}.
    \end{equation}
    Let $\fj^+$ be the ideal of $\fj$ generated by the subspace $\fj'$.
    Denote the preimage of 
    $\fj^+/\fg_{x,r_{(1)}}\subset\fj/\fg_{x,r_{(1)}}\simeq J/G_{x,r_{(1)}}$
    in $J$ by $J^+$.
    
    \begin{lem}
    	The linear form $\tphi$ is trivial on $\fj^+$.
    \end{lem}
    \begin{proof}
    	The linear form $\tphi$ evaluated on $\fj'$ is by
    	first projecting to $\fs_0$ by quotienting out root subspaces,
    	then applying $\phi$.
    	The first step annihilates $\tau_i$.
    	From the definition of $r_{(i)}$,
    	the depth of $\phi$ on $\fa_i$ is $r_{(i)}$.
    	We obtain that $\tphi(\fj')=0$.
    	Since $\tphi$ is a Lie algebra homomorphism,
    	it is trivial on the ideal generated by $\fj'$.
    \end{proof}
    
    \begin{rem}
    	We expect that one can choose Lagrangians $\fm_i\subset V_i$
    	appropriately so that $\fj'$ is itself an ideal of $\fj$.
    	See \eqref{eq:Lagrangian}
    	and Lemma \ref{l:Lagrangian} 
    	for the case of $p=1$.
    \end{rem}

    \subsubsection{}\label{sss:central support}
    Denote 
    \begin{equation}\label{eq:bj}
    	\bj:=\fj/\fj^+.
    \end{equation}
    
    Recall $\fj^+\subset\fj\subset\fg[\![t]\!]$,
    so that $\fj^++\bC\bone\subset\fj+\bC\bone$ are subalgebras of $\hg$.
    Consider
    \begin{equation}\label{eq:Vac_j+}
    	\Vac_{\fj^+}
    	:=\Ind^{\hg}_{\fj^++\bC\bone}\bC
    	=\Ug/\Ug\fj^+
    	=\Ind^{\hg}_{\fj+\bC\bone} U(\bj)=\Ug\otimes_{U(\fj+\bC\bone)}U(\bj).
    \end{equation}
    
    We have
    \[
    \End\Vac_{\fj^+}
    \simeq\Hom_{\fj}(U(\bj),\Ug\otimes_{U(\fj+\bC\bone)}U(\bj))
    \simeq\Vac_{\fj^+}^{J^+}.
    \]
    
    Note that $\bj\simeq\fs_0/\fs_0\cap\fj^+$ is abelian.
    We have $U(\bj)\hookrightarrow\End\Vac_{\fj^+}$.
    Denote $\fZ_{\fj^+}=\mathrm{Im}(\fZ\rightarrow\End\Vac_{\fj^+})$.
    Then $\Op_{\fj^+}:=\Spec\fZ_{\fj^+}$ defines a subspace of opers on $D^\times$.

    Define
    \begin{equation}\label{eq:A intersection}
    	A:=U(\bj)\cap\fZ_{\fj^+}\subset\End(\Vac_{\fj^+}),
    \end{equation}
    which fits into commutative diagram
    \begin{equation}\label{eq:c local quant diagram}
    	\begin{tikzcd}
    		A \arrow[r,hook] \arrow[d,hook] &\fZ_{\fj^+} \arrow[d,hook]\\
    		U(\overline{\fj}) \arrow[r,hook] &\End(\Vac_{\fj^+}) \arrow[r,hook] &\Vac_{\fj^+}
    	\end{tikzcd}
    \end{equation}
    
    \begin{conj}\label{c:local conj}
    	\begin{itemize}
    		\item [(i)]
    		For a sequence of twisted Levi subgroups
    		$G_1,...,G_p$ and a collection of refined leading terms
    		$\phi_{(i)}:\fa_i\rightarrow\bC$, $1\leq i\leq p$
    		as in \S\ref{s:sc refined leading term}
    		composed with projections $\bj\rightarrow\fa_i$,
    		they define $\ker\phi_{(i)}\subset U(\bj)$.
    		Denote $A'=A\cap(\cap_i\ker\phi_{(i)})\subset A$.
    		Let $X_{(i)}\in\check{\fa}_i$ be the dual refined leading term.
    		Let $\Loc(X_{(i)},R,\theta)$ be the space of 
    		connections with such refined leading terms.
    		Then 
    		\begin{equation}
    			\Spec\fZ_{\fj^+}/\langle A'\rangle
    			\simeq p^{-1}(\Loc(X_{(i)},R,\theta))
    		\end{equation}
    		as subspaces of $\Op_{\fj^+}$.
    		In particular, 
    		for any connection underlying an oper in 
    		$\Spec\fZ_{\fj^+}/\langle A'\rangle$,
    		the associated Levi subgroups $\LM_1,...,\LM_p$
    		are dual to $G_1,...,G_p$.
    		
    		\item [(ii)]
    		For any character $\tphi$ of $J$ lifted from $(S,\phi)$
    		as in \S\ref{s:sc refined leading term}, 
    		let $\fZ_{\fj,\tphi}=\fZ_{\fj^+}/\ker(\tphi|_A)$,
    		i.e. the central support of $\Vac_{\fj,\tphi}:=\Vac_{\fj^+}/\ker\tphi$.
    		Then $\Spec(\fZ_{\fj,\tphi})=p^{-1}(\nabla)$
    		for a unique formal connection $\nabla$
    		whose associated slopes are $R=\{r_{(1)},...,r_{(p)}\}$,
    		associated Levi subgroups $\LM_{(1)},...,\LM_{{(p)}}$
    		are dual to $G_1,...,G_p$,
    		and the refined leading terms $X_{(1)},...,X_{(p)}$
    		are dual to refined leading terms $\phi_{(1)},...,\phi_{(p)}$
    		of $\phi$.
    		
    		\item [(iii)]
    		For all the $J$ and refined leading terms defined from pair $(S,\phi)$ as in \S\ref{s:sc refined leading term},
    		the union of the corresponding local systems $\Loc(X_{(i)},R,\theta)$ 
    		are all the irreducible formal connections
    		such that the centralizer of the 
    		image of wild inertia group $\rho_\nabla(I_{\diff}^{0+})$ 
    		is a maximal torus.
    		
    		\item [(iv)]
    		If two such pairs $(S,\phi)$ and $(S',\phi')$
    		are $G(F)$-conjugated,
    		then the induced modules $\Vac_{\fj,\tphi}$, $\Vac_{\fj',\tphi'}$
    		have the same central supports.
    	\end{itemize}
        \end{conj}
        
        We will confirm the above conjecture for toral supercuspidal representations
        and irreducible isoclinic connections,
        see Remark \ref{r:local conj for toral}.
    
        \begin{cor}\label{c:K-type and L-parameter}
        	    Assume Conjecture \ref{c:local conj} holds.
                Let $\nabla$ be a formal connection,
                $\chi$ an oper form of $\nabla$.
                Then $\hg-\mathrm{mod}_\chi$ contains 
                a nonzero $(J,\tphi)$-integrable object 
                for some $(J,\tphi)$ as in \S\ref{s:sc refined leading term}
                only if 
                $\nabla$ is an irreducible formal connection 
                such that the centralizer of the 
                image of wild inertia group is a maximal torus.
        \end{cor}
    	\begin{proof}
    	    Note that $\Vac_{\fj,\tphi}\simeq\Ind^{\hg}_{\ker\tphi+\bC\bone}\bC$.
    		By the same argument as the proof of \cite[Lemma 10.3.2]{FrenkelBook},
    		there exists a nonzero object of $\hg-\mathrm{mod}_\chi$
    		with a nonzero morphism from $\Ind^{\hg}_{\ker\tphi+\bC\bone}\bC$.
    		Thus $\chi$ is in the central support of $\Vac_{\fj,\tphi}$,
    		so that its underlying connection 
    		is as in Conjecture \ref{c:local conj}.(iii).
    	\end{proof}
    
        We expect the converse of the above corollary to hold as well.
        We confirm this expectation in the case when the slope is $1+\frac{1}{h}$,
        $h$ the Coxeter number, using global method, 
        see Remark \ref{r:Airy Langlands}.(ii).
        
        \begin{rem}\label{r:explicit local Langlands}
        	\begin{itemize}
        		\item [(i)]
        		Part (iv) of Conjecture \ref{c:local conj}
        		is inspired by \cite[Corollary 3.7.10]{KalethaRegular}.
        		
        		\item [(ii)]
        		We expect that the algebra $A$ is contained in 
        		the algebra generated by images of 
        		Segal-Sugawara operators $S_{i,j}\in\fZ$(see \S\ref{ss:SS operators})
        		satisfying the index condition
        		$-r_1\leq -j+(d_i-1)(r_1+1)\leq\sum_{a=1}^{p-1}(h_a-1)(r_{(a)}-r_{(a+1)})
        		+(h_p-1)r_{(p)}$
        		as in Proposition \ref{p:supp of oper form}.
        		In fact, the proof of Proposition \ref{p:local quant diagram}
        		for the toral supercuspidal case
        		shows that the images of $S_{i,j}$
        		for $-r_1\leq -j+(d_i-1)(r_1+1)<0$ are contained in $A$.
        		However, 
        		unlike the toral supercuspidal case \S\ref{s:isoclinic local Langlands},
        		in general $A$ is not expected to be freely generated by
        		$S_{i,j}$: Lemma \ref{l:dim match conn-oper} indicates that
        		if images of $S_{i,j}$ for $-r_1\leq -j+(d_i-1)(r_1+1)<0$ in $A$
        		are algebraically independent,
        		then $\LM_{(1)}$ should be a maximal torus.
        		
        		\item [(iii)]
        		If $A$ is generated by the images of those $S_{i,j}$ in (ii),
        		then the underlying connection of any oper
        		in the central support of $\Vac_{\fj,\tphi}$
        		can be given by an oper canonical form
        		whose coefficient $v_{i,j}$ for those $(i,j)$ in (ii)
        		is given by $\tphi(S_{i,j})$,
        		and all other $v_{i,j}$ are zero.
        		This provides an explicit recipe for the Langlands parameters.
        		We will see in \S\ref{s:isoclinic local Langlands} that 
        		this is the case for toral supercuspidal representations.
        	\end{itemize}
        \end{rem}

	\section{Isoclinic formal connections and opers}\label{s:isoclinic oper}\label{s:isoclinic and oper}
	We study the relationship between oper canonical forms 
	and connection canonical forms 
	for isoclinic formal connections.
	
	\subsection{Isoclinic formal connections}\label{ss:isoclinic conn}
	
	\subsubsection{}	
	Following \cite{JYDeligneSimpson},
	we call a formal connection $\nabla$ \emph{isoclinic}
	if in its canonical form \eqref{eq:canonical form},
	the leading term $X_{r_1}$ is regular semisimple.
	For such connections, $R=\{r_1\}$, $p=1$,
	the sequence of Levi subgroups $(\LM_{i_j})_j$
	consists of only a single maximal torus centralizing $X_{r_1}$,
	and there are no higher leading terms beyond $X_{r_1}$.
	Clearly this notion is independent of the choice of canonical form
	and the gauge transformation.
	In particular, such $\nabla$ is irregular.
	The irreducible isoclinic formal connections are indeed the de Rham counterpart
	of the \emph{toral supercuspidal parameters}
	\cite[Definition 6.1.1]{KalethaRegular}.

	For an irregular formal connection $\nabla$ with slope $\nu$,
	write $\nu=\frac{N}{m}$ where $(N,m)=1$.
	We collect some results on isoclinic formal connections 
	from \cite{JYDeligneSimpson}.
	Recall that $\crho:\bGm\rightarrow\cT$.
	Fix once and for all a primitive $m$-th root of unity $\zeta_m$.
	Let $\theta=\Ad_{\crho(\zeta_m)}$ be an inner automorphism of $\Lg$
	that defines a grading
	\begin{equation}\label{eq:Lg grading}
		\Lg=\bigoplus_{i\in\bZ/m\bZ}\Lg_i.
	\end{equation}

	\begin{lem}\label{l:isoclinic}
		A formal connection $\nabla$ is isoclinic 
		with slope $\nu=\frac{N}{m}>0$, $(N,m)=1$,
		if and only if
		$m$ is a regular number of $W$,
		and it is $G(\overline{F})$-gauge equivalent
		to a form as follows:
		\begin{equation}\label{eq:isoclinic canonical form}
			\td+(X_{-N}t^{-N/m}+\cdots+X_{-1}t^{-1/m}+X_0)\frac{\td t}{t},
		\end{equation}
		where $X:=X_{-N}$ is regular semisimple,
		$X_i\in\fz_{\Lg}(X)=:\ct_X$,
		$X_i\in\Lg_i$.
		Moreover, if $m$ is elliptic, then $\nabla$ is irreducible.
	\end{lem}
	\begin{proof}
		The if part follows from \cite[9.8, Proposition]{BV}.
		
		For the only if part, 
		first we know $m$ is a regular number from
		\cite[Lemma 2.2]{JYDeligneSimpson}.
		By \cite[Theorem 3.2.5]{OY}, 
		$x=\crho/m$ is $m$-regular in the sense of 
		\cite[Definition 2.4]{JYDeligneSimpson}.
		By \cite[Lemma 2.10]{JYDeligneSimpson},
		$\nabla$ can be $G(\overline{F})$-gauge transformed to
		a form as \eqref{eq:isoclinic canonical form}
		except $X_0$.
		Since $X$ is regular semisimple,
		we can use a gauge transformation to make $X_0\in\ct_X$.
		Let $\cT_X=C_{\LG}(X)$, $W\simeq N_{\LG}(\cT_X)/\cT_X$.
		Since $X$ is $\theta$-eigen,
		$\theta$ acts on $\ct_X$,
		where it acts as a regular element of order $m$ of $W$.
		By the proof of \cite[Lemma 2.2]{JYDeligneSimpson},
		there exists regular $w$ with 
		the same action as $\theta$ on $X_{-N},...,X_{-1}$,
		and $w$ fixes $X_0$.
		Since $X\in\ct_X$ is a regular vector fixed by $w^{-1}\theta$,
		we have $\theta=w$ in $W$.
		Thus $X_0$ is fixed by $\theta$, i.e. $X_0\in\Lg_0$.
		
		Note that the differential Galois group of $\nabla$
		contains the smallest torus whose Lie algebra contains $X$,
		together with a lift $w$ of an order $m$ regular element
		in the Weyl group $W(\LG,\cT_X)$.
		By the same proof as \cite[Proposition 12.(v)]{CYTheta},
		if $w$ is elliptic, $\nabla$ is irreducible. 
	\end{proof}

    \begin{lem}\label{l:isoclinic canonical form}
    	Two irreducible isoclinic formal connections are isomorphic,
    	i.e. differ by a $\LG(F)$-gauge transform,
    	if and only if they have isomorphic canonical forms,
    	i.e. differ by a $\LG(\overline{F})$-gauge transform.
    \end{lem}
    \begin{proof}
    	Only the if part needs to be proved.
    	Since we assumed $\LG$ to be simple, 
    	the proof is contained in the proof of \cite[Lemma 2.10]{JYDeligneSimpson}
    	which we recall.
    	
    	Let $\nabla,\nabla'$ be irreducible isoclinic formal connections
    	with isomorphic canonical forms.
    	Let $\rho,\rho':I_{\diff}\rightarrow\LG$
    	be their monodromy representations.
    	Then there exists a finite covering $D^\times_a=\Spec\bC(\!(t^{1/a})\!)$
    	with differential Galois group $I_{\diff}^a\subset I_{\diff}$
    	such that the restrictions $\rho|_{I_{\diff}^a},\rho'|_{I_{\diff}^a}$
    	can be $\LG$-conjugated to the same monodromy representation
    	$\pi:I_{\diff}^a\rightarrow\cT$
    	landing inside a maximal torus $\cT$.
    	
    	We may assume $\rho|_{I_{\diff}^a}=\rho'|_{I_{\diff}^a}=\pi$.
    	Note $I_{\diff}/I_{\diff}^a\simeq\mu_a$ is cyclic with a generator $\gamma$.
    	By \cite[Lemma 2.2.(1)]{JYDeligneSimpson},
    	$\rho(\gamma),\rho'(\gamma)$ normalize $\cT$ and map to the same element
    	$wT$ in the Weyl group,
    	which must be elliptic by Lemma \ref{l:isoclinic}.
    	For elliptic elements, their lifts to $\LG$ are all $\cT$-conjugated.
    	Thus $\rho,\rho'$ are $\cT$-conjugated.
    	This proves the statement.
    \end{proof}
    
    \subsubsection{}
    Denote by $\Loc(X,\nu)$ the isomorphism classes of 
    irreducible isoclinic formal connections with slope $\nu$ 
    and leading term $X$ in its canonical form.
    Any $\nabla\in\Loc(X,\nu)$ has a canonical form as in Lemma \ref{l:isoclinic}.
    Observe that its differential Galois group is contained in
    $\cT_X$ and a lift of order $m$ regular Weyl group element $w_m$.
    Thus for it to be irreducible, 
    $w_m$ must have no nonzero fixed point on $\ct_X$,
    i.e. $w_m$ is \emph{elliptic}.
    Equivalently, $m$ is a regular elliptic number of $W$,
    and the grading define by $\crho/m$
    are exactly the inner stable gradings,
    see \cite{RLYG} where such $m$ and stable gradings are classified.
    
    \begin{lem}\label{l:Loc(X,nu)}
    	The isomorphism classes of irreducible isoclinic connections
    	$\Loc(X,\nu)$ with fixed leading term $X$ and slope $\nu$
    	are in one-to-one correspondence with the following connection forms:
    \end{lem}
    \begin{equation}\label{eq:Loc(X,nu)}
    	\{\td+(X_{-N}t^{-N/m}+\cdots+X_{-1}t^{-1/m})\frac{\td t}{t}\mid
    	X_{-N}=X, X_i\in\ct_X\cap\Lg_i=:\ct_{X,i}\}.
    \end{equation}
    \begin{proof}
    	By Lemma \ref{l:isoclinic canonical form},
    	isomorphism classes of irreducible isoclinic connections
    	are in bijection with isomorphism classes of their canonical forms.
    	If two canonical forms as in \eqref{eq:isoclinic canonical form}
    	with $X_{-N}=X$ are gauge equivalent,
    	their irregular parts must differ by a gauge transform 
    	by some element $g\in\cT_X$,
    	which acts as identity on such equations.
    	Also, note that $X_0\in\ct_{X,0}=0$ as $w_m$ is elliptic.
    	Thus such canonical form is unique.
    \end{proof}
    
    \subsubsection{}
    We collect an observation for future use.
    \begin{lem}\label{l:isoclinic can remove regular part}
    	Assume $N\geq1$.
    	An irregular formal connection of the form
    	\[
    	\nabla=\td+\sum_{i=-N}^\infty X_it^i \frac{\td t}{t},\quad
    	X_i\in\Lg, 
    	\]
    	with $X_{-N}$ regular semisimple
    	has the same canonical form as 
    	$\nabla'=\td+\sum_{i=-N}^0 X_it^i\frac{\td t}{t}$.
    \end{lem}
    \begin{proof}
    	Let $X=X_{-N}$, $\ct_X=\fz_{\Lg}(X)$ be a Cartan subalgebra.
    	Write $\Lg=\ct_X\bigoplus_{\alpha\in\Phi(\Lg,\ct_X)}\Lg_\alpha$.
    	Let $\pi_\alpha$ be the projection from $\Lg$ to $\Lg_\alpha$.
    	
    	The first step in transforming $\nabla$ and $\nabla'$ into canonical form
    	is to apply the gauge transformations 
    	to make them into $\td+\ct_X(\!(t)\!)\td t$.
    	Explicitly,
    	we gauge transform by convergent products of elements of the form
    	$\exp(-\alpha(X)^{-1}\pi_\alpha(X_i')t^{N+i})$
    	for $\alpha\neq 0$ and $i=-N+1,-N+2,...$,
    	where $X_{-N+1}'=X_{-N+1}$, and $X_i'$ for $i>-N+1$ are determined by
    	recursive relations.
    	We obtain two formal connections of the form
    	$\td+\sum_{i=-N}^\infty Y_it^i \frac{\td t}{t} $
    	where $Y_{-N}=X$, all $Y_i\in\ct_X$,
    	and $Y_{-N},...,Y_0$ are determined by $X_{-N},...,X_0$.
    	The two connections only differ in their regular parts.
    	
    	The second step is to remove the regular part.
    	This is done by applying the gauge transform by 
    	$\exp(-\sum_{i\geq 1}t^ii^{-1}Y_i)\in\cT_X(\!(t)\!)$,
    	which does not change $Y_i$ for $i<0$.
    	The two connections arrive at the same canonical form.
    \end{proof}
	
	\begin{rem}
		The above lemma is false for regular formal connections.
	\end{rem}

	\subsection{Oper canonical forms of irreducible isoclinic formal connections}
	Let $m$ be a regular elliptic number of $W$, $(N,m)=1$, $\nu=\frac{N}{m}$. 
	We know from \cite[Lemma 2.5]{JYDeligneSimpson} that
	$\Lg_{-N}$ contains regular semisimple elements.
	Let $X\in\Lg_{-N}^{\rs}$,
	$\ct_X=\fz_{\Lg}(X)$, $\cT_X=C_{\LG}(X)$.
	Since $X$ is $\theta$-eigen, $\theta$ acts on $\ct_X$,
	$\ct_X=\bigoplus_{i\in\bZ/m\bZ}\ct_{X,i}$.
    
    Denote the space of opers whose underlying connections belong to $\Loc(X,\nu)$
    by 
    \begin{equation}\label{eq:Op(X,nu)}
    	\Op_{\Lg}(X,\nu):=p^{-1}(\Loc(X,\nu)).
    \end{equation}
    
    Define
    \begin{equation}\label{eq:ell_ij}
    	\ell_{i,j}:=
    	(d_i-1)N+m(d_i-1-j)
    \end{equation}
    and
    \begin{equation}\label{eq:A_nu tA_nu}
    	A_\nu
    	=\{(i,j)\mid 1\leq i\leq n,\ 
    	-N+1\leq\ell_{i,j}\leq-1\},\quad
    	\tilde{A}_\nu
    	=\{(i,j)\mid 1\leq i\leq n,\ 
    	-N\leq\ell_{i,j}\leq-1\}.
    \end{equation}    

    \begin{lem}\label{l:dim match conn-oper}
    	\begin{itemize}
    		\item [(i)]
    		The set $A_{\nu,\ell}:=\{(i,j)\in A_\nu\mid \ell=\ell_{i,j} \}$
    		has cardinality $\dim\ct_{X,\ell}$
    		for any $\ell$.
    		
    		\item [(ii)]
    		$\Loc(X,\nu)\simeq\bA^{|A_\nu|}$.
    	\end{itemize}
    \end{lem}
    \begin{proof}
    	Denote $k_\ell:=|\{1\leq i\leq n\mid d_i-1\equiv\ell\mod m\}|$.
    	Denote by $N^{-1}\in\bZ/m\bZ$ the inverse of $N\mod m$.
    	We have
    	\begin{align*}
    	&|\{(i,j)\in A_\nu\mid \ell=\ell_{i,j} \}|\\
    	=&|\{1\leq i\leq n\mid \ell\equiv(d_i-1)N\mod m\}|\\
    	=&|\{1\leq i\leq n\mid d_i-1\equiv N^{-1}\ell\mod m\}|\\
    	=&k_{N^{-1}\ell}.
    	\end{align*}
        
        Let $\theta'=\theta^{-N^{-1}}$, 
        $\Lg_i'=\Lg_{-Ni}$ the $\zeta_m^i$-eigenspace of $\theta'$.
        Then $X\in\Lg_{-N}=\Lg_1'$ is regular semisimple.
        Applying \cite[Theorem 4.2.(i)]{Panyushev} and its proof to $\theta'$,
        we obtain 
        \[
        k_{N^{-1}\ell}
        =\dim(\ct_X\cap\Lg_{-N^{-1}\ell}')
        =\dim(\ct_X\cap\Lg_\ell)
        =\dim(\ct_{X,\ell}).
        \]
        This proves part (i).
        Part (ii) follows from (i) immediately.
    \end{proof}

    By Lemma \ref{l:Loc(X,nu)}, we have isomorphism
    \[
    \Loc(X,\nu)\simeq\bigoplus_{-N<\ell\leq -1}\ct_{X,\ell}.
    \]
    Fix a basis $\{E_a^\ell\mid 1\leq a\leq\dim\ct_{X,\ell}\}$ of $\ct_{X,\ell}$.
    Let $e_a^\ell$ be the coordinate on $\Loc(X,\nu)$ for $E_a^\ell$.
    
	\begin{prop}\label{p:oper fiber over Loc}\mbox{}		
		\begin{itemize}
			\item [(i)] The space $\Op_{\Lg}(X,\nu)$ is isomorphic to
			the following affine space of oper canonical forms:
			\[
			\{
			\td+(p_{-1}+\sum_{i=1}^n\sum_{\ell_{i,j}\geq -N} v_{i,j}t^{-j-1}p_i)\td t\mid\ v_{i,j}\in\bC,\ 
			p_{-1}+\sum_{m|d_iN}v_{i,d_i+\frac{d_iN}{m}-1}p_i\in\LG\cdot X
			\}.
			\]
			
			\item [(ii)]
			In terms of the basis $v_{i,j}$ and $E_a^\ell$,
			the surjection $p:\Op_{\Lg}(X,\nu)\rightarrow\Loc(X,\nu)$
			between two affine spaces
			is of the form
			\begin{equation}
				e_a^\ell=f_{a,\ell}(v_{ij}\mid \ell_{i,j}<\ell)
				+\sum_{\ell_{i,j}=\ell}c_{(i,j),(a,\ell)}v_{i,j},
			\end{equation}
		    where $-N<\ell\leq -1$,
		    $f_{a,\ell}$ is a polynomial in those $v_{ij}$ satisfying $\ell_{i,j}<\ell$,
		    $c_{(i,j),(a,\ell)}\in\bC$.
		    
		    \item [(iii)]
		    The map $p:\Op_{\Lg}(X,\nu)\rightarrow\Loc(X,\nu)$
			factors as the projection forgetting 
			$v_{i,j}$, $\ell_{i,j}\geq0$
			composed with an isomorphism
			\begin{equation}\label{eq:Op=Loc(X,nu)}
				\begin{split}
					&\{\td+(p_{-1}+\sum_{i=1}^n\sum_{-N\leq\ell_{i,j}\leq -1} v_{ij}t^{-j-1}p_i)\td t\mid\ v_{i,j}\in\bC,\
 					p_{-1}+\sum_{m|d_iN}v_{i,d_i+\frac{d_iN}{m}-1}p_i\in\LG\cdot X\}\\
					&\simeq\Loc(X,\nu).
				\end{split}
			\end{equation}
		    In particular,
		    the linear transformation
		    $\phi_\ell:
		    (v_{i,j})_{\ell_{i,j}=\ell}\mapsto
		    (\sum_{\ell_{i,j}=\ell}c_{(i,j),(a,\ell)}v_{i,j})_a$	
		    is invertible for any $-N<\ell\leq -1$ with
		    $\dim\ct_{X,\ell}>0$.		    
		\end{itemize}
	\end{prop}
    \begin{proof}
    	(i):
    	Let $\nabla\in\Loc(X,\nu)$.
    	Let $\nabla_\chi\in p^{-1}(\nabla)$ 
    	with oper canonical form \eqref{eq:oper form}.
    	From the slope formula \eqref{eq:oper slope},
    	$v_{i,j}=0$ for $\ell_{i,j}<-N$,
    	and at least one of $v_{i,j}$ with $\ell_{i,j}=-N$ is nonzero.
    	
    	Fix $u=t^{1/m}$.
    	Applying $\Ad_{\crho(u^{N+m})}$ to $\nabla_\chi$,
    	we obtain
    	\begin{equation}\label{eq:oper to canonica form}
    	\td+m[u^{-N}(p_{-1}+\sum_{m|d_iN}v_{i,d_i+\frac{d_iN}{m}-1}p_i)+
    	\sum_{i=1}^n\sum_{\ell_{i,j}>-N}v_{i,j}u^{\ell_{i,j}}p_i]\frac{\td u}{u}
    	-(N+m)\crho\frac{\td u}{u}.
    	\end{equation}
    	
    	Note that the leading term of $\nabla$ with respect to $u$
    	is $mX$, due to $\frac{\td t}{t}=m\frac{\td u}{u}$.
    	By definition,
    	the semisimple part of $p_{-1}+\sum_{m|d_iN}v_{i,d_i+\frac{d_iN}{m}-1}p_i$
    	is conjugate to $X$.
    	Since $X$ is regular semisimple,
    	$p_{-1}+\sum_{m|d_iN}v_{i,d_i+\frac{d_iN}{m}-1}p_i$ must also be.
    	Thus $p_{-1}+\sum_{m|d_iN}v_{i,d_i+\frac{d_iN}{m}-1}p_i\in\LG\cdot X$
    	and is the unique such element in the Kostant section.
    	This proves (i).
    	
    	(ii):
    	We resume the computation of canonical form in part (i).
    	By Lemma \ref{l:isoclinic can remove regular part},
    	we can drop $v_{i,j}$ for $\ell_{i,j}>0$.
    	Also, since $m$ is elliptic, $\ct_{X,0}=0$,
    	so that the canonical form has vanishing regular singular term.
    	This proves the first assertion in part (iii).
    	
    	To reduce the equation into a canonical form as in \eqref{eq:Loc(X,nu)},
    	we first conjugate by a fixed constant element $g\in\LG$
    	such that $\Ad_g(p_{-1}+\sum_{m|d_iN}v_{i,d_i+\frac{d_iN}{m}-1}p_i)=X$.
    	Next we apply the same algorithm as in the proof of 
    	Lemma \ref{l:isoclinic can remove regular part}
    	to project $v_{i,j}u^{\ell_{i,j}}p_i$ into $\ct_{X,\ell_{i,j}}u^{\ell_{i,j}}$
    	and eliminate the regular part.
    	Part (ii) is clear from the calculation procedure.
    	
    	(iii):
    	We have already seen in the above discussion that
    	$p$ factors through the projection to 
    	\[
    	\{\td+(p_{-1}+\sum_{i=1}^n\sum_{-N\leq\ell\leq -1} v_{ij}t^{-j-1}p_i)\td t\mid\
    	v_{i,j}\in\bC,\ p_{-1}+\sum_{m|d_iN}v_{i,d_i+\frac{d_iN}{m}-1}p_i\in\LG\cdot X\},
    	\]
    	on which it is isomorphic to a surjective linear map from
    	$\{v_{i,j}\mid -N<\ell_{i,j}\leq -1\}$
    	to $\{e_a^\ell\mid-N<\ell\leq -1\}$.
    	By Lemma \ref{l:dim match conn-oper},
    	this linear map must be invertible.
    	This completes the proof of (iii).
    \end{proof}

    We can regard the equation \eqref{eq:Op=Loc(X,nu)} as 
    the `minimal' oper canonical form of the formal connection.

    \begin{rem}
    	\begin{itemize}
    		\item [(i)]
    		Let $S$ be the centralizer of
    		$\tilde{X}=p_{-1}+\sum_{m|d_iN}v_{i,d_i+\frac{d_iN}{m}-1}t^{-d_i-\frac{d_iN}{m}}p_i$ in $\LG(F)$.
    		Then the connections in \eqref{eq:Op=Loc(X,nu)}
    		are called \emph{$S$-toral} of slope $\nu$ in \cite{KSRigid}.
    		
    		\item [(ii)]
    		When $m$ is not necessarily elliptic,
    		we still have a slight variant of 
    		Proposition \ref{p:oper fiber over Loc}
    		with $\Loc(X,\nu)$
    		replaced by \emph{irregular types}
    		of such formal connections.
    	\end{itemize}
    \end{rem}

	\subsection{Global extension}
	We write down an extension to $\bP^1-\{0\}$ of the isoclinic formal connections
	$\Loc(X,\nu)$.
	
	\begin{prop}\label{p:global ext nu>=1}
		Assume $\nu=\frac{N}{m}$.
		Let $\nabla$ be a $\LG$-connection over $\bGm$ with the following connection form:
		\begin{equation}\label{eq:global ext}
			\td+(t^{-2}p_{-1}+\sum_{i=1}^n\sum_{-N\leq\ell_{i,j}\leq -1} v_{ij}t^{-j+2d_i-3}p_i)\td t,\quad	p_{-1}+\sum_{m|d_iN}v_{i,d_i+\frac{d_iN}{m}-1}p_i\in\LG\cdot X.
		\end{equation}
	    Denote by $\Loc(X,\nu,\bP^1-\{0\})$ the isomorphism classes of such connections.
	\begin{itemize}
		\item [(i)]
		$\nabla|_{D_0^\times}$ is isomorphic to the isoclinic formal connection
		\[
		\td+(p_{-1}+\sum_{i=1}^n\sum_{-N\leq\ell_{i,j}\leq -1} v_{ij}t^{-j-1}p_i)\td t.
		\]
		The restriction to $D_0^\times$
		induces an isomorphism 
		$\Loc(X,\nu,\bP^1-\{0\})\simeq\Loc(X,\nu)$
		factoring through \eqref{eq:Op=Loc(X,nu)}.
		In particular, 
		every connection in $\Loc(X,\nu,\bP^1-\{0\})$
		is represented by a unique equation of the form \eqref{eq:global ext}.
		
		\item [(ii)]
		 When $\nu\geq 1$, 
		 $\nabla$ has trivial monodromy at $\infty$.
	\end{itemize}
	\end{prop}
    \begin{proof}
    	(i):
    	Replace $t=u^m$ and apply gauge transform by $\crho(u^{N-m})$ to
    	\eqref{eq:global ext}.
    	We obtain
    	\[
    	\td+m(p_{-1}u^{-N}+\sum_{i=1}^n\sum_{-N\leq\ell_{i,j}\leq -1} v_{ij}u^{\ell_{i,j}}p_i)\frac{\td u}{u}-(N-m)\crho\frac{\td u}{u}.
    	\]
    	
    	By the proof of Proposition \ref{p:oper fiber over Loc}.(ii),
    	the above equation has the same canonical form as 
    	\eqref{eq:oper to canonica form},
    	thus is isomorphic to a connection as \eqref{eq:Op=Loc(X,nu)}.
    	Part (i) follows from Lemma \ref{l:Loc(X,nu)} 
    	and Proposition \ref{p:oper fiber over Loc}.
    	
    	(ii):
    	Denote $s=t^{-1}$.
    	We have
    	\[
    	\nabla|_{D_\infty^\times}\simeq
    	\td-(p_{-1}+\sum_{i=1}^n\sum_{-N\leq\ell_{i,j}\leq-1}v_{ij}s^{j-2d_i+1}p_i)\td s.
    	\]
    	Note that 
    	\[
    	\ell_{i,j}\leq-1
    	\quad\Rightarrow\quad
    	\frac{(d_i-1)(N-m)+1}{m}-1\leq j-2d_i+1.
    	\] 
    	Since $N\geq m$,
    	$j-2d_i+1>-1$,
    	so that $\nabla|_{D_\infty^\times}$ has trivial monodromy.
    \end{proof}
    
    \begin{rem}
    	\begin{itemize}
    		\item [(i)]
    		In the above and the proof of surjectivity in \eqref{eq:Op=Loc(X,nu)},
    		we did not essentially use the ellipticity of $m$.
    		Thus we gain an explicit construction of
    		a connection on $\bP^1-\{0\}$
    		with arbitrary isoclinic irregular type at $0$.
    		In general, it is proved in
    		\cite[Corollary 1.14.(ii)]{JYDeligneSimpson} that
    		there exists a connection on $\bGm$
    		with arbitrary isoclinic irregular type at $0$
    		and arbitrary regular monodromy at $0$
    		(when $\nu\geq1$, the monodromy can be arbitrary).
    		Note that in \emph{loc. cit.}, $0$ and $\infty$ are switched.
    		
    		\item [(ii)]
    		If we regard the equations in \eqref{eq:Op=Loc(X,nu)}
    		as having coefficients in $\bC[t,t^{-1}]$,
    		then they have regular singularity at $\infty$,
    		but with not necessarily trivial monodromy.
    	\end{itemize}
    \end{rem}

	\section{Central support of toral K-types}\label{s:toral central support}
	The counterpart on the automorphic side
	of irreducible isoclinic formal connections 
	will be the so-called
	\emph{toral supercuspidal representations} \cite[\S4.8]{KalethaRegular}.
	This family of supercuspidal representations
	is originally due to Adler \cite{Adler},
	and contains epipelagic representations as special cases.
	We give a brief review of them following the exposition in
	\cite[\S3.6]{FintzenIHES}.
	
	\subsection{Recollection on toral supercuspidal representations}\label{ss:toral sc repn}
	
	In brief, 
	toral supercuspidal representations
	are those supercuspidal representations
	induced from Yu data of the form
	$(S\subset G;x;\nu;\phi,1)$,
	where the only proper twisted Levi subgroup is an elliptic torus.
	To make it more explicit,
	we carry out the construction in \S\ref{s:sc refined leading term}
	in this situation.
	
	\subsubsection{}
	Let $m$ be a regular elliptic number of $W=N_G(T)/T$.
	Let $\crho_G\in X_*(T/Z(G))$ be the half sum of positive coroots,
	and $x$ the point in the closure of fundamental alcove
	of Bruhat-Tits building $\mathcal{B}(G,T)$
	that is conjugated to $\frac{\crho_G}{m}$.
	Write $\lambda=mx\in X_*(T/Z(G))$, which is conjugated to $\crho_G$. 
	Then $\Theta=\Ad_{\lambda(\zeta_m)}$ defines a grading
	\[
	\fg=\bigoplus_{i\in\bZ/m\bZ}\fg_i.
	\]
		
	Let $\nu=\frac{N}{m}$ where $(N,m)=1$.
	As on the spectral side,
	there exists a regular semisimple element 
	$Y\in\fg_{-N}^{\mathrm{rs}}\simeq\fg_N^{*,\mathrm{rs}}$.
	Denote $T_Y=C_G(Y)$, $\ft_Y=\fz_{\fg}(Y)$.
	Since $Y$ is $\Theta$-eigen, $\Theta$ preserves $\ft_Y$
	and induces grading
	\[
	\ft_Y=\bigoplus_{i\in\bZ/m\bZ}\ft_{Y,i}.
	\]
	Note that since the action of $\Theta$ on $T_Y$
	is a regular elliptic element of order $m$ in the Weyl group $W(G,T_Y)$,
	we have $\ft_{Y,0}=0$.
	
	Fix $u=t^{1/m}$ and denote $E=\bC(\!(u)\!)$.
	We have a Moy-Prasad grading on $\fg(\!(t)\!)$:
	\begin{equation}\label{eq:g((t)) to g((u))}
		\psi:\fg(\!(t)\!)\simeq\widehat{\bigoplus_{i\in\bZ}}\fg_iu^i,\quad
		X(t)\mapsto \Ad_{\lambda(u)}X(u^m).
	\end{equation}

	Denote $r_m:u\mapsto\zeta_m^{-1}u$ and 
	$\sigma=\Theta\times r_m\in\mathrm{Aut}(\fg(E))$.
	The right-hand side of the above is $\fg(E)^\sigma$.
	Henceforth we ignore the isomorphism $\psi$
	and will not distinguish $\fg(\!(t)\!)$
	from $\widehat{\bigoplus}_{i\in\bZ}\fg_iu^i$.
	Denote $\fg(\!(t)\!)_i\simeq\fg_iu^i$
	and $\fg(\!(t)\!)_{\geq i}\simeq\widehat{\bigoplus}_{j\geq i}\fg(\!(t)\!)_i$.
	
	The element $Y\in\fg_{-N}^{\mathrm{rs}}$ defines an element
	$Y'=\psi^{-1}(Yu^{-N})$, 
	which is regular semisimple in $\fg(\!(t)\!)$.
	Let $S=C_{G(\!(t)\!)}(Y')$ be the centralizer
	and $\fs=\fz_{\fg(\!(t)\!)}(Y')$.
	Since $S\simeq(\mathrm{Res}_{E/F}(T_Y\times_\bC E))^{\Theta\times r_m}$,
	$S$ is an elliptic torus.
	We have
	\[
	\fs\simeq\widehat{\bigoplus_{i\in\bZ}}\ft_{Y,i}u^i.
	\]
	
	Note that $x=\lambda/m$ is the unique point in the Bruhat-Tits building of $S$.
	Denote by $S_r$ the Moy-Prasad subgroup of $S$ with depth $r\in\bR_{\geq0}$
	and $S_{r+}=\bigcup_{s>r}S_s$.
	
	Now take arbitrary elements $Y_i\in\ft_{Y,i}$ for $-N+1\leq i\leq -1$
	and denote $Y_{-N}=Y$.
	Then $\psi^{-1}(\sum_{i=-N}^{-1}Y_iu^i)$
	determines a character of $S_{0+}$ of depth $\nu$.
	Let $\phi:S\rightarrow\bC^\times$ be a character 
	whose differential is as above.
	Such $(S,\phi)$ gives a \emph{tame elliptic regular pair} of depth $\nu$
	\cite[Definition 3.7.5]{KalethaRegular}.
	
	\subsubsection{}
	We need to extend $\phi$ further.
	Let $P=G_x$ be the parahoric subgroup determined by $x$,
	and $P(i)=G_{x,i/m}$ its Moy-Prasad subgroups.
	Denote their Lie algebras by $\fp,\fp(i)$.
	Denote $S(i)=S\cap P(i)$, $\fs(i)=\fs\cap\fp(i)$.
	Note $S(1)=S_{0+}$.
	The construction varies slightly depending on the parity of $N$.
	
	Case 1: $N$ is odd.
	
	In this case $G_{x,\frac{\nu}{2}}=G_{x,\frac{\nu}{2}+}=P(\frac{N+1}{2})$.
	Consider the root subspaces with respect to $\ft_Y$:
	\begin{equation}\label{eq:tau}
		\fg=\ft_Y\oplus\tau,\qquad
		\tau=\bigoplus_{\alpha\in\Phi(\fg,\ft_Y)}\fg_\alpha.
	\end{equation}
	The automorphism acts on $\ft_Y$ and $\Phi(\fg,\ft_Y)$,
	thus acts on $\tau$.
	Let $\tau_i=\tau\cap\fg_i$, then
	\[
	\fg_i=\ft_{Y,i}\oplus\tau_i.
	\]
	Here although the notations $\fg_i,\tau_i$ 
	coincide with the Lie algebras of twisted Levi subalgebras in a general Yu datum, 
	this should raise no confusion.
	
	Since $P(\frac{N+1}{2})/P(N+1)$ is abelian and unipotent,
	$P(\frac{N+1}{2})/P(N+1)\simeq\fp(\frac{N+1}{2})/\fp(N+1)$.
	Consider projection
	\begin{align*}
		\pi:P(\frac{N+1}{2})\twoheadrightarrow&
		P(\frac{N+1}{2})/P(N+1)\simeq\fp(\frac{N+1}{2})/\fp(N+1)
		\simeq
		\bigoplus_{i=\frac{N+1}{2}}^{N}\fg_iu^i\\
		\twoheadrightarrow&
		\bigoplus_{i=\frac{N+1}{2}}^{N}\ft_{Y,i}u^i
		\simeq S(\frac{N+1}{2})/S(N+1).
	\end{align*}
	
	Let $J=S(1)P(\frac{N+1}{2})$.
	Define a character $\tphi$
	by $\tphi(G_{x,\nu+})=1$,
	$\tphi|_{S(1)}=\phi$,
	$\tphi|_{G_{x,\frac{\nu}{2}}}=\phi\circ\pi$.\\	
	
	Case 2: $N$ is even.
	
	Note that $Y\in\fg_N^*$ induces a non-degenerate symplectic form on
	$G_{x,\frac{\nu}{2}}/S(\frac{N}{2})G_{x,\frac{\nu}{2}}+
	\simeq\fg_{\frac{N}{2}}/\fs_{\frac{N}{2}}\simeq\tau_{\frac{N}{2}}$ by
	\begin{equation}\label{eq:symp form}
		(x,y)\in\tau_{\frac{N}{2}}\times\tau_{\frac{N}{2}}
		\rightarrow\langle Y,[x,y]\rangle.
	\end{equation}

	We choose a Lagrangian subspace $\fm\subset\tau_{\frac{N}{2}}$ as follows.
	Recall that the order $m$ automorphism $\Theta$ is regular elliptic,
	so that the subgroup of automorphisms generated by $\Theta$ 
	acts freely on $\Phi(\fg,\ft_Y)$.
	For each $\Theta$-orbit $O\subset\Phi(\fg,\ft_Y)$, 
	pick a representative $\alpha_O\in O$.
	Fix a basis $X_{\alpha_O}\in\fg_{\alpha_O}$.
	Then a basis of $\tau_{\frac{N}{2}}$ is given by
	\begin{equation}\label{eq:tau_N/2 basis}
		X_{O,\frac{N}{2}}=X_{\alpha_O}+\zeta_m^{-\frac{N}{2}}\Theta(X_{\alpha_O})+\zeta_m^{-N}\Theta^2(X_{\alpha_O})+\cdots+\zeta_m^{-(m-1)\frac{N}{2}}\Theta^{m-1}(X_{\alpha_O}).
	\end{equation}
	Here $\Theta^j(X_{\alpha_O})\in\fg_{\Theta^j(\alpha_O)}$.
	We claim that for any root $\alpha$ and $j$,
	$\Theta^j\alpha\neq-\alpha$.
	Otherwise $\Theta^{2j}\alpha=\alpha$,
	so that $m|2j$.
	Since $N$ is even, $(N,m)=1$, thus $m$ is odd, $m|j$,
	$\Theta^j\alpha=\alpha\neq-\alpha$.
	
	Now we define a partition $\Theta\backslash\Phi(\fg,\ft_Y)=A\sqcup B$
	of orbits as follows.
	Take any orbit $O$ and put it in $A$.
	For any $\alpha\in O$, from previous discussion we see $-\alpha\not\in O$,
	and we put the orbit of $-\alpha$ in B.
	Repeat this for the rest of orbits.
	The resulting partition satisfies $|A|=|B|=\frac{1}{2}\dim\tau_{\frac{N}{2}}$,
	and any $\alpha,\beta$ in the union of orbits in $A$ satisfy $\alpha\neq-\beta$.
	We let
	\begin{equation}\label{eq:Lagrangian}
		\fm=\bigoplus_{O\in A}\bC X_{O,\frac{N}{2}}.
	\end{equation}
	Thus $\dim\fm=\frac{1}{2}\dim\tau_{\frac{N}{2}}$.
	For any $O,O'\in A$,
	since $\alpha\neq-\beta$ for any $\alpha\in O$ and $\beta\in O'$, 
	we have $[X_{O,\frac{N}{2}},X_{O',\frac{N}{2}}]\in\tau\cap\fg_N=\tau_N$.
	Since the pairing of $\ft_Y$ with $\tau$ vanishes,
	$\fm$ is isotropic with respect to \eqref{eq:symp form}, 
	thus a Lagrangian of $\tau_{\frac{N}{2}}$.
	
    Let $P(\frac{N}{2})_{\fm}$
	be the preimage of $\fm$
	under $P(\frac{N}{2})/S(\frac{N}{2})P(\frac{N}{2}+1)\simeq\tau_{\frac{N}{2}}$.
	Let $J=S(1)P(\frac{N}{2})_{\fm}$
	and let $\tphi:J\rightarrow\bC^\times$ 
	be the composition of $\phi$ with projection to $S(1)$. 
	
	\subsubsection{}
	We call $(J,\tphi)$ a \emph{toral supercuspidal K-type}.
	It is the level structure associated to 
	Yu datum $(S\subset G;x;\nu;\phi,1)$.

	\subsection{Segal-Sugawara operators}\label{ss:SS operators}
	Recall $\hg=\fg(\!(t)\!)\oplus \bC\cdot\bone$ is the 
	affine Kac-Moody algebra at the  critical level associated to $\fg$. 
	The element $\bone$ is central, 
	and the Lie bracket is given by
	\begin{equation}\label{eq:hg_c Lie bracket}
	[A\otimes f(t), B\otimes g(t)]
	=[A,B]\otimes f(t)g(t)
	-\frac{1}{2}\mathrm{tr}(\ad_A\circ\ad_B)\Res(gdf)\cdot\bone.
	\end{equation}
	 
	Recall that $\Op_{\Lg}(D^\times)$ has coordinates $v_{i,j}$,
	$1\leq i\leq n, j\in\bZ$.
	Let $S_{i,j}\in\fZ$ be the image of $v_{ij}$ via the above isomorphism. 
	This gives a set of topological generators of $\fZ$
	called \emph{Segal-Sugawara operators}, 
	cf. \cite[\S4.3.1., \S4.3.2.]{FrenkelBook}.
	
	To write down $S_{i,j}$ more explicitly,
	we fix a topological basis of $\Ug$ as follows.
	First we fix a topological basis of $\fg(\!(t)\!)$.
	Recall the (completed) grading given by \eqref{eq:g((t)) to g((u))}.
	Each $\fg(\!(t)\!)_i\simeq\fg_iu^i=\ft_{Y,i}u^i\oplus\tau_iu^i$.
	Fix arbitrary basis $A_{i,k}$, $1\leq k\leq\dim\ft_{Y,i}$ of $\ft_{Y,i}$.
	For $\tau_i$,
	for each $\Theta$-orbit $O$ on $\Phi(\fg,\ft_Y)$, 
	pick a representative $\alpha_O\in O$.
	Fix a basis $X_{\alpha_O}\in\fg_{\alpha_O}$.
	As in the case $i=\frac{N}{2}$, 
	a basis of $\tau_i$ can be given by
	\begin{equation}\label{eq:tau_i basis}
		X_{O,i}=X_{\alpha_O}+\zeta_m^{-i}\Theta(X_{\alpha_O})+\zeta_m^{-2i}\Theta^2(X_{\alpha_O})+\cdots+\zeta_m^{-(m-1)i}\Theta^{m-1}(X_{\alpha_O}).
	\end{equation}
    In particular, $\dim\tau_i=|\Theta\backslash\Phi(\fg,\ft_Y)|=\frac{|\Phi|}{m}$ 
    is constant as seen in \cite[Theorem 4.2.(ii)]{Panyushev}.
	Moreover, 
	$\prod:=\{A_{i,k}u^i,X_{O,i}u^i\mid i\in\bZ, 1\leq k\leq\dim\ft_{Y,i},
	O\in\Theta\backslash\Phi(\fg,\ft_Y)\}$
	gives a basis of $\fg(\!(t)\!)$.
	We put a total order on $\prod$ by requiring
	\[
	\cdots<\fg(\!(t)\!)_{i-1}<\fg(\!(t)\!)_i<\fg(\!(t)\!)_{i+1}<\cdots,
	\]
	$\tau_{Y,i}u^i<\ft_{Y,i}u^i$,
	and fix arbitrary total order on $\{A_{i,k}u^i\}_k$
	and $\{X_{O,i}u^i\}_O$.
	By PBW theorem, a topological basis of $\Ug$ can be given by 
	\begin{equation}\label{eq:Ug PBW basis}
	\{X_{a_1}\otimes X_{a_2}\otimes\cdots\otimes X_{a_k}\mid\
	k\geq 0,\ X_{a_i}\in\prod,\
	X_{a_1}\leq X_{a_2}\leq\cdots\leq X_{a_k}\}.
	\end{equation}
	
	\begin{lem}\label{l:SS operator}
		The operator $S_{i,j}\in\fZ$ 
		acts on any vector in a smooth $\hg$-module
		as a finite sum of tensors
		\[
		X_{a_1}\otimes X_{a_2}\otimes\cdots\otimes X_{a_k}
		\] 
		satisfying
		\[
		\mathrm{(i)}\ X_{a_r}\in\fg_{a_r}u^{a_r}\cap\prod,\qquad
		\mathrm{(ii)}\ k\leq d_i,\qquad
		\mathrm{(iii)}\ \sum_{r=1}^k a_r=m(j-d_i+1).
		\]
	\end{lem}
	\begin{proof}
		As in \cite[Lemma 54]{KXY},
		$S_{i,j}$ acts as a finite linear combination of tensors
		\[
		E_{\beta_1}t^{n_1}\otimes\cdots\otimes E_{\beta_k}t^{n_k}
		\]
		where $k\leq d_i$, 
		$E_{\beta_r}$ is a basis of $\fg_{\beta_r}$ with respect to $\ft$,
		$\sum_{r=1}^kn_r=j-d_i+1$, $\sum_{r=1}^k\beta_r=0$.
		Rewriting each $E_{\beta_r}t^{n_r}$
		in terms of basis $\prod$,
		we obtain a linear combination of tensors satisfying (i), (ii) of the lemma.
		Moreover, note that the isomorphism $\psi$ replaces $t=u^m$
		and applies $\Ad_{\lambda(u)}$.
		We obtain
		\[
		\sum_{r=1}^k a_r
		=m\sum_{r=1}^k n_r+\sum_{r=1}^k\langle\lambda,\beta_r\rangle
		=m(j-d_i+1).
		\]
	\end{proof}

    \subsection{Central characters}
    Recall we defined a subspace $\fj'\subset\fj$ \eqref{eq:j'}.
    Explicitly, in the current situation it is given by
    \[
    \fj'=
    \begin{cases}
    	\bigoplus_{i=\frac{N+1}{2}}^N\tau_iu^i\oplus\fp(N+1),\hspace{1.75cm} 
    	N\text{ odd,}\\
    	\bigoplus_{i=\frac{N}{2}+1}^N\tau_iu^i
    	\oplus\fm u^{\frac{N}{2}}
    	\oplus\fp(N+1),\quad 
    	N\text{ even.}
    \end{cases}
    \]
    
    \begin{lem}\label{l:Lagrangian}
    	The subspace $\fj'\subset\fj$ is an ideal, 
    	so that $\fj'=\fj^+$.
    \end{lem}
    \begin{proof}
    	We need to show $[\fj,\fj']\subset\fj'$.
    	Observe $\fj=\fj'+\fs(1)$.
    	It is easy to see that the adjoint action of
    	$\fs(1)$ preserves $\fj'$.
    	To show $[\fj',\fj']\subset\fj'$,
    	from the depth consideration
    	it suffices to show 
    	$[\fm,\fm]\subset\tau_N$
    	when $N$ is even.
    	This follows from the construction of $\fm$ \eqref{eq:Lagrangian}. 
    \end{proof}

    Therefore,
    the quotient \eqref{eq:bj} becomes
    $\bj=\fj/\fj^+\simeq\fs(1)/\fs(N+1)
    \simeq\bigoplus_{1\leq i\leq N}\ft_{Y,i}u^i$.
    
    As in \S\ref{sss:central support},
    $\fj^+\subset\fj\subset\fg[\![t]\!]$,
    $\fj^++\bC\bone\subset\fj+\bC\bone$ are subalgebras of $\hg$.
    Consider the $\hg$-module $\Vac_{\fj^+}$
    defined in \eqref{eq:Vac_j+}.

    We have
    \[
    \End\Vac_{\fj^+}
    \simeq\Hom_{\fj}(U(\bj),\Ug\otimes_{U(\fj+\bC\bone)}U(\bj))
    \simeq\Vac_{\fj^+}^{J^+}.
    \]
    
    Since $\bj\simeq\fs(1)/\fs(N+1)$ is abelian,
    $U(\bj)\hookrightarrow\End\Vac_{\fj^+}$.
    Recall $\fZ_{\fj^+}=\mathrm{Im}(\fZ\rightarrow\End\Vac_{\fj^+})$.
    Define
    \begin{equation}
    	A':=\bC[S_{i,j}\mid (i,j)\in\tilde{A}_\nu]
    \end{equation}
    where $\tilde{A}_\nu$ is as in \eqref{eq:A_nu tA_nu}.
    
    \begin{prop}\label{p:local quant diagram}
    	The image of $A'\to \fZ_{\fj^+}$ is contained in $U(\bj)$, 
    	so that we have commutative diagram
    	\begin{equation}\label{eq:local quant diagram}
    		\begin{tikzcd}
    			A' \arrow[r] \arrow[d] &\fZ_{\fj^+} \arrow[d,hook]\\
    			U(\overline{\fj}) \arrow[r,hook] &\End(\Vac_{\fj^+}) \arrow[r,hook] &\Vac_{\fj^+}
    		\end{tikzcd}
    	\end{equation}
    \end{prop}
    \begin{proof}
    	The action of central operators $S_{i,j}$
    	on $\Vac_{\fj^+}$
    	is determined by their action on $1\in\Vac_{\fj^+}$.
    	By Lemma \ref{l:SS operator},
    	$S_{i,j}\cdot 1$
    	can be written as a finite linear combination of tensors
    	\[
    	X_{a_1}\otimes X_{a_2}\otimes\cdots\otimes X_{a_k},
    	\]
    	where $k\leq d_i$, $X_{a_r}\in\fg_{a_r}u^{a_r}\cap\prod$, $\sum_{r=1}^ka_r=m(j-d_i+1)$,
    	$a_1\leq a_2\leq\cdots\leq a_k$.
    	
    	For $(i,j)\in\tilde{A}_\nu$, 
    	\[
    	(d_i-1)N+1\leq\sum_{r=1}^k a_r=m(j-d_i+1)=(d_i-1)N-\ell_{i,j}\leq d_iN.
    	\]
    	
    	Note that for $a\geq N+1$, $\fg_au^a\subset\fp(N+1)\subset\fj^+$
    	annihilates $1\in\Vac_{\fj^+}$.
    	Thus if $a_k\geq N+1$, the tensor has zero coefficient in $S_{i,j}\cdot 1$.
    	We can assume $a_1\leq\cdots\leq a_k\leq N$.
    	Then
    	\[
    	a_1\geq\sum_{r=1}^ka_r-(k-1)N\geq(d_i-1)N+1-(d_i-1)N=1.
    	\] 
    	
    	Suppose $k\leq d_i-1$.
    	Then
    	\[
    	a_1\geq\sum_{r=1}^ka_r-(k-1)N\geq(d_i-1)N+1-(d_i-2)N
    	=N+1,
    	\]
    	contradicting the assumption $a_1\leq a_k\leq N$.
    	Hence $k=d_i$.
    	
    	Assume 
    	\[
    	1\leq a_1\leq a_2\leq\cdots\leq a_p\leq\frac{N}{2}
    	<\frac{N+1}{2}\leq a_{p+1}\leq\cdots\leq a_{d_i}.
    	\]
    	
    	If $p=0$, then all $a_r>\frac{N}{2}$, 
    	$X_{a_r}\in\fg_{a_r}u^{a_r}\subset\fj^+$.
    	The tensor is contained in $U(\bj)$.
    	
    	If $p>0$, we have inequality
    	\[
    	(d_i-1)N+1\leq\sum_{r=1}^{d_i}a_r\leq p\lfloor\frac{N}{2}\rfloor+(d_i-p)N.
    	\]
    	We get
    	\[
    	p\leq\frac{N-1}{N-\lfloor\frac{N}{2}\rfloor}
    	=\begin{cases}
    		2\frac{N-1}{N+1},\quad N\text{ odd}\\
    		2\frac{N-1}{N},\quad N\text{ even}
    	\end{cases}
        <2.
    	\]
    	We obtain $p=1$ and $a_2,...,a_{d_i}\geq\frac{N+1}{2}$,
    	$X_{a_r}\in\fj$ for $r\geq 2$.
    	
    	If $X_{a_{d_i}}\in\tau_{a_{d_i}}u^{a_{d_i}}$,
    	then the tensor acts as $0$ on $1$.
    	Thus we can assume $X_{a_{d_i}}\in\ft_{Y,a_{d_i}}u^{a_{d_i}}$.
    	We show $X_{a_r}\in\ft_{Y,a_r}u^{a_r}$ for all $2\leq r\leq d_i$
    	by induction.
    	Given $2\leq q<d_i$,
    	suppose $X_{a_r}\in\ft_{Y,a_r}u^{a_r}$ for all $q<r\leq d_i$.
    	If $X_{a_q}\in\tau_{a_q}u^{a_q}$,
    	we claim the tensor acts as $0$ on $1$.
    	In fact,
    	we can move $X_{a_q}$ to the rightmost term,
    	after which the tensor acts as $0$ on $1$.
    	In this process, 
    	two kinds of shorter tensors are produced by \eqref{eq:hg_c Lie bracket}:
    	replacing two adjacent terms in the tensor 
    	either by a multiple of the unit $\bone$, which gives a length $d_i-2$ tensor,
    	or by the Lie bracket of the two terms, which gives a length $d_i-1$ tensor.
    	However, since the Killing form between $\ft_Y$ and $\tau$ vanishes,
    	only the second case can contribute nonzero tensors,
    	where the sum $\sum_{r=1}^k a_r$ is the same as that of the original tensor.
    	We can repeat this process to rewrite these shorter tensors
    	as a linear combination of basis in $\prod$.
    	From the procedure, these tensors all satisfy the conditions on
    	$X_{a_1}\otimes X_{a_2}\otimes\cdots\otimes X_{a_k}$,
    	and moreover $k<d_i$.
    	From previous discussion, we know such shorter tensors all act as $0$ on $1$.
    	We conclude that $X_{a_r}\in\ft_{Y,a_r}u^{a_r}$ for all $2\leq r\leq d_i$.
    	
    	It remains to show $X_{a_1}\in\ft_{Y,a_1}u^{a_1}$.
    	Suppose $X_{a_1}\in\tau_{a_1}u^{a_1}$ instead.
    	Note that the central element $S_{i,j}$
    	commutes with (the preimage in $\Ug$ of) $\ft_{Y,d}u^d$ for all $d$.
    	Let $\Ug_{\leq e}$ be the subspace of tensors in $\Ug$
    	with length $\leq e$,
    	and $\gr_e\Ug:=\Ug_{\leq e}/\Ug_{\leq e-1}$.
    	Note that $\gr_e\Ug$ has a basis by images of tensors 
    	as in Lemma \ref{l:SS operator}
    	with $k=e$.
    	Denote by $\gr_e^d\Ug$
    	the subspace of $\gr_e\Ug$
    	generated by tensors satisfying $\sum_{r=1}^ka_r=d$,
    	and denote its basis as above by $C(e,d)$.
    	The action of $\ft_{Y,d}u^d$ sends 
    	$\gr_{d_i}^{m(j-d_i+1)}\Ug$ to $\gr_{d_i}^{m(j-d_i+1+d)}\Ug$.
    	Moreover, this action sends the subspace of $\gr_{d_i}^{m(j-d_i+1)}\Ug$
        spanned by
        \[
        \tau_{a_1}u^{a_1}\otimes X_{a_2}\otimes\cdots\otimes X_{a_{d_i}}
        \]
        into the subspace spanned by
        \[
        \tau_{a_1+d}u^{a_1+d}\otimes X_{a_2}\otimes\cdots\otimes X_{a_{d_i}}.
        \]
        It also sends those basis in $C(d_i,m(j-d_i+1))$ not contained in 
        $\tau_{a_1}u^{a_1}\otimes X_{a_2}\otimes\cdots\otimes X_{a_{d_i}}$
        into the subspace generated by those basis in $C(d_i,m(j-d_i+1+d))$ 
        not contained in
        $\tau_{a_1+d}u^{a_1+d}\otimes X_{a_2}\otimes\cdots\otimes X_{a_{d_i}}$.
        
        Let $Zu^{a_1}\otimes X_{a_2}\otimes\cdots\otimes X_{a_{d_i}}$,
        $Z\in\tau_{a_1}$
        be the component of the image of $S_{i,j}$ in $\gr_{d_i}^{m(j-d_i+1)}\Ug$
        inside $\tau_{a_1}u^{a_1}\otimes X_{a_2}\otimes\cdots\otimes X_{a_{d_i}}$
        in terms of the fixed basis.
        From the above discussion together with 
        $[\ft_{Y,d}u^d,S_{i,j}]=0$, 
        we obtain
        \[
        [\ft_{Y,d}u^d,Zu^{a_1}]=0,\quad \forall d.
        \]
        Thus $[\ft_{Y,d},Z]=0,\forall d$, so that $[\ft_Y,Z]=0$.
        Since $Z\in\tau$, we conclude $Z=0$.
        This completes the proof.
    \end{proof}

    \subsection{Local Hitchin image at level $\fj^+$}\label{ss:local Hit}\mbox{}
    
    We prepare some results on local Hitchin images demanded 
    for a more detailed study of \eqref{eq:local quant diagram}.
    
    \subsubsection{}
    We first fix a Kostant section of $\fg$ as follows.
    Recall in \S\ref{ss:SS operators}
    we have fixed a principal $\mathfrak{sl}_2$-triple $\{p_{-1},2\crho_{\LG},p_1\}$
    in $\Lg$ together with a basis of Kostant section
    $p_{-1}+\Lg^{p_1}=p_{-1}+\sum_{i=1}^n\bC p_i$.
    Consider the dual principal $\mathfrak{sl}_2$-triple
    $\{p_{-1},2\crho_G,p_1\}$ in $\fg$ where we abuse notation.
    We have isomorphisms
    \[
    p_{-1}+\Lg^{p_1}\simeq\Lg/\!\!/\LG\simeq\ct/\!\!/W
    \simeq\ft^*/\!\!/W\simeq\ft/\!\!/W
    \simeq\fg/\!\!/G
    \simeq p_{-1}+\fg^{p_1}
    \]
    where $\ft\simeq\ft^*$ is induced from Killing form.
    The above morphisms are compatible with $\bGm$-action,
    where on $p_{-1}+\Lg^{p_1}$ the action is given by $\crho_{\LG}$ on $\Lg^{p_1}$
    and on $p_{-1}+\fg^{p_1}$ the action is given by $\crho_G$ on $\fg^{p_1}$.
    The weight subspaces of $\Lg^{p_1}$ are isomorphic to those of $\fg^{p_1}$,
    and the homogeneous basis $p_i$ of $\Lg^{p_1}$
    are sent to a homogeneous basis of $\fg^{p_1}$
    which we still denote by $p_i$.
    
    Consider the local Hitchin map 
    \begin{equation}\label{eq:local Hithin map}
    	h^{cl}:\fg^*(\!(t)\!)\td t\simeq\fg(\!(t)\!)\td t
    	\ra\Hit(D^\times)\simeq\bigoplus_{i=1}^n\bC(\!(t)\!)(\td t)^{d_i},
    \end{equation}
    where the last isomorphism is via invariant polynomials defined by 
    Kostant section $p_{-1}+\sum_i\bC p_i\simeq\fg/\!/G$. 
    Let $h_{i,j}$ be the coordinate of $\bC t^{-j-1}(\td t)^{d_i}$.
    Under the isomorphism $\gr\fZ\simeq\Fun\Hit(D^\times)$,
    the image of $S_{i,j}$ is $h_{i,j}$.

    \subsubsection{}
    Let $\fj^{+,\perp}\subset\fg^*(\!(t)\!)\td t$ 
    be the $\cO$-lattice that is orthogonal to $\fj^+$,
    i.e.
    \[
    \fj^{+,\perp}=
    \{X\td t\in\fg^*(\!(t)\!)\td t\mid
    \Res\langle X,Y\rangle\td t=0,\ \forall\ Y\in\fj^+\}.
    \]
    With respect to the filtration induced from $\Ug$,
    $\gr\Vac_{\fj^+}\simeq\Fun\fj^{+,\perp}$.
    Identifying $\fg^*\simeq\fg$ using Killing form
    and applying $\psi$,  we have
    \begin{equation}\label{eq:j^+,perp}
    	\fj^{+,\perp}\simeq
    	\begin{cases}
    		(\bigoplus_{-N\leq i\leq-\frac{N+1}{2}}\ft_{Y,i}u^i
    		\oplus\prod_{i\geq-\frac{N-1}{2}}\fg_iu^i)\frac{\td u}{u},\hspace{2.7cm}
    		N\text{ odd,}\\[0.2cm]
    		
    		(\bigoplus_{-N\leq i\leq-\frac{N}{2}-1}\ft_{Y,i}u^i
    		\oplus\fm^\perp u^{-\frac{N}{2}}\oplus\prod_{i\geq-\frac{N}{2}+1}\fg_iu^i)\frac{\td u}{u},\qquad
    		N\text{ even,}
    	\end{cases}
    \end{equation}
    where $\ft_{Y,-\frac{N}{2}}\subset\fm^\perp\subset\fg_{-\frac{N}{2}}$ 
    is the orthogonal complement of $\fm$
    with respect to Killing form.
    
    Similarly define lattice $\fj^\perp$ with isomorphism
    \begin{equation}
    	\fj^\perp\simeq
    	\begin{cases}
    		(\bigoplus_{-\frac{N-1}{2}\leq i\leq-1}\tau_iu^i\oplus
    		\prod_{i\geq0}\fg_iu^i)\frac{\td u}{u},\hspace{4.1cm}
    		N\text{ odd,}\\[0.2cm]
    		
    		((\tau_{-\frac{N}{2}}\cap\fm^\perp)u^{-\frac{N}{2}}\oplus
    		\bigoplus_{-\frac{N}{2}+1\leq i\leq-1}\tau_iu^i\oplus
    		\prod_{i\geq0}\fg_iu^i)\frac{\td u}{u},\qquad
    		N\text{ even.}
    	\end{cases}
    \end{equation}
    Note that 
    $\fj^{+,\perp}/\fj^\perp
    \simeq(\bigoplus_{-N\leq i\leq -1}\ft_{Y,i}u^i)\frac{\td u}{u}
    \simeq\bj^*$.
    
    Denote by $\Hit(D)_{\fj^+}$ the closure of the image
    $h^{cl}(\fj^{+,\perp})$,
    then $\gr\fZ_{\fj^+}\simeq\Fun\Hit(D)_{\fj^+}$.
    
    Define
    \begin{equation}
    	\Hit(D)_{\fj^+}'
    	:=\bigoplus_{i=1}^nt^{-d_i-\lfloor\frac{d_iN}{m}\rfloor}\bC[\![t]\!](\td t)^{d_i}.
    \end{equation}
    
    \begin{prop}\label{p:local Hitchin image}
    	$\Hit(D)_{\fj^+}=\Hit(D)_{\fj^+}'$.
    \end{prop}
    \begin{proof}
    	Since $\fj^+\supset\fp(N+1)$, $\fj^{+,\perp}\subset\fp(N+1)^\perp$.
    	By \cite[Proposition 10]{Zhu}, 
    	\begin{equation}
    		h^{cl}(\fp(N+1)^\perp)\subset\Hit(D)_{\fj^+}'.
    	\end{equation}
    	It remains to show the closure of $h^{cl}(\fj^{+,\perp})$ 
    	contains $\Hit(D)_{\fj^+}'$.
    	
    	Since $\Theta$ is $\bZ$-regular and elliptic \cite[\S2]{RLYG},
    	we know by \cite[Proposition 8]{RLYG} that
    	$\Ad_{\crho(\zeta_m)}$ and $\Ad_{\crho(\zeta_m^N)}$
    	are conjugate,
    	so $\Theta=\Ad_{\lambda(\zeta_m)}$
    	and $\Theta^N=\Ad_{\lambda(\zeta_m^N)}$
    	are also conjugate.
    	Thus we can find a $\Theta$-adapted principal $\mathfrak{sl}_2$-triple 
    	$\{e,h,f\}$ with $e\in\fg_N,f\in\fg_{-N},h\in\fg_0$.
    	Let $g_0\in G$ be the unique element up to center
    	that conjugates $\{p_1,2\crho,p_{-1}\}$ to $\{e,h,f\}$.
    	Denote $p_i'=g_0\cdot p_i\in\fg_{(d_i-1)N}$.
    	
    	An element
    	$\sum_{i=1}^n\sum_{j\leq d_i-1+\lfloor\frac{d_iN}{m}\rfloor}h_{i,j}t^{-j-1}(\td t)^{d_i}
    	=\sum_{i=1}^nt^{-\lfloor\frac{d_iN}{m}\rfloor}c_i(t)(\frac{\td t}{t})^{d_i}
    	\in\Hit(D)_{\fj^+}'$,
    	rewritten in terms of $u$,
    	is the image of
    	\[
    	(p_{-1}+\sum_{i=1}^n m^{d_i}u^{-d_i-m\lfloor\frac{d_iN}{m}\rfloor}c_i(u^m)p_i)\td u.
    	\]
    	
    	If $N$ is odd, applying conjugation by $g_0$, we get
    	\[
    	(f+\sum_{i=1}^n m^{d_i}u^{-d_i-m\lfloor\frac{d_iN}{m}\rfloor}c_i(u^m)p_i')\td u.
    	\]
    	Applying conjugation by $h(u^{\frac{N+1}{2}})$, 
    	we get
    	\begin{equation}\label{eq:Kostant section conn}
    	u^{-N}(f+
    	\sum_{i=1}^n m^{d_i}u^{d_iN-m\lfloor\frac{d_iN}{m}\rfloor}c_i(u^m)p_i')\frac{\td u}{u}
    	\in\prod_{i\geq-N}\fg_iu^i\frac{\td u}{u},
    	\end{equation}
    	where $c_i(t)\in\bC[\![t]\!]$ with
    	$c_i(0)=h_{i,d_i-1+\lfloor\frac{d_iN}{m}\rfloor}$.
    	
    	If $N$ is even, we first apply conjugation by 
    	$\crho(u)\in G_{\mathrm{ad}}(\!(u)\!)$ to 
    	$(p_{-1}+\sum_{i=1}^n m^{d_i}u^{-d_i-m\lfloor\frac{d_iN}{m}\rfloor}c_i(u^m)p_i)\td u$,
    	then apply $h(u^{\frac{N}{2}})$,
    	we arrive at \eqref{eq:Kostant section conn} as well. 
    	
    	The leading term of \eqref{eq:Kostant section conn} is
    	\begin{equation}\label{eq:Z}
    		Z=f+\sum_{m|d_iN}m^{d_i}c_i(0)p_i'\in\fg_{-N}.
    	\end{equation}
    	
    	It suffices to show that for generic $c_i(0)$,
    	the form \eqref{eq:Kostant section conn}
    	can be conjugated into $\fj^{+,\perp}$.
    	Since $\fg_{-N}$ contains regular semisimple element $Y$,
    	its regular semisimple locus is open and nonempty,
    	so that we can assume $Z$ is regular semisimple.
    	Observe that $\ft_{Y,-N}$ is a Cartan subspace of $\fg_{-N}$.
    	Thus $Z$ can be conjugated into $\ft_{Y,-N}$.
    	Without loss of generality, 
    	we assume $Z\in\ft_{Y,-N}$.
    	Applying conjugation by elements of the form 
    	$\exp(\tau_ku^k)$ for $k\geq 1$
    	in the same way as in the first step of the proof of 
    	Lemma \ref{l:isoclinic can remove regular part},
    	we conjugate \eqref{eq:Kostant section conn}
    	into $\prod_{i\geq-N}\ft_{Y,i}u^i\frac{\td u}{u}\subset\fj^{+,\perp}$.
    	This completes the proof.
    \end{proof}

    Denote 
    \begin{equation}\label{eq:Hit(D)_bj}
    	\Hit(D)_{\bj}:=\Spec\bC[h_{i,j}\mid-N\leq\ell_{i,j}\leq-1].
    \end{equation}
    Note that $\gr A\simeq\Fun\Hit(D)_{\bj}$.
    We have projection $\Hit(D)_{\fj^+}\twoheadrightarrow\Hit(D)_{\bj}$
    using the basis $h_{i,j}$.
    
    \begin{prop}\label{p:local classical diagram}
    	\begin{itemize}
    		\item [(i)]
    		The composition
    		$\fj^{+,\perp}\rightarrow\Hit(D)_{\fj^+}\rightarrow\Hit(D)_{\bj}$
    		factors through the quotient
    		\[
    		\fj^{+,\perp}\rightarrow\fj^{+,\perp}/\fj^\perp\simeq\bj^*.
    		\]
    		The associated graded of the commutative diagram 
    		\eqref{eq:local quant diagram}
    		is
    		\begin{equation}\label{eq:local classical diagram}
    			\begin{tikzcd}
    				\fj^{+,\perp} \arrow[r] \arrow[d] &\bj^* \arrow[d]\\
    				\Hit(D)_{\fj^+} \arrow[r] &\Hit(D)_{\bj}
    			\end{tikzcd}
    		\end{equation}
    	
    	    \item [(ii)]
    	    Let $\bj^{*,\rs}\subset\bj^*$ be the open locus consisting of elements
    	    whose components in $\ft_{Y,N}^*\simeq\ft_{Y,-N}$ are regular semisimple.
    	    Let $\Hit(D)_{\bj}^\rs\subset\Hit(D)_{\bj}$ be the open locus
    	    whose coefficients $h_{i,j}$ satisfy that 
    	    the vector \eqref{eq:Z}
    	    $Z=f+\sum_{\ell_{i,j}=-N}m^{d_i}h_{i,j}p_i'$ 
    	    is regular semisimple.
    	    Then $h^{cl}(\bj^{*,\rs})=\Hit(D)_{\bj}^\rs$.
    	    
    	    \item [(iii)]
    	    The restriction 
    	    $h^{cl}:\bj^{*,\rs}\rightarrow\Hit(D)_{\bj}^\rs$
    	    is \'etale
    	    with finite fibers that are isomorphic to the 
    	    free orbits of 
    	    the regular semisimple coefficient in $\ft_{Y,-N}$
    	    under the action of the finite little Weyl group $W_0:=N_{G_0}(\ft_{Y,-N})/C_{G_0}(\ft_{Y,-N})$.
    	\end{itemize}
    \end{prop}
    \begin{proof}
    	(i):
    	The forms in $\fj^{+,\perp}$ whose leading terms 
    	in $\ft_{Y,-N}u^{-N}\frac{\td u}{u}$
    	are regular semisimple
    	make up an open subspace of $\fj^{+,\perp}$.
    	For such a form $\omega$, 
    	it can be conjugated into Kostant section \eqref{eq:Kostant section conn},
    	where the leading $N$ coefficients 
    	$\prod_{-N\leq i\leq-1}\fg_iu^i\frac{\td u}{u}$ 
    	of \eqref{eq:Kostant section conn}
    	are exactly given by the projection of $\omega$ to $\bj\simeq\bj^*$.
    	This concludes the desired factorization
    	and the commutative diagram.
    	
    	(ii):
    	This is clear from the discussion in part (i).
    	
    	(iii):
    	Let $\phi_{cl}=(h_{i,j})_{-N\leq\ell_{i,j}\leq-1}\in\Hit(D)_{\bj}^\rs$.
    	As in part (i), 
    	the fiber $h^{cl,-1}(\phi_{cl})$
    	consists of elements $\tphi=\sum_{i=1}^N\tphi_iu^i\in\bj$
    	that are conjugate to the Kostant section \eqref{eq:Kostant section conn}
    	of $\phi_{cl}$.
    	By the same proof as Proposition \ref{p:oper fiber over Loc}.(iii),
    	elements in $h^{cl,-1}(\phi_{cl})$ are determined by their components in $\ft_{Y,-N}$.
    	Thus $h^{cl,-1}(\phi_{cl})$ is isomorphic to the elements in $\ft_{Y,-N}$
    	that are conjugated to the regular semisimple leading term $Z$ in part (ii).
    	By \cite[Theorem 4.6.(i)]{Panyushev} and \cite[Theorem 2]{Vinberg},
    	$h^{cl,-1}(\phi_{cl})$ is isomorphic to $W_0\cdot Z$.
    	Since the action of $W_0$ on $\ft_{Y,-N}$
    	is given by Weyl group \cite[Theorem 4.6.(ii)]{Panyushev},
    	it acts freely on regular semisimple elements.
    	Note that $W_0$ is finite by \cite[Proposition 8]{Vinberg}.
    	
    	We conclude that $h^{cl}:\bj^{*,\rs}\rightarrow\Hit(D)_{\bj}^\rs$
    	is isomorphic via Kostant section with a free quotient by $W_0$.
    \end{proof}

    \begin{cor}\label{c:A intersection}
    	\begin{itemize}
    		\item [(i)]
    		The maps $A'\rightarrow\fZ_{\fj^+}$ and $A'\rightarrow U(\bj)$
    		are injective.
    		
    		\item [(ii)]
    		The algebra $A'$ is isomorphic to 
    		the intersection of $U(\bj)$ and $\fZ_{\fj^+}$
    		in $\End\Vac_{\fj^+}$,
    		i.e. $A'\simeq A$ in \eqref{eq:A intersection}.
    		
    		\item [(iii)]
    		There is an isomorphism $A\simeq\Fun\Hit(D)_{\bj}$
    		that identifies $\Spec U(\bj)\rightarrow\Spec A$
    		with $\bj\rightarrow\Hit(D)_{\bj}$.
    		
    		\item [(iv)]
    		Let $(\Spec A)^\rs\subset\Spec A$ be the open subscheme
    		of elements whose coefficients $S_{i,j}$
    		correspond under Feigin-Frenkel isomorphism
    		to those coefficients $v_{i,j}$
    		such that $p_{-1}+\sum_{\ell_{i,j}=-N}v_{i,j}p_i$
    		is regular semisimple.
    		Then $(\Spec A)^\rs\simeq\Hit(D)_{\bj}^\rs$
    		under the isomorphism in (iii).
    	\end{itemize}
    \end{cor}
    \begin{proof}
    	(i):
    	The injectivity of $A'\rightarrow\fZ_{\fj^+}$ and $A'\rightarrow U(\bj)$
    	follows from the surjectivity 
    	of the $\Spec$ of the associated graded algebras,
    	i.e. $\Hit(D)_{\fj^+}\rightarrow\Hit(D)_{\bj}$
    	and $\bj\rightarrow\Hit(D)_{\bj}$.
    	
    	(ii):
        It suffices to prove for associated graded algebras,
        i.e. any regular function on $\fj^{+,\perp}$
        that factors through both $\Hit(D)_{\fj^+}$ and $\bj$
        must factor through $\Hit(D)_{\bj}$.
        Moreover, it suffices to prove over the open locus of $\fj^{+,\perp}$
        with regular semisimple leading term,
        which we denote by $\fj_{\rs}^{+,\perp}$.
        
        Indeed, let $f$ be a regular function on $\fj_{\rs}^{+,\perp}$
        that factors through the local Hitchin map
        and is constant over cosets of $\fj^\perp$.
        The local Hitchin map on $\fj_{\rs}^{+,\perp}$ 
        can be computed by
        first conjugating the form into $\prod_{i\geq-N}\ft_{Y,i}u^i\frac{\td u}{u}$,
        next conjugating it into Kostant section \eqref{eq:Kostant section conn},
        and then taking the coefficients of Kostant section 
        to be the values of $h_{i,j}$.
        By $\fj^\perp$-invariance, 
        the values of $f$ depend on the first $N$ leading terms 
        of the Kostant section.
        Thus $f$ factors through $\Hit(D)_{\bj}$. 
        
        (iii):
        We just need to show that $A$ is isomorphic to its associated graded algebra.
        In fact, we have seen in the proof of Proposition \ref{p:local quant diagram}
        that the action of $S_{i,j}$ on $1\in\Vac_{\fj^+}$
        factors through the projection to the length $d_i$ tensors in $S_{i,j}$.
        Since the action gives an embedding of $A$ by (i),
        we get $A\simeq\gr A$ as algebras.
        
        (iv):
        Recall the Kostant sections in $\Lg$ and $\fg$
        are identified via $\Lg/\!\!/\LG\simeq\fg^*/\!\!/G\simeq\fg/\!\!/G$.
        Also note that the tuples $(f,p_i')$ and $(p_{-1},p_i)$ are conjugated.
        The statement follows from Proposition \ref{p:local classical diagram}.(ii).
    \end{proof}

    \subsection{Central support of $\Vac_{\fj^+}$}\mbox{}
    
    Let $\Op_{\Lg}(D)_{\leq\nu}$ be the space of opers on $D^\times$
    with slope bounded above by $\nu=\frac{N}{m}$.
    By the oper slope formula \eqref{eq:oper slope},
    \[
    \Fun\Op_{\Lg}(D)_{\leq\nu}=\bC[S_{i,j}\mid \ell_{i,j}\geq-N].
    \]
    Denote $\Op(D)_{\fj^+}=\Spec\fZ_{\fj^+}\hookrightarrow\Op_{\Lg}(D^\times)$.
    
    \begin{lem}\label{l:Op_j+=Op_=<nu}
    	$\Op_{\Lg}(D)_{\fj^+}=\Op_{\Lg}(D)_{\leq\nu}$.
    \end{lem}
    \begin{proof}
    	For any $S_{i,j}$ with $\ell_{i,j}<-N$,
    	write $S_{i,j}$ as linear combination of tensors as in 
    	the proof of Proposition \ref{p:local quant diagram}.
    	We have
    	\[
    	\sum_{r=1}^ka_r=m(j-d_i+1)=(d_i-1)N-\ell_{i,j}>d_iN.
    	\]
    	Since $k\leq d_i$, $a_1\leq\cdots\leq a_k$,
    	we have $a_k\geq\frac{\sum_{r=1}^ka_r}{d_i}>N$.
    	Therefore $S_{i,j}\cdot 1=0$ in $\Vac_{\fj^+}$.
    	The action of $\fZ$ on $\Vac_{\fj^+}$ factors as
    	\[
    	\fZ\rightarrow\Fun\Op_{\Lg}(D)_{\leq\nu}
    	\twoheadrightarrow\fZ_{\fj^+}\hookrightarrow\End\Vac_{\fj^+}.
    	\]
    	
    	By Proposition \ref{p:local Hitchin image},
    	the associated graded of the middle surjection in the above
    	is an isomorphism:
    	\[
    	\gr\Fun\Op_{\Lg}(D)_{\leq\nu}=\Fun\Hit(D)_{\fj^+}'
    	=\Fun\Hit(D)_{\fj^+}=\gr\fZ_{\fj^+}.
    	\]
    	The lemma follows.
    \end{proof}

	\section{Isoclinic local geometric Langlands}\label{s:isoclinic local Langlands}
	We construct a correspondence between 
	irreducible isoclinic formal connections
	and 
	toral supercuspidal representations
	in the sense of Frenkel-Gaitsgory \cite{FGLocal}.
	
    \subsection{Matching canonical forms with $K$-types}\mbox{}
    
    We first match the irreducible isoclinic formal connections
    with pairs $(J,\tphi)$.
    
    \subsubsection{}
    First, the matching of the 
    leading terms of the isoclinic formal connection 
    and $\tphi$
    are given by \cite[Lemma 3]{CYTheta}.
    Explicitly,
    recall $X\in\Lg_{-N}$ and $Y\in\fg_N^*$ 
    are regular semisimple elements.
    Denote by $\LG_0$ (resp. $G_0$)
    the subgroup of $\LG$(resp. $G$) with Lie algebra $\Lg_0$(resp. $\fg_0$).
    \begin{lem}\label{l:isoclinic leading term match}
    	There exists an isomorphism 
    	\begin{equation}
    		\fg_N^*/\!\!/G_0\simeq\Lg_{-N}/\!\!/\LG_0
    	\end{equation}
        that restricts to a bijection between regular semisimple conjugacy classes
        \begin{equation}\label{eq:match rs leading terms}
        	\fg_N^{*,\rs}/G_0\simeq\Lg_{-N}^{\rs}/\LG_0.
        \end{equation}
    \end{lem}
    \begin{proof}
    	Note that the longest element of Weyl group $w_0\in W$
    	conjugates $\crho$ to $-\crho$.
    	Thus $\theta=\Ad_{\crho(\zeta_m)}$ and $\theta^{-1}$
    	are conjugated by $\Ad_{w_0}$.
    	This induces an isomorphism 
    	$\Lg_{-N}\simeq\Lg_N$.
    	The isomorphism follows from 
    	replacing $\zeta_m$ with $\zeta_m^N$ in \cite[Lemma 3]{CYTheta},
    	where regular semisimple eigenvectors are the same as stable vectors
    	by Lemma 9 of \emph{loc. cit.}.
    	Here although the grading on $\fg$ is defined by $\Theta=\Ad_{\lambda(\zeta_m)}$
    	rather than $\Ad_{\crho(\zeta_m)}$ in \emph{loc. cit.},
    	since $\lambda$ is conjugated to $\crho$,
    	the $\zeta_m^N$-eigenspaces of 
    	the gradings by $\Theta$ and $\Ad_{\crho(\zeta_m)}$ 
    	are conjugated.
    \end{proof}
    
    \subsubsection{}\label{sss:matching}
    Fix $X\in\Lg_{-N}^{\rs}$
    and pick any irreducible isoclinic formal connection
    $\nabla\in\Loc(X,\nu)$.
    Let $\chi\in\Op_{\Lg}(X,\nu)$
    be any oper above $\nabla$.
    By Proposition \ref{p:oper fiber over Loc}.(iii),
    the values of $\chi$ on $S_{i,j}$ for $(i,j)\in\tilde{A}_\nu$
    are uniquely determined by $\nabla$,
    independent of the choice of $\chi$.
    This determines a character $\phi_{cl}$ of $A$ 
    that lands inside $(\Spec A)^\rs$.
    By (iii) and (iv) of Corollary \ref{c:A intersection},
    we can regard $\phi_{cl}$ as a point of $\Hit(D)_{\bj}^\rs$.
    By Proposition \ref{p:local classical diagram}.(ii),
    $\phi_{cl}$ can be extended to a character $\tphi$ of $\bj$
    that corresponds to a point in $\bj^{*,\rs}\simeq\bj^\rs$,
     i.e. $\tphi\in h^{cl,-1}(\phi_{cl})$.
    Inflating along $J\rightarrow J/J^+\simeq\bj$ and composing with the exponential,
    we obtain a character of $J$ of depth $\nu$,
    which we still denote by $\tphi$.
    
    Denote $\Vac_{\fj,\tphi}=\Ind^{\hg}_{\fj+\bC\bone}\tphi$
    and $\fZ_{\fj,\tphi}=\mathrm{Im}(\fZ\rightarrow\End\Vac_{\fj,\tphi})$.
    Let $\Op_{\Lg}(\fj,\tphi):=\Spec\fZ_{\fj,\tphi}$.
    Denote by $\Op_{\Lg}(\nabla)=p^{-1}(\nabla)$
    the opers with underlying connection $\nabla$.
    
    \begin{lem}\label{l:central support of toral sc repn}\mbox{}
    	\begin{itemize}
    		\item [(i)]
    		Let $\psi$ be the character on 
    		$A_{-N}:=\bC[S_{i,j},\ell_{i,j}=-N]\subset A'\simeq A$
    		determined by the coefficients of the unique element
    		$p_{-1}+\sum_{\ell_{i,j}=-N}v_{i,j}p_i\in\LG\cdot X$.
    		Then $\Op_{\Lg}(X,\nu)$ is the central support of
    		$\Vac_{\fj^+}/\ker \psi$.
    		
    		\item [(ii)]
    		For any $\tphi_1,\tphi_2\in h^{cl,-1}(\phi_{cl})$,
    		$\Op_{\Lg}(\fj,\tphi_1)=\Op_{\Lg}(\fj,\tphi_2)$.
    		
    		\item [(iii)]
    		$\Op_{\Lg}(\nabla)
    		=\bigcup_{\tphi\in h^{cl,-1}(\phi_{cl})}\Op_{\Lg}(\fj,\tphi)
    		=\Op_{\Lg}(\fj,\tphi)$ 
    		for any $\tphi\in h^{cl,-1}(\phi_{cl})$.
    	\end{itemize}
    	
    \end{lem}
    \begin{proof}
    	(i):
    	This follows immediately from Lemma \ref{l:Op_j+=Op_=<nu}
    	and Proposition \ref{p:oper fiber over Loc}.(i).
    	
    	(ii):
    	We first exhibit the central support in a more explicit form.
    	For $\tphi=\tphi_1$ or $\tphi_2$,
    	\[
    	\End\Vac_{\fj,\tphi}
    	\simeq\Hom_{\fj}(\bC_{\tphi},\Vac_{\fj,\tphi})
    	\simeq\Vac_{\fj,\tphi}^{\ker\tphi}
    	\subset\Vac_{\fj,\tphi}
    	=\Ug/\Ug\ker\tphi.
    	\]
    	Thus 
    	\[
    	\ker(\fZ\rightarrow\fZ_{\fj,\tphi})
    	=\ker(\fZ\rightarrow\Ug/\Ug\ker\tphi)
    	=\fZ\cap(\Ug\ker\tphi).
    	\]
    	
    	By Proposition \ref{p:local classical diagram}.(iii)
    	and its proof,
    	any two elements $\tphi_1,\tphi_2\in h^{cl,-1}(\phi_{cl})$
    	satisfy $\tphi_2=\Ad_w\tphi_1$ 
    	for some element $w\in N_{G_0}(\ft_{Y,-N})$,
    	so that $\ker\tphi_2=\Ad_w\ker\tphi_1\subset U(\fj)$.
    	Since $\Ad_w$ acts trivially on $\fZ$, we obtain 
    	\[
    	\fZ\cap(\Ug\ker\tphi_1)
    	=\Ad_w(\fZ\cap(\Ug\ker\tphi_1))
    	=\fZ\cap(\Ug\ker\tphi_2).
    	\]
    	
    	From the above discussion, 
    	$\Vac_{\fj,\tphi_1}$ and $\Vac_{\fj,\tphi_2}$ have the same central support.
    	
    	(iii):
        The second equality follows from (ii).
        To prove the first equality,
    	we first show the inclusion
    	$\Op_{\Lg}(\fj,\tphi)\subset\Op_{\Lg}(\nabla)$
    	for any $\tphi\in h^{cl,-1}(\phi_{cl})$.
    	Let $\chi\in\Op_{\Lg}(\fj,\tphi)$.
    	Note that we have surjection
    	\[
    	q:\Vac_{\fj^+}=\Ind^{\hg}_{\fj+\bC\bone}U(\bj)
    	\rightarrow
    	\Vac_{\fj,\tphi}=\Ind^{\hg}_{\fj+\bC\bone}\tphi.
    	\]
    	Thus $\chi\in\Op_{\Lg}(D)_{\fj^+}=\Op_{\Lg}(D)_{\leq\nu}$.
    	Also, by Proposition \ref{p:local quant diagram}
    	$A$ acts on $\Vac_{\fj,\tphi}$ via $\tphi$.
    	By Proposition \ref{p:oper fiber over Loc}.(iii),
    	the underlying connection of $\chi$ is $\nabla$.
    	
    	Conversely, let $\phi_{cl}$ be the character of $A$
    	that is determined by $\nabla$ via Proposition \ref{p:oper fiber over Loc}.(iii).
    	By Lemma \ref{l:Op_j+=Op_=<nu} and Proposition \ref{p:local quant diagram},
    	$\Op_{\Lg}(\nabla)$ is the central support of
    	\[
    	\Vac_{\phi_{cl}}:=\Ind^{\hg}_{\fj+\bC\bone}(U(\bj)/\ker\phi_{cl}).
    	\]
    	
    	Regard $\phi_{cl}\in\Hit(D)_{\bj}^\rs$.
    	Then $U(\bj)/\ker\phi_{cl}\simeq\Fun h^{cl,-1}(\phi_{cl})$.
    	By Proposition \ref{p:local classical diagram}.(iii),
    	$|h^{cl,-1}(\phi_{cl})|=|W_0|$ is finite.
    	We obtain
    	\[
    	U(\bj)/\ker\phi_{cl}\simeq\bigoplus_{\tphi\in h^{cl,-1}(\phi_{cl})}U(\bj)/\ker\tphi
    	\]
    	and
    	\[
    	\Vac_{\phi_{cl}}\simeq\bigoplus_{\tphi\in h^{cl,-1}(\phi_{cl})}\Ind^{\hg}_{\bj+\bC\bone}\bC_{\tphi}.
    	\]
    	
    	Thus the central support of $\Vac_{\phi_{cl}}$
    	is the union of central supports of $\Vac_{\fj,\tphi}$'s.
    \end{proof}

    \subsection{Isoclinic local geometric Langlands}\label{ss:isoclinic Langlands}
    \subsubsection{}
    Fix $\nu=\frac{N}{m}>0$, $(N,m)=1$, where $m$ is a regular elliptic number.
    Fix $Y\in\fg_N^{*,\rs}$ and $X\in\Lg_{-N}^\rs$ that
    match under \eqref{eq:match rs leading terms}.
    Observe that in \S\ref{ss:toral sc repn}, 
    the subgroup $J\subset G(F)$ is determined by $\nu,Y$.
    Recall $\Loc(X,\nu)$ consists of the isomorphism classes of 
    irreducible isoclinic formal connections with slope $\nu$ 
    and leading term $X$ in its canonical form.
    We keep the notation from the previous sections.
    
    \begin{thm}\label{t:main}
    	The following map
    	\begin{equation}\label{eq:isoclinic Langlands}
    		\gamma:\{\tphi\in\bj^{*,\rs}\mid\tphi\text{ has component }Y\text{ in }\ft_{Y,N} \}
    		\rightarrow\Loc(X,\nu),\qquad
    		\tphi\mapsto p(\Op_{\Lg}(\fj,\tphi))
    	\end{equation}
        is well-defined and surjective, 
        with finite fibers of the same order as $W_0=N_{G_0}(\ft_{Y,-N})/C_{G_0}(\ft_{Y,-N})$.
    	        
        Explicitly, 
        $\nabla=\gamma(\tphi)$ 
        has a unique connection form as \eqref{eq:Op=Loc(X,nu)}
        where $v_{i,j}=\tphi(\overline{S_{i,j}})$
        for $\overline{S_{i,j}}$ the image of $S_{i,j}$
        in $A\subset U(\bj)$, $-N\leq\ell_{i,j}\leq-1$.
        Conversely, $\tphi$ can be any preimage of $\phi_{cl}\in\Spec A$
        for the $\phi_{cl}$ associated to $\nabla$.
    \end{thm}    
    \begin{proof}
    	By Lemma \ref{l:central support of toral sc repn} and its proof,
    	we see $\gamma$ is well-defined,
    	with $\nabla=\gamma(\tphi)$ as described.
    	
    	Conversely, given any $\nabla\in\Loc(X,\nu)$,
    	we have seen in the discussion of \S\ref{sss:matching}
    	that $\nabla$ determines a character $\phi_{cl}$ of algebra $A$,
    	and $\gamma^{-1}(\nabla)=h^{cl,-1}(\phi_{cl})$.
    	By Proposition \ref{p:local classical diagram}.(iii),
    	the fiber has size $|W_0|$.
    \end{proof}
    
    \begin{rem}\label{r:local conj for toral}
    	To sum up, 
    	for toral supercuspidal representations and irreducible isoclinic connections,
    	Conjecture \ref{c:local conj}.(i)-(iii) have been established
    	in Lemma \ref{l:central support of toral sc repn}
    	and Theorem \ref{t:main}.
    	The part (iv) of the conjecture follows from the same proof as 
    	Lemma \ref{l:central support of toral sc repn}.(ii).
    \end{rem}

    \subsubsection{}
    For an oper $\chi\in\Op_{\Lg}(D^\times)$,
    denote by $\hg-\mathrm{mod}_\chi$
    the category of smooth $\hg$-modules
    on which the center $\fZ$ acts via $\chi$.
    We say a $\hg$-module is $(J,\tphi)$-equivariant if
    there is a nonzero morphism from $\Vac_{\fj,\tphi}$ to it.
    Since $J$ is pro-unipotent,
    this is equivalent to that it contains a nonzero $\ker\tphi$-equivariant object.
    
    We have the following refinement of Corollary \ref{c:K-type and L-parameter}
    in the case of isoclinic formal connections.
    \begin{cor}
    	The category $\hg-\mathrm{mod}_\chi$ contains 
    	a nonzero $(J,\tphi)$-equivariant object
    	for some $\tphi\in\bj^{*,\rs}$ with component $Y$ in $\ft_{Y,N}$ 
    	only if $\chi\in\Op_{\Lg}(X,\nu)$,
    	i.e. the underlying connection $\nabla$ of $\chi$
    	is isoclinic of slope $\nu$
    	with leading term $X$.
    \end{cor}
    \begin{proof}
    	By the same argument as the proof of \cite[Lemma 10.3.2]{FrenkelBook},
    	there exists a module $M\in\hg-\mathrm{mod}_\chi$
    	with a nonzero map $\Vac_{\fj,\tphi}\rightarrow M$.
    	Therefore 
    	\[
    	\chi\in\Op_{\Lg}(\fj,\tphi)\subset\Op_{\Lg}(\nabla)\subset\Op_{\Lg}(X,\nu).
    	\]
    \end{proof}

    \subsection{Comparison with Kaletha's correspondence}
    \subsubsection{}    
    In \cite[Proposition 6.1.2]{KalethaRegular},
    Kaletha gives a correspondence between 
    toral supercuspidal parameters and toral $L$-packet data.
    In view of Definition 6.1.1 of \emph{loc. cit.},
    irreducible isoclinic formal connections are exactly the de Rham counterpart
    of toral supercuspidal parameters.
    We now compare Kaletha's correspondence with ours \eqref{eq:isoclinic Langlands}.
    
    Take a toral $K$-type $(J,\tphi)$ 
    of depth $\nu=N/m$ and a leading term $Y\in\fg_N^{*,\mathrm{rs}}$
    associated to a character $\phi:S\rightarrow\bC^\times$.
    In view of the proof of \cite[Proposition 6.1.2]{KalethaRegular},
    the irregular type of the corresponding $L$-parameter $\nabla$
    should be given by a morphism from the wild inertia group 
    into a dual torus $\check{S}\subset\LG$,
    determined by the local class field theory for torus
    applied to $(S,\phi)$.
    Since $\nabla$ is irreducible isoclinic, 
    by Lemma \ref{l:isoclinic canonical form}
    it is determined by its canonical form.
    Also, 
    the canonical forms of irreducible isoclinic formal connections have vanishing tame component,
    so they are uniquely determined by the irregular type.
    
    In view of the above discussion, 
    it suffices to compare the irregular type of $\gamma(\tphi)$
    with the irregular type of the parameter of $(S,\phi)$
    under the de Rham local class field theory. 
    
    \subsubsection{}
    Recall
     $S_{0+}/S(1+N)\simeq\fs(1)/\fs(1+N)\simeq\bigoplus_{i=1}^N\ft_{Y,i}u^i\simeq\overline{\fj}$.
    The character $\phi$ determines an element $\tphi\in\bj^{*,\rs}$ 
    with leading term $Y$.
    Fix a Cartan subalgebra $\check{\fs}\subset\Lg$ and
    an isomorphism $\check{\fs}\simeq\ft_Y^*$.
    The grading on $\ft_Y$ induces a grading $\check{\fs}=\bigoplus_i\check{\fs}_i$.
    Then the residue pairing determines a unique element
    $A_\phi\in \bigoplus_{-N\leq i\leq -1}\check{\fs}_i u^i$ satisfying
    \[
    \tphi(Z)=m\Res(A_\phi,Z)\frac{\td u}{u},\quad \forall Z\in \fs(1)\simeq\widehat{\bigoplus_{i\geq 1}}\ft_{Y,i}u^i.
    \]
    Here the multiplication by $m$
    is to match with $m\frac{\td u}{u}=\frac{\td t}{t}$.
    More explicitly, via $\ft_{Y,i}^*\simeq\check{\fs}_{-i}$, 
    $A_\phi=\sum_{i=-N}^{-1} A_{\phi,i}u^i$, where $A_{\phi,i}\in\check{\fs}_i$
    corresponds to $\tphi|_{\ft_{Y,-i}u^{-i}}\in\ft_{Y,-i}^*$.
    Observe that $A_{\phi,-N}$ matches with $Y$ under GIT quotient,
    thus is regular semisimple.
    
    The local class field theory for a split torus in the de Rham setting 
    goes back to Beilinson-Drinfeld,
    see for example \cite[Proposition 6.3.1.2]{HR} for a recollection.
    In particular, as said in \cite{HR}, 
    the duality is compatible with Contou-Carr\`ere duality
    whose infinitesimal form is exactly the residue pairing.
    For non-split torus, 
    we need to take norm map on the automorphic side
    and take Galois action on the torus on the spectral side.
    Altogether, we obtain that the irregular type of the Langlands parameter
    for $(S,\phi)$ is the irregular type 
    \begin{equation}\label{eq:A_phi}
    	[\td+A_\phi\frac{\td u}{u}].
    \end{equation} 
    
    \subsubsection{}
    We compare $\td+A_\phi\frac{\td u}{u}$ with \eqref{eq:Loc(X,nu)}.
    Observe that in the above discussion, 
    we can choose the dual Cartan subalgebra $\check{\fs}\subset\Lg$
    and the isomorphism $\check{\fs}\simeq\ft_Y^*$ arbitrarily.
    In particular, with the order $m$ stable grading $\theta$ on $\Lg$
    and the leading term $X\in\Lg_{-N}$ fixed as in \S\ref{ss:isoclinic Langlands},
    we can let $\check{\fs}=\ct_X$.
    By \cite[Lemma 2.2]{JYDeligneSimpson},
    the grading on $\check{\fs}$ induced from the isomorphism
    $\check{\fs}\simeq\ft_Y^*$
    is given by the conjugation of a unique Weyl group element 
    $w\in W(\Lg,\ct_X)$ which is $\bZ$-regular elliptic of order $m$.
    Meanwhile, the action of $\theta$ on $\ct_X=\check{\fs}$
    is also via a $\bZ$-regular elliptic order $m$ element 
    $w'\in W(\Lg,\ct_X)$.
    By \cite[Proposition 8]{RLYG}, $w$ and $w'$ are conjugated in $W(\Lg,\ct_X)$.
    Thus we can choose the isomorphism $\ct_X=\check{\fs}\simeq\ft_Y^*$
    so that $w=w'$.
    Moreover, since $X$ and $Y$ match under 
    Lemma \ref{l:isoclinic leading term match},
    by \cite[Theorem 2]{Vinberg} we know 
    $\frac{1}{m}A_{\phi,-N}\in\check{\fs}_{-N}=\ct_{X,-N}$ and $X\in\ct_{X,-N}$
    are conjugated by an element of $W(\Lg,\ct_X)$ 
    that has a representative in $\LG_0$.
    Therefore, we can further modify the choice of the isomorphism 
    $\ct_X=\check{\fs}\simeq\ft_Y^*$ so that 
    the grading by $w$ and $w'$ match, and also $X=\frac{1}{m}A_{\phi,-N}$. 
    Under the above choice, we conclude that
    \[
    \td+A_\phi\frac{\td u}{u}=\td+\frac{1}{m}A_\phi\frac{\td t}{t}\in\Loc(X,\nu).
    \]
    
    \subsubsection{}
    \begin{prop}\label{p:irr comparison}
    	The irregular types of $\gamma(\tphi)$ and $[\td+A_\phi\frac{\td u}{u}]$
    	are $\LG(\overline{F})$-gauge equivalent.
    \end{prop}
    \begin{proof}
    	By \cite[Lemma 3.1.2]{XYRigid},
    	any isoclinic formal connection over $\Spec\bC(\!(t)\!)$
    	can be transformed into its canonical form over the degree $2m$
    	covering.
    	Denote $F'=\bC(\!(t^{1/2m})\!)$.
        Denote the Hitchin maps for $G(F'),\LG(F')$
        by $h^{cl}_G,h^{cl}_{\LG}$ respectively,
        where the Hitchin bases are denoted by 
        $\Hit_G(F'),\Hit_{\LG}(F')$.
        We choose principal $\mathfrak{sl}_2$-triples $\{p_1,2\rho,p_{-1}\}$
        of $\fg,\Lg$
        and homogeneous basis $p_i$ for $\fg^{p_1},\Lg^{p_1}$
        so that two sets of $p_i$ coincide under the canonical isomorphism
        $p_{-1}+\fg^{p_1}\simeq\fg/\!\!/G\simeq\Lg/\!\!/\LG\simeq p_{-1}+\Lg^{p_1}$.
        We have isomorphisms
        \[
        \Hit_G(F')=\fg^*/\!\!/G\times^{\bGm}\omega_{F'}^\times
        \simeq\bigoplus_i\omega_{F'}^{d_i}
        \simeq\Lg/\!\!/\LG\times^{\bGm}\omega_{F'}^\times
        =\Hit_{\LG}(F').
        \]  
        In the following we will regard differential forms and Hitchin bases
        over $\bC(\!(t)\!),\bC(\!(u)\!)$ as subspaces of the corresponding spaces
        over $F'$.
        
        On the one hand,
        for $\tphi\in\bj^*\simeq\fj^{+,\perp}/\fj^\perp$ 
        as in \eqref{eq:local classical diagram}, 
        take its lifting $\tphi\in\fj^{+,\perp}$ 
        via the section
        $\bj^*\simeq\sum_{-N\leq i\leq -1}\ft_{Y,i}u^i\hookrightarrow\fj^{+,\perp}$
        in \eqref{eq:j^+,perp},
        which we still denote by $\tphi$.
        Let the coordinates of $h^{cl}_G(\tphi)$ in terms of $\td t$ be $h_{i,j}$.
        Consider the following oper connection form:
        \[
        \nabla_{\tphi}:=\td+(p_{-1}+\sum_i\sum_{\ell_{i,j}\geq-N} h_{i,j}t^{-j-1}p_i)\td t.
        \] 
        By the proof of Proposition \ref{p:local quant diagram},
        for $-N\leq\ell_{i,j}\leq-1$, only the top degree terms of the operator $S_{i,j}$ acts nontrivially on the vacuum vector $1\in\Vac_{\fj^+}$.
        Therefore for $-N\leq\ell_{i,j}\leq-1$, 
        $\tphi(S_{i,j})=h_{i,j}$.
        By Proposition \ref{p:oper fiber over Loc}.(iii)
        and Theorem \ref{t:main},
        $\nabla_{\tphi}$ is equivalent to $\gamma(\tphi)$.
        
        On the other hand,
        the differential form $A_\phi\frac{\td u}{u}$ 
        can be regarded as the image of $\tphi$ under the residue pairing
        $\ft_Y(F')^*\simeq\ft_Y\otimes\omega_{F'}
        \simeq\ft_Y^*\otimes\omega_{F'}
        \simeq\check{\fs}\otimes\omega_{F'}$,
        so that $h^{cl}_G(\tphi)=h^{cl}_{\LG}(A_\phi\frac{\td u}{u})$.
        Therefore we have
        \[
        h^{cl}_{\LG}((p_{-1}+\sum_i\sum_{\ell_{i,j}\geq-N} h_{i,j}t^{-j-1}p_i)\td t)
        =h^{cl}_G(\tphi)
        =h^{cl}_{\LG}(A_\phi\frac{\td u}{u}).
        \]
        Changing the coordinate from $t$ to $u$, we deduce that
        \[
        h^{cl}_{\LG}(mu^m(p_{-1}+\sum_i\sum_{\ell_{i,j}\geq-N} h_{i,j}u^{-m(j+1)}p_i))
        =h^{cl}_{\LG}(A_\phi).
        \]
        
        Our goal is to compare the canonical forms of $\nabla_{\tphi}$ and
        $\td+A_\phi\frac{\td u}{u}$.
        Denote $A:=mu^m(p_{-1}+\sum_i\sum_{\ell_{i,j}\geq-N} h_{i,j}u^{-m(j+1)}p_i)$,
        so that $\nabla_{\tphi}\simeq\td+A\frac{\td u}{u}$.
        By the proof of Lemma \ref{l:isoclinic can remove regular part}
        and Proposition \ref{p:oper fiber over Loc},
        there exists an element $g=g_0\crho(u^{N+m})$, $g_0\in\LG(\cO_{F'})$,
        such that $\Ad_g A\in\check{\fs}(F')$ with regular semisimple leading term,
        and
        \[
        g\cdot\nabla_{\tphi}\in\td+(\Ad_g A+\Lg(\cO_{F'}))\frac{\td u}{u}.
        \]
        Observe that $A_\phi,\Ad_g A\in\check{\fs}(F')$
        and $h^{cl}_{\LG}(A_\phi)=h^{cl}_{\LG}(\Ad_g A)$.
        Since $A_\phi$ and $\Ad_g A$ have regular semisimple leading terms, 
        they are both regular semisimple elements of
        $\Lg(\overline{F})$ with the same Hitchin image.
        Therefore they are conjugated over $\LG(\overline{F})$,
        and further conjugated by the Weyl group with respect to $\check{\fs}$.
        Take any scalar lifting $w\in N_{\LG}(\check{\fs})$ of the Weyl group element,
        we get
        \[
        w\cdot(\td+A_\phi\frac{\td u}{u})
        =\td+\Ad_w A_\phi\frac{\td u}{u}
        =\td+\Ad_g A\frac{\td u}{u}
        \in g\cdot\nabla_{\tphi}+\Lg(\cO_{F'})\frac{\td u}{u}.
        \]
        By Lemma \ref{l:isoclinic can remove regular part},
        the canonical forms of $w\cdot(\td+A_\phi\frac{\td u}{u})$ and $g\cdot\nabla_{\tphi}$
        can differ only by a regular singular term $\Lg\frac{\td u}{u}$,
        i.e. they have the same irregular type.
        The proof is complete.
    \end{proof}
    
    \section{Application: Airy connections}\label{s:Airy conn}
    As an example and application of our local results,
    we prove that
    the Langlands parameters of Hecke eigensheaves 
    constructed in \cite{JKY}
    are Airy connections for reductive groups
    \cite{KSRigid,HJRigid}.
    
    \subsection{Airy connections and Airy automorphic data}
    \subsubsection{Airy connections}
    Following \cite[Definition 5.5.1]{HJRigid},
    an \emph{Airy $\LG$-connection}
    is a $\LG$-connection on $\bP^1-\{0\}$\footnote{We let the singularity be at $0$ instead of $\infty$ for convenience of notations.}
    that is isoclinic of slope $\nu=\frac{1+h}{h}$ at $0$.
    Here $h$ is the Coxeter number of $\LG$.
    
    For an Airy connection $\nabla$,
    its restriction $\nabla_0:=\nabla|_{D_0^\times}$
    is an isoclinic formal connection.
    Moreover, since $h$ is a regular elliptic number,
    by the same proof as that of \cite[Lemma 7]{YiFG}
    we know $\nabla_0$ is irreducible.
    By \cite[Theorem 1.8, Theorem 4.1]{JYDeligneSimpson},
    any isoclinic formal connection $\nabla_0$ of slope $\nu$
    is the formal type at $0$ of an Airy connection $\nabla$.
    By \cite[Theorem 5.5.2]{HJRigid},
    Airy connections are \emph{physically rigid},
    i.e. for any Airy connection $\nabla$ and another $\LG$-connection $\nabla'$
    on $\bP^1-\{0\}$,
    $\nabla|_{D_0^\times}\simeq\nabla'|_{D_0^\times}$ implies
    $\nabla\simeq\nabla'$.
    Therefore Airy connections are in bijection with
    isoclinic formal connections of slope $\nu$.
    Denote by $\mathrm{Ai}_{\LG}(X)$ the isomorphism classes 
    of Airy connections with leading term $X$ in its canonical form at $0$.
    
    Since $\nu>1$,
    combining the above with \eqref{eq:Op=Loc(X,nu)}
    and Proposition \ref{p:global ext nu>=1},
    we obtain
    \begin{prop}\label{p:Airy eq}
    	We have isomorphisms
    	\[
    	\mathrm{Ai}_{\LG}(X)
    	\simeq\Loc(X,\frac{1+h}{h},\bP^1-\{0\})
    	\simeq\Loc(X,\frac{1+h}{h}),
    	\]
    	where each such Airy connection has a unique connection form of the form
    	\[
    	\td+(t^{-2}p_{-1}+v_{n,2h}t^{-3}p_n+\sum_{i=1}^n v_{i,2d_i-1}t^{-2}p_i)\td t,
    	\quad	p_{-1}+v_{n,2h}p_n\in\LG\cdot X.
    	\]
    \end{prop}
    
    Comparing the above with \cite[equation (8)]{KSMon},
    we conclude that every Airy $\LG$-connection is
    \emph{framable}.
    Observe that \eqref{eq:KS Airy} is a special case of the above equation.

    \subsubsection{Airy automorphic data}
    The Airy automorphic data \cite[Definition 8]{JKY}
    are special cases of toral $K$-types $(J,\tphi)$
    of depth $\nu=\frac{1+h}{h}$
    defined in \S\ref{ss:toral sc repn}\footnote{The use of notations $\tphi,\phi$ are slightly different from \cite{JKY}.}.
    Here $m=h$, $N=1+h$.
    We recall the particular choice of $\tphi$ in the \emph{loc. cit.},
    which we will see is necessary.
    Recall $\tphi\in\bj^*\simeq\bigoplus_{-1-h\leq i\leq -1}\ft_{Y,i}u^i\frac{\td u}{u}$,
    where $Y\in\fg_{-1-h}^\rs$ is the component of $\tphi$ in $\fg_{-1-h}$.
    We require $\tphi$ to have zero components in
    $\ft_{Y,i}$ for $-h\leq i\leq-(1+\frac{h}{2})$ when $h$ is even
    (resp. for $-h\leq i\leq-(n+1)$ when $h=2n+1$ is odd, $\fg=\mathfrak{sl}_{2n+1}$).
    We call such $\tphi$ \emph{special}.
    
    Denote by $\cG'$ the group scheme on $\bP^1$
    satisfying $\cG'|_{\bP^1-\{0\}}\simeq G\times(\bP^1-\{0\})$
    and $\cG'(\cO_0)=J^+$.
    Denote by $\Bun_{\cG'}$
    the moduli stack of $\cG'$-bundles.
    The group $J$ acts on $\Bun_{\cG'}$ at $0$.
    Let $\cL_{\tphi}=\tphi^*(\td-\td t)$ be a character $D$-module on $J$.
    
    \begin{prop}\cite[Lemma 21]{JKY}\label{p:rigid datum}
    	For $\tphi$ special, 
    	there exists unique $(J,\cL_{\tphi})$-equivariant irreducible holonomic
    	$D$-module $\cA_{\tphi}$ on $\Bun_{\cG'}$ up to isomorphism.
    \end{prop}

    We explain below 
    why the assumption of $\tphi$ being special is necessary
    in Proposition \ref{p:rigid datum}.
    	
    Recall that relevant orbits are $J$-orbits on $\Bun_{\cG'}$
    over which $(J,\cL_{\tphi})$-equivariant sheaves can be supported.
    In other words,
    these are orbits on which 
    $\cL_{\tphi}$ is trivial on the stabilizer of $J$.
    The proof in \cite[\S4.2.1,\S4.2.2]{JKY} actually goes through
    without assuming $\tphi$ being special,
    which shows that there is at most one relevant orbit
    (assuming $G$ is simply-connected):
    the orbit of trivial $\cG'$-bundle.
    However, we have the following observation:
    \begin{prop}
    	The unit orbit is relevant if and only if $\tphi$ is special.
    \end{prop}
    \begin{proof}
    	We assume that $h$ is even.
    	The proof for $h$ odd is just replacing all the $\frac{h}{2}$
    	with $n+1$, $h=2n+1$.
    	
    	Denote by $\fg=\bigoplus_{i\in\bZ}\fg(i)$
    	the $\bZ$-grading defined by $\crho_G$.
    	Then $\fg_i=\fg(i)\oplus\fg(i+h)$ for $-h-1\leq i\leq -1$.
    	Regard $\tphi\in\bj^*\simeq\bigoplus_{-1-h\leq i\leq -1}\ft_{Y,i}u^i\frac{\td u}{u}$.
    	By \cite[\S4.4]{JKY}, for the unit orbit to be relevant, we need
    	\begin{equation}\label{eq:2}
    	\Res(\kappa(\tphi,\bigoplus_{i=1+\frac{h}{2}}^{h-1}\fg(i)))\frac{\td t}{t}
    	\simeq\Res(\kappa(\tphi,\bigoplus_{i=1+\frac{h}{2}}^{h-1}\fg(i)u^i))\frac{\td u}{u}
    	=0.
    	\end{equation}
    	
    	Let $\tphi_i=Y_iu^i\frac{\td u}{u}$ 
    	be the component of $\tphi$ in $\ft_{Y,i}u^i\frac{\td u}{u}$.
    	Then \eqref{eq:2} holds if and only if
    	$\kappa(Y_{-i},\fg(i))=0$ for $1+\frac{h}{2}\leq -i\leq h-1$
    	(note $Y_{-h}\in\ft_{Y,-h}=\ft_{Y,0}=0$).
    	Denote by $p_-$ the projection from 
    	$\ft_{Y,-i}\subset\fg_{-i}$ to $\fg(-i)$.
    	Then 
    	\[
    	\kappa(Y_{-i},\fg(i))=\kappa(p_-(Y_{-i}),\fg(i))=\kappa(p_-(Y_{-i}),\fg_i).
    	\]
    	By \cite[Lemma 15]{JKY}, $p_-$ is injective.
    	Since $\kappa$ is non-degenerate on $\fg_{-i}\times\fg_i$,
    	$\kappa(p_-(Y_{-i}),\fg_i)=0$
    	if and only if $p_-(Y_{-i})=0$
    	if and only if $Y_{-i}=0$.
    \end{proof}
    
     We will also need the following observation.
    Denote
    \begin{equation}
    	\Op_{\cG}=\Op_{\Lg}(\bP^1-\{0\})\times_{\Op_{\Lg}(D_0^\times)}\Op_{\fj^+}.
    \end{equation}
    
    We have natural map $p:\Op_{\cG}\rightarrow\Op_{\fj^+}\rightarrow\Spec A$.
    
    \begin{lem}\label{l:global Airy opers}
    	We have isomorphisms
    	\begin{equation}\label{eq:Airy oper}
    		\{\td+(t^{-2}p_{-1}+v_{n,2h}t^{-3}p_n+\sum_{i=1}^n v_{i,2d_i-1}t^{-2}p_i)\td t\mid v_{i,j}\in\bC\}
    		\simeq\Op_{\cG}\xrightarrow[\sim]{p}\Spec A
    	\end{equation}
    	sending $v_{i,j}$ to $S_{i,j}\in A$.
    \end{lem}
    \begin{proof}
    	We begin with
    	\[
    	\Op_{\Lg}(\bP^1-\{0\})\simeq\td+(p_{-1}+\sum_{i=1}^n\bC[t^{-1}]p_i)\td(t^{-1}).
    	\]
    	Apply gauge transform by
    	$\crho(t^{-2})$, $\exp(-p_1t^{-1})$, and $\crho(-1)$,
    	we obtain the following space of local opers on $D_0^\times$:
    	\[
    	\td+(p_{-1}+t^{-2}p_1+\sum_{i=1}^n\bC[t^{-1}]t^{-2d_i}p_i)\td t.
    	\]
    	By Lemma \ref{l:Op_j+=Op_=<nu}, 
    	the above oper lands inside $\Op_{\fj^+}$ if and only if the coefficients
    	of $t^{-j-1}p_i$ are zero for $j\geq 2d_i-1+\frac{d_i}{h}$, i.e.
    	\begin{equation}\label{eq:local Airy}
    		\td+(p_{-1}+t^{-2}p_1+\bC t^{-2h-1}p_n+\sum_{i=1}^n\bC t^{-2d_i}p_i)\td t.	
    	\end{equation}
    	The map $p$ sends above equations to the coefficients of 
    	$t^{-2h-1}p_n$ and $t^{-2d_i}p_i$,
    	thus is isomorphic.
    	
    	Next, the equations on the left-hand side of \eqref{eq:Airy oper}
    	and isomorphic to \eqref{eq:local Airy} over $\bP^1-\{0\}$
    	by applying gauge transforms of $\crho(t^{-2})$ and $\exp(p_1t^{-1})$.
    \end{proof}

    \subsection{The Langlands correspondence}
    \begin{thm}\label{t:Airy Langlands}
    	Let $(J,\tphi)$ be such that $\tphi$ is special.
    	Let $\nabla_0$ be the isoclinic formal connection 
    	corresponding to $\tphi$ as in Theorem \ref{t:main}.
    	Let $\nabla$ be the Airy connection with formal type $\nabla_0$ at $0$. 
    	Then $\nabla$ is the eigenvalue of $\cA_{\tphi}$.
    \end{thm}
    \begin{proof}
    	By Proposition \ref{p:Airy eq} and Lemma \ref{l:global Airy opers},
    	there exists a unique global oper $\chi\in\Op_{\cG}$ 
    	whose underlying connection is the Airy connection $\nabla$.
    	
    	Let $\omega_{\Bun_{\cG}}^{1/2}$ be a square root line bundle
    	of the canonical sheaf of $\Bun_{\cG}$
    	which exists by \cite[Proposition 7]{Zhu}.
    	Consider
    	\begin{equation}\label{eq:A_chi}
    		\cA_{\tphi}:=
    		\omega_{\Bun_{\cG}}^{-1/2}\otimes
    		(\cD_{\Bun_\cG}'\otimes^{\mathrm{L}}_{\Fun\Op_{\cG}\otimes_A U(\bar{\mathfrak{j}})}\bC_{\chi,\tphi}).
    	\end{equation}
        By construction,
        $\cA_{\tphi}$ is $(J,\tphi)$-equivariant.
        By the proof of \cite[Theorem 8]{Zhu},
        $\cA_{\tphi}$ is Hecke eigen with eigenvalue $\nabla$.
        It only remains to show that $\cA_{\tphi}$
        is nonzero and holonomic,
        so that it is a nonzero complex of copies of the 
        unique irreducible $(J,\tphi)$-equivariant holonomic $D$-module
        in Proposition \ref{p:rigid datum}.
         
        Let $L^-G$ be the opposite loop group.
        Then $\Bun_{\cG'}\simeq[L^-G\backslash LG/J^+]$.
        By \cite[Theorem 13, Proposition 19]{JKY},
        the unique relevant orbit 
        $j:O=[L^-G\backslash L^-G J/J^+]\hookrightarrow\Bun_{\cG'}$
        is locally closed.
        As in \S6.5 of \cite{CYTheta},
        $\cA_{\tphi}\simeq j_*j^*\cA_{\tphi}$ is holonomic.
        
        It remains to show $\cA_{\tphi}\neq 0$.
        By Lemma \ref{l:global Airy opers},
        $\cA_{\tphi}=
        \omega_{\Bun_{\cG}}^{-1/2}\otimes
        (\cD_{\Bun_\cG}'\otimes^{\mathrm{L}}_{U(\bar{\mathfrak{j}})}\bC_{\tphi})$.
        It suffices to show $\cD_{\Bun_\cG}'\otimes_{U(\bar{\mathfrak{j}})}\bC_{\tphi}\neq 0$.
        By \cite[Proposition 19]{JKY},
        we have open embedding
        $j':O'=[L^-G\backslash L^-G I(1)/J^+]\hookrightarrow\Bun_{\cG'}$. 
        It suffices to show 
        $\cD_{O'}'\otimes_{U(\bar{\mathfrak{j}})}\bC_{\tphi}\neq 0$.
        Recall $I(2+h)$ is a normal subgroup of $J^+$. 
        Denote the inflated character of $\tphi$ on $J/I(2+h)$ still by $\tphi$.
        Denote $U=L^-G\cap I(1)$, so that $O'\simeq[U\backslash I(1)/J^+]$.
        Then $\cD_{O'}'\otimes_{U(\bar{\mathfrak{j}})}\bC_{\tphi}$
        is the descent of 
        $\cD_{[U\backslash I(1)/I(2+h)]}'\otimes_{U(\bar{\mathfrak{j}})}\bC_{\tphi}$
        under $J^+/I(2+h)$-action,
        and it suffices to show the latter is nonzero.
        Note that $U$ has trivial intersection with $I(2+h)$,
        so that $U$ embeds as a subgroup of $I(1)/I(2+h)$.
        Therefore $[U\backslash I(1)/I(2+h)]$ is a smooth scheme. 
        Now observe that $J,I(1)\subset G(\cO)$,
        so that the differential actions by $\Lie(J)$ on $[U\backslash I(1)/I(2+h)]$
        is trivial on the twisted line bundle.
        Therefore the stalk of
        $\cD_{[U\backslash I(1)/I(2+h)]}'\otimes_{U(\bar{\mathfrak{j}})}\bC_{\tphi}$
        at unit is
        \[
        \mathrm{Sym}(\Lie(I(1)/I(2+h)))/\langle\Lie(U),\{x-\tphi(x)\mid x\in\Lie(J/I(2+h))\}\rangle.
        \]
        The above is nonzero if $\tphi(U\cap J)=1$,
        which is exactly the fact that the unit bundle sits inside the relevant orbit.
        This completes the proof.
    \end{proof}

    \begin{rem}\label{r:Airy Langlands}
    	\begin{itemize}
    		\item [(i)]
    		For $\LG=\GL_n$, 
    		Theorem \ref{t:Airy Langlands} has been proved in 
    		\cite[Corollary 36]{JKY}.
    		\item [(ii)]
    		Theorem \ref{t:Airy Langlands} implies that
    		the converse of Corollary \ref{c:K-type and L-parameter}
    		is true for $\nu=1+\frac{1}{h}$
    		and special $\tphi$.
    	\end{itemize}
    	
    \end{rem}
    
    We confirm Conjecture 29 of \cite{JKY}.
    
    \begin{cor}
    	Let $(J,\tphi)$ be an Airy automorphic datum
    	where $\tphi$ is special.
    	Denote the eigenvalue $\LG$-connection
    	of the associated Hecke eigensheaf by
    	$\nabla_{\tphi}$.
    	\begin{itemize}
    		\item [(i)]
    		The irregularity of $\nabla_{\tphi}^\Ad$ at $0$ is $n(h+1)$.
    		
    		\item [(ii)]
    		The local differential Galois group $\cI_0\subset\LG$
    		of $\nabla_{\tphi}$ at $0$ satisfies
    		$\Lg^{\cI_0}=0$.
    		
    		\item [(iii)]
    		$\nabla_{\tphi}$ is cohomologically rigid.
    	\end{itemize}
    \end{cor}
    \begin{proof}
    	We know from Theorem \ref{t:Airy Langlands} that
    	$\nabla_{\tphi}$ is an Airy connections.
    	For Airy connections, 
    	(i)-(iii) are proved in 
    	\cite[Lemma 6.2, Proposition 6.4]{JYDeligneSimpson}.    	
    \end{proof}

\end{document}